\documentclass[english]{amsart}

\usepackage{stmaryrd}
\usepackage[pdftex]{graphicx}
\usepackage{mathrsfs}
\usepackage{amssymb}
\usepackage{color}

\newtheorem{theorem}   {Theorem}[section]
\newtheorem{proposition}  {Proposition}[section]
\newtheorem{corollary}   {Corollary}[section]
\newtheorem{lemma} {Lemma}[section]
\newtheorem{remark}   {Remark}[section]
\newtheorem{definition}  {Definition}[section]

\title{Solutions without any symmetry for semilinear elliptic problems}

\author{Weiwei Ao}
\address{Weiwei Ao. Department of Mathematics, The Chinese University of Hong Kong, Shatin, Hong Kong.}
\email{wwao@math.cuhk.edu.hk}

\author{Monica Musso}
\address{Monica Musso. Departamento de Matem\'atica, Pontificia Universidad Catolica de
Chile, Avda. Vicu\~na Mackenna 4860, Macul, Chile.}
\email{mmusso@mat.puc.cl}

\author[Frank Pacard] {Frank Pacard}
\address{Frank Pacard. Centre de
Math\'ematiques Laurent Schwartz, UMR-CNRS 7640, \'Ecole polytechnique,
91128 Palaiseau, France and Institut Universitaire de France. }
\email{Email: frank.pacard@math.polytechnique.fr}

\author{Juncheng Wei}
\address{Juncheng Wei. Department of Mathematics, The
Chinese University of Hong Kong, Shatin, Hong Kong.}
\email{wei@math.cuhk.edu.hk}

\thanks{This work has been partly supported by the
 contract C09E06 from the ECOS-CONICYT. The research of the second  author
has been partly supported by Fondecyt Grant 1120151 and
CAPDE-Anillo ACT-125, Chile. The third
author is partially supported by the ANR-08-BLANC-0335-01 grant. The
 research of the first and fourth author is supported by an Earmarked Grant from
 RGC of Hong Kong and Oversea Joint Grant of NSFC}

\begin{document}

\maketitle

\begin{abstract}
We prove the existence of infinitely many solitary waves for the nonlinear Klein-Gordon or Schr\"odinger equation
\[
\Delta u-u+ u^3 =0 ,
\]
in ${\bf R}^2$, which have finite energy and whose maximal group of symmetry reduces to the identity. 
\end{abstract}

\section{Introduction}

In this paper, we develop tools to construct infinitely many entire solutions of 
\begin{equation}
\Delta u - u + u^3 =0,
\label{eq:nls}
\end{equation}
which are defined in ${\bf R}^2$ and in particular, solutions whose maximal group of symmetry (i.e. the largest subgroup of isometries of ${\bf R}^2$ leaving the solution $u$ fixed) is discete. The solutions we are interested in can be either positive, negative or may change sign but they have finite energy in the sense that their energy
\[
\mathcal E (u) : =  \frac{1}{2} \int_{{\bf R}^2} ( |\nabla u|^2 + u^2) \, dx - \frac{1}{4} \int_{{\bf R}^2} u^4 \, dx ,
\]
is finite.

Equations like (\ref{eq:nls}), in dimension $2$ or in higher dimensions, have been thoroughly studied over the last decades since they are ubiquitous in various models in physics, mathematical physics or biology. For example, the study of standing waves (or solitary waves) for the nonlinear Klein-Gordon or Schr\"odinger equations reduces to (\ref{eq:nls}). We refer to \cite{BL,BL2} for further references and motivations on the subject.

Let us remind the reader of a few classical and well known results concerning the solvability of (\ref{eq:nls}) and the properties of the solutions of this equation. We will restrict our attention to the results which are relevant to the $2$-dimensional case and we have deliberately chosen not to mention results which hold in higher dimensions since the list of results and contributors is by now fairly long.

It is known \cite{BL} that there exists a unique positive, radial solution of (\ref{eq:nls}). This solution, which will be denoted by $u_0$, has the property that it decays exponentially to $0$ at infinity. More precisely, it follows from \cite{BL} that $u_0$ behaves at infinity like one of the modified Bessel's functions of the second kind and hence, that there exists a constant $C >0$ such that
\begin{equation}
u_0 (r) =  C\, e^{-r} \, r^{-1/2} \left( 1 + \mathcal O \left(\frac{1}{r}\right) \right),
\label{eq:exuo}
\end{equation}
as $r$ tends to infinity. The classical result of Gidas, Ni and Nirenberg \cite{GNN} asserts that any finite energy, positive solution of (\ref{eq:nls}) is (up to a translation) radially symmetric and hence finite energy, positive solutions of (\ref{eq:nls}) are all congruent to $u_0$.

As far as sign changing solutions are concerned, Berestycki and Lions \cite{BL2} have proved that  (\ref{eq:nls}) has infinitely many radial solutions which change sign. Again, these solutions do have finite energy. To complete this description, let us mention that it is proven in \cite{MPW} that there exists solutions of (\ref{eq:nls}) which have less symmetry than the ones constructed by Berestycki and Lions. In fact, given an integer $k\geq 7$, it is proven in \cite{MPW} that there exist infinitely many solutions of (\ref{eq:nls}) whose group of symmetry is the dihedral group of symmetry leaving a regular $k$-polygon fixed. Again, these solutions also change sign and have finite energy. In view of these results, a natural question is the following~: 

\noindent 

\begin{center}
{\em Do all solutions of (\ref{eq:nls}) have a nontrivial group of symmetry ?} 
\end{center}

Surprisingly, the answer to this question is negative. In fact, we prove the~:
\begin{theorem}
There exist infinitely many solutions of (\ref{eq:nls}) which have finite energy but whose maximal group of symmetry reduces to the identity.
\label{th:main}
\end{theorem}

The proof of this result relies on an extension of the construction in \cite{MPW}. As we will see, we will be able to find solutions of (\ref{eq:nls}) whose maximal group of symmetry reduces or not to the identity and hence, our construction provides a wealth of non congruent solutions of (\ref{eq:nls}) which change sign and have finite energy.  

Let us observe that solutions of (\ref{eq:nls}) which have {\em infinite energy} do exist in abundance and it is even known that positive solutions without any symmetry do exist in this context, i.e. if the finite energy assumption is relaxed. Concerning  infinite energy solutions there are two different classes of interest depending on the behavior of 
\[
\mathcal E_R (u) : = \frac{1}{2} \int_{D(0,R)} ( |\nabla u|^2 + u^2) \, dx - \frac{1}{4} \int_{D(0,R)} u^4 \, dx ,
\]
as $R$ tends to infinity, where the integrals are understood over the disc of radius $R$, centered at the origin. For example, non constant, doubly periodic  solutions are easy to construct using variational methods. These solutions have the property that $\mathcal E_R (u) \sim R^2$ as $R$ tends to infinity.  Non constant singly periodic solutions are also known to exist and they correspond to solutions for which $\mathcal E_R (u) \sim R$ as $R$ tends to infinity. Solutions sharing this later property have been constructed by Malchiodi in  \cite{M} and geometrically different solutions were also obtained in \cite{DKPW1} and in \cite{SW}. One of the main differences between solutions of (\ref{eq:nls}) with infinite energy and  solutions of (\ref{eq:nls}) with finite energy is that (once the action of the group of isometries of ${\bf R}^2$ has been taken into account) the moduli space of solutions with finite energy is expected to be discrete while the moduli space of infinite energy solutions is expected to have positive (finite) dimension.

\section{Description of the construction and comments}
\label{se:2}

The proof of Theorem~\ref{th:main} is  quite involved and, to help the reader, we now spend some time to briefly describe the main ideas behind the construction, without paying much attention on technical details such as estimates and functions spaces which will be used. Since we are working in ${\bf R}^2$, it will be  convenient to identify ${\bf R}^2$ with the complex plane $\bf C$. The scalar product in $\bf C$ will be denoted by $\langle \, \, , \, \, \rangle_{\bf C}$ so that 
\[
\langle z, z'\rangle_{\bf C} : =  \Re \, (\bar z \, z').
\]

In a nutshell, the idea of the construction is to start with two {\em finite} sets of points 
\[
Z^+ : =  \{ z_j^+ \in {\bf C} \, : \, j =1, \ldots, n^+\} \qquad \mbox{and} \qquad Z^- : =  \{ z_j^-  \in {\bf C} \, : \, j=1, \ldots, n^-\} ,
\]
and define an approximate solution to (\ref{eq:nls}) by simply adding copies of $+u_0$ centered at the points $z_j^+$ and copies of $-u_0$ centered at the points $z_j^-$. More precisely, with these notations, we define an approximate solution $\tilde u$ by the formula
\[
\tilde u : =  \sum_{z \in Z^+} u_0 (\cdot - z) - \sum_{z' \in Z^-} u_0 (\cdot - z').
\]

We set
\[
Z : = Z^+ \cup  Z^- ,
\]
and we agree that
\[
\ell : =  \min_{z \neq z' \in Z} \, |z-z'| ,
\]
denotes the minimum of the distances between the points of $Z$ (we assume that the points of $Z$ are all distinct so that $\ell >0$). Since the solution $u_0$ is exponentially decreasing to $0$ at infinity, the fact that $\tilde u$ is a fairly good approximate solution of (\ref{eq:nls}) as $\ell$ tends to infinity should not come as a surprise. Indeed, if 
\[
\tilde E  : =  \Delta \tilde u - \tilde u + \tilde u^3,
\]
it is not hard to check that 
\[
\| \tilde E \|_{L^\infty (\bf C)}Ê\leq  C \, e^{-\ell} \, \ell^{-1/2}.  
\]
for some constant $C >0$ which does not depend on $\ell \gg 1$.

The natural idea is then to let $\ell$ tend to infinity and to look for a solution $u$ of (\ref{eq:nls}) as a (small) perturbation of $\tilde u$. Writing $u= \tilde u +v$, this amounts to solve a nonlinear problem of the form 
\begin{equation}
\tilde L v +  \tilde E + \tilde Q (v) =0, 
\label{eq:szq}
\end{equation}
where 
\[
\tilde L  : = \Delta -1 + 3 \, \tilde u^2 ,
\]
is the linearized operator about $\tilde u$ and where 
\[
\tilde Q(v) := v^3 + 3 \, \tilde u\, v^2,
\]
collects all the nonlinear terms. In order to solve (\ref{eq:szq}), we try to invert $\tilde L$ so that we can rephrase the problem as a fixed point problem which we solve using a fixed point theorem for contraction mapping.  It turns out that this part of the argument is rather delicate due to the presence of small eigenvalues associated to the operator $\tilde L$. Indeed, the bounded kernel of the operator 
\[
L_0  : = \Delta -1 + 3 \, u_0^2,
\]
clearly contains the functions $\partial_x u_0$ and $\partial_y u_0$ and, transplanting these functions at any of the points of $Z$, one can prove that there exist $2 \, (n^+ + n^-)$ eigenfunctions of $\tilde L$ which are associated to small eigenvalues which in addition tend to $0$ as $\ell$ tends to infinity (in fact, in absolute value, these small eigenvalues can be seen to tend to $0$ exponentially fast as $\ell$ tends to infinity).  As usual when this phenomenon happens, one is lead to work orthogonally to the space of eigenfunctions associated to small eigenvalues of $\tilde L$ since, on such a space, the operator $\tilde L$ is invertible and has inverse whose norm can be controlled uniformly as $\ell$ tends to infinity.  This amounts to replace the equation $\tilde L v =f$ by 
\[
\tilde L \, v + \sum_{z \in Z} \langle c_{z} , \nabla u_0 (\cdot - z)\rangle_{\bf C} = f,
\]
where the solution is now the function $v$ and the complex numbers $c_z \in {\bf C}$. 
Once this is understood, one can make use of a fixed point theorem for contraction mappings to perturb $\tilde u$ into $u := \tilde u + v$ (where $v$ is a small function) solution of 
\begin{equation}
\Delta u -u + u^3 = \sum_{z \in Z} \langle F_{z} , \nabla u_0 (\cdot - z)\rangle_{\bf C} ,
\label{eq:nnls}
\end{equation}
where, for each $z \in Z$,  the complex number $F_z \in {\bf C}$ depends on all the coordinates of the points of $Z$. 

At this stage, the solvability of (\ref{eq:nls}) reduces to the search of a set of points $Z$ (which become {\em parameters} of the construction) in such a way that 
\begin{equation}
F_z =0,  \quad \text{for all} \quad z\in Z.
\label{eq:sysfz}
\end{equation} 
Observe that, {\it a priori} the number of equations and the number of unknowns are both equal to $2\, (n^+ + n^-)$ which gives some hope for the solvability of the system (\ref{eq:sysfz}), even if we will see later on that the story is not that simple. This procedure is what is usually called a Liapunov-Schmidt type argument~: the solvability of a nonlinear partial differential equation is reduced to the solvability of a system of equations in finite dimension. 

As one can suspect, it is not possible to derive the exact expression of the complex numbers $F_z$ in terms of the coordinates of the points of $Z$, but it is nevertheless possible to get a nice expansion of $F_z$ as $\ell$, the minimum of the distances between the points of $Z$, tends to infinity and we find, in essence, that
\begin{equation}
F_z \sim  \sum_{z' \in Z-\{z\}} \eta_z \, \eta_{z'} \, \Upsilon (|z'-z|) \, \frac{z'-z}{|z'-z|} ,
\label{eq:expfz}
\end{equation}
where the {\em interaction function} $\Upsilon$, which will defined later on, is explicitly known and is known to satisfy
\[
\Upsilon (s) \sim e^{-s} \, s^{-1/2},
\]
as $s$ tends to infinity and where $\eta_z =+1$ if, in the definition of $\tilde u$, there is a positive copy of $u_0$ centered at the point $z$ and $\eta_z=-1$ if, in the definition of $\tilde u$, there is negative copy of $u_0$ centered at the point $z$. 

At this stage, even if we assume that $\ell$ is large, finding the sets of points of $Z$ in such a way that $F_z =0$ for all $z \in Z$ seems to be a rather difficult and even hopeless task. However, in view of the asymptotic behavior of $\Upsilon$, one quickly realizes that, in the expression of $F_z$ given by (\ref{eq:expfz}), only the {\em closest neighbors of $z$ in $Z$} are of interest since the influence of the other points will be of higher order and hence, will be negligible. This suggests that we should restrict our attention to the sets of points $Z$ satisfying the following condition~: 

\begin{equation}
\label{eq:restrict}
\begin{array}{lll}
\text{\em There exists $C >0$ and $\delta >0$ such that, if $z \neq z' \in Z$, then}Ê\\[1mm] \hspace{20mm} \text{\em either} \quad  \ell \leq |z'-z|   \leq \ell + C ,\quad \text{\em or} \quad |z'-z|  \geq  (1+\delta) \, \ell.
\end{array}
\end{equation}

Here, $\ell$ is considered as a parameter which will be taken very large, while $C>0$ and $\delta >0$ are constants which are fixed (large enough) independently of $\ell$ (in particular, we assume that $C \ll \delta \, \ell$).  Under this condition, we define, for all $z \in Z$
\[
N_z : = \{ z'\in Z-\{z\} \, : \, |z'-z| \leq \ell +C\},
\]
to be the set of {\em closest neighbors of $z$ in $Z$} and, for each $z' \in N_z$, we define $\lambda_{zz'}Ê\in {\bf R}$ by
\[
|z'-z| = \ell - \lambda_{zz'} . 
\]
Under condition (\ref{eq:restrict}) and  using these notations, we find that, at main order 
\[
e^{\ell} \, \ell^{1/2}  \, F_z \sim  \sum_{z' \in N_z}  \eta_z \, \eta_{z'} \, e^{\lambda_{zz'}} \, \frac{z'-z}{|z'-z|} .
\]
Therefore, in order to find a set of points satisfying (\ref{eq:sysfz}), it is reasonable to perturb a set $Z$ for which 
\begin{equation}
\sum_{z' \in N_z} a_{zz'} \, \frac{z'-z}{|z'-z|} = 0,
\label{eq:bc}
\end{equation}
for all $z \in Z$, where we have defined 
\[
a_{zz'} : = \eta_z \, \eta_{z'} \, e^{\lambda_{zz'}}  \in {\bf R} -\{ 0 \} .
\]

In other words, the question  reduces now to be able to find a set of points $Z$, as well as parameters $a_{zz'} \in {\bf R}-\{0\}$ for each $z,z' \in Z$ such that $z' \in N_z$, in such a way that (\ref{eq:bc}) holds. But, we also need to require that
\begin{equation}
|z'-z| = \ell - \ln |a_{zz'}|,
\label{eq:dist}
\end{equation}
for all $z\neq z'\in Z$ such that $z'\in N_z$.  As we will see, finding a configuration of points $Z$ satisfying (\ref{eq:bc}) and (\ref{eq:dist}) is not an easy task but there is an explicit algorithm that leads to configurations of such points. This is what we will explain in sections~3 and 4 which, in our opinion, constitute the most important and original part of the paper. 

Once the construction of $Z$ is understood, we proceed in the next sections with the proof of Theorem~\ref{th:main} as an application of the material developed in sections~3 and 4. This starts in section~5 with the construction of the approximate solution. In section~6, we proceed with the analysis of the operator $\tilde L$. This analysis is by now standard and in fact, it borrows some elements already present in \cite{MPW}. In section~7, we use this analysis so solve (\ref{eq:nnls}) using a fixed point theorem for contraction mappings. In section~8, we prove that the expansion of $F_z$ as given by (\ref{eq:expfz}) holds.  In section~9, we give the final arguments to complete the proof of a general existence result, Theorem~\ref{th:general-result}, which guaranties the existence of infinitely many solutions of (\ref{eq:nls}). This general result, together with the examples given in section~10, will complete the proof of Theorem~\ref{th:main}.

Let us emphasize that the Liapunov-Schmidt reduction argument we use in this paper has already been used in many constructions in geometry, geometric analysis and nonlinear analysis. In our context, it is close to the arguments already used in \cite{MPW}. The main novelty in the present paper is a general construction of the sets $Z$ satisfying both (\ref{eq:bc}) and (\ref{eq:dist}). To our knowledge this analysis is completely new and it can be used for many constructions which are, in essence, similar to the ones we describe in this paper. Indeed, the material we introduce in sections~3 and 4 is common to the construction of constant mean curvature surfaces in Euclidean 3-space, the construction of solutions to the Ginzburg-Landau equation with magnetic field, the construction of solutions to the Chern-Simons-Higgs model, \ldots We shall return to this issue in section~11 and we shall give more applications of the material of sections ~3 and 4 in forthcoming papers. 

Our main theorem is very much inspired from the construction of compact and complete, non compact constant mean curvature surfaces by Kapouleas \cite{K,K2, K3}. Indeed, the construction of networks $Z$ satisfying both (\ref{eq:bc}) and (\ref{eq:dist}) which we will describe in the next sections can be easily adapted to shed light on the configurations used by Kapouleas  to construct both compact and non compact constant mean curvature surfaces and in fact this provides a systematic construction of flexible graphs used in \cite{K, K3} or $c$-graphs used in sections~2 and  3 of \cite{K2}. More precisely, what we call {\em unbalanced flexible graphs} are graphs which can be used to construct complete, non compact constant mean curvature surfaces and they can also be used to generalize the construction of  infinite energy  solutions of (\ref{eq:nls}) by Malchiodi \cite{M}. While, what we call {\em closable, balanced networks} are the ones which can be used to construct compact constant mean curvature surfaces. 

As we will see, in our case and in contrast with the analysis of \cite{K, K2, K3}, we need to restrict our attention to what we call {\em embedded networks} and we also have to handle some delicate issue which will be described in section~5. These are two additional constraints which are not present in the construction of compact (and complete, non compact) constant mean curvature surfaces. We shall further comment on this in the last section.

We should also mention the work of Traizet on the construction of minimal surfaces which have no symmetry \cite{Tra}. In this paper, finitely many parallel planes are connected together by small catenoids at specific points to produce complete, embedded minimal surfaces which have finitely many ends and in particular to produce minimal surfaces which have no symmetry. Even though the analysis of potential configurations of points is much easier in this context, it has been a source of inspiration when we were looking for a criteria which would ensure the existence of potential configurations of points $Z$ for our construction.

\section{Planar networks}

We provide a general construction of the sets $Z$ introduced in the previous section. The aim being to be able to find a systematic procedure to construct configurations of points $Z$ satisfying both (\ref{eq:bc}) and (\ref{eq:dist}).

\subsection{Definitions and basic properties}
We introduce some definitions concerning planar networks and we also present the basic properties of the objects we introduce.

As already mentioned, it will be convenient to identify ${\bf R}^2$ with the complex plane $\bf C$. The scalar product in $\bf C$ will be denoted by $\langle \, \, , \, \, \rangle_{\bf C}$ so that 
\[
\langle z, z'\rangle_{\bf C} : =  \Re \, (\bar z \, z'),
\]
and the standard symplectic form in ${\bf C}$ will be denoted by $\wedge$ so that 
\[
z \wedge z' = \langle  i \,  z,  z' \rangle_{\bf C} = \Im \, (\bar z\, z'),
\]
for all $z,z' \in {\bf C}$.

By definition, a finite planar network $\mathscr N: = (\mathscr V, \mathscr E)$ in ${\bf C}$ is given by its set of vertices $\mathscr V \subset {\bf C}$ and its set of edges $\mathscr E$ joining the vertices. If $[p,q] \in \mathscr E$, then the points $p,q \in \mathscr V$ are called the {\em end points} of the edge $[p,q]$. Naturally, we identify $[p,q]$ and $[q,p]$. The number of vertices of a given network $\mathscr N$ will be denoted by $n$ and its number of edges will be denoted by $m$ (see Fig. 1).  

For each $p \in \mathscr V$, we denote by $\mathscr V_p \subset \mathscr V$ the set of vertices $q \in \mathscr V$ such that $[p, q] \in \mathscr E$, namely
\begin{equation}
\mathscr V_p := \{ q \in  \mathscr V \, : \, [p,q] \in \mathscr E\}.
\label{eq:Vp}
\end{equation}

We have the obvious~:
\begin{definition}
A network $\mathscr N = (\mathscr V, \mathscr E)$ is said to be {\em connected} if any two of its vertices in $\mathscr V$ can be joined by a sequence of edges of $\mathscr E$, i.e. if, given $p \neq \tilde p \in \mathscr V$, there exist an integer $k \geq 1$ and a sequence $p=q_0, \ldots, q_k=\tilde p$ of points of $\mathscr V$, such that $[q_{j+1}, q_j] \in \mathscr E$, for each $j=0, \ldots, k-1$.
\label{de:net-1.1}
\end{definition}

\begin{center}
\includegraphics[width=8.5cm]{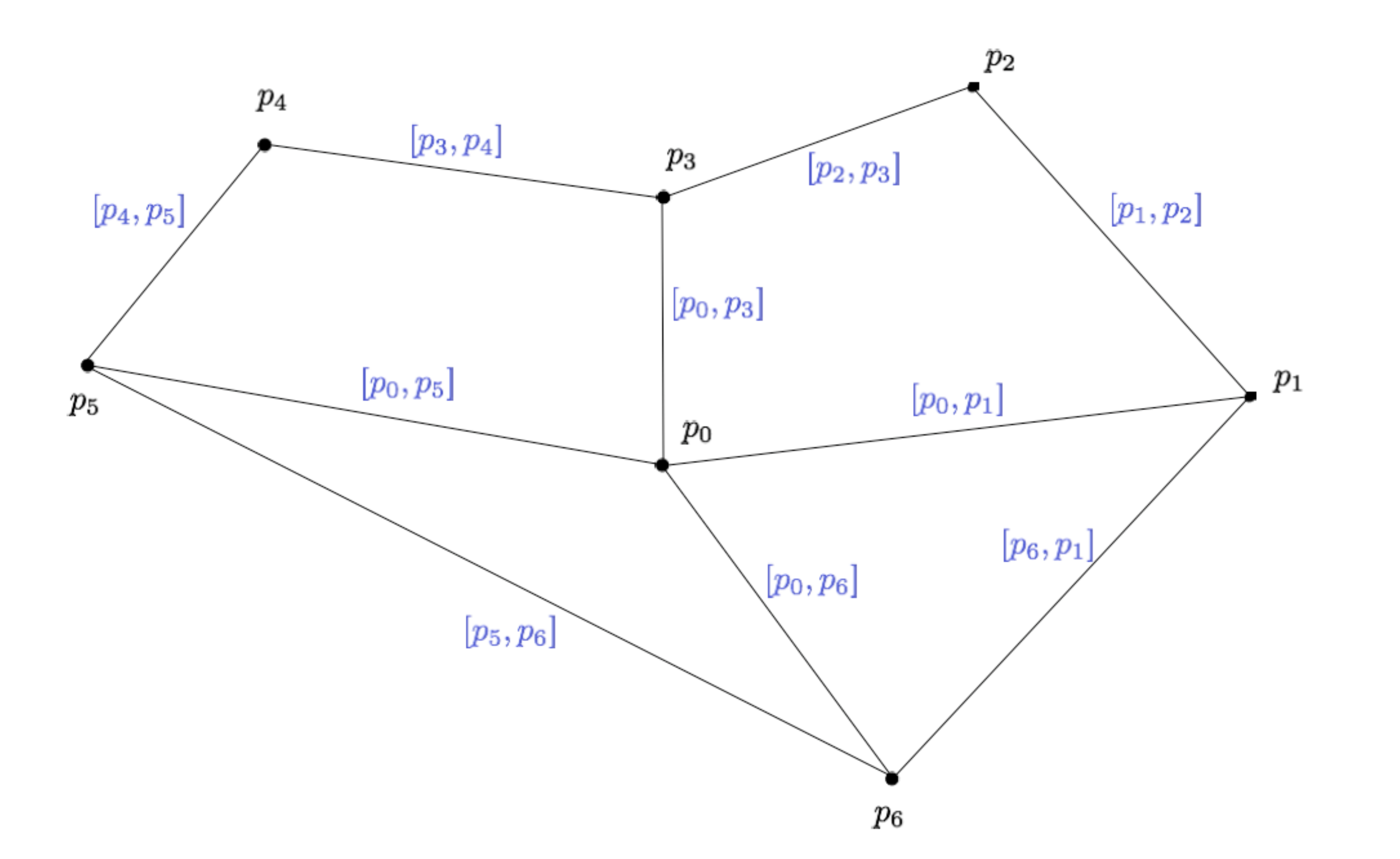} \\
Fig. 1 : An example of a network with $n=7$ vertices and $m=10$ edges.
\end{center}

The second definition is also quite natural~:
\begin{definition}
A network $\mathscr N$ is said to be {\em embedded}, if two edges $[p, q] \neq [\tilde p, \tilde q] \in \mathscr E$ are either disjoint or only intersect at one of  their end points (in which case $\{p,q\} \cap \{\tilde p, \tilde q\} \neq \varnothing$). 
\label{de:net-1.2}
\end{definition}
All the networks we consider in this paper are connected and embedded and we shall not mention these properties anymore. For other applications, for example in the construction of  compact and complete non compact constant mean curvature surfaces, it is also interesting to consider networks which are not embedded. 

The {\em length} of a network $\mathscr N$ is defined to be as the collection of the lengths of the edges of $\mathscr E$, namely
\[
{\bf L}_{\mathscr N} : =\left(  |p-q| \right)_{ [p,q] \in \mathscr E}.
\]
We have the~:
\begin{definition}
A network $\mathscr N$ is said to be {\em unitary} if $|p-q|=1$ for all $[p,q] \in \mathscr E$. 
\label{de:net-1.3}
\end{definition}

If $\mathscr N =(\mathscr V, \mathscr E)$ is a network and if $a :  \mathscr E \rightarrow {\bf R} -\{0\}$ is a function, we will say that  $(\mathscr N,a)$ is a {\em weighted network}. The image of $[p,q] \in \mathscr E$ by $a$ will be denoted by $a_{[p,q]}$. 

For all $p \in \mathscr V$, we define the {\em force of the weighted network $(\mathscr N,a)$ at the vertex $p$} by
\[
{\bf F}_{(\mathscr N,a)} (p)  : = \sum_{q \in \mathscr V_p} a_{[p,q]} \, \frac{q-p}{|q-p|} ,
\]
and 
\[
{\bf  F}_{(\mathscr N,a)}  : = \left( {\bf F}_{(\mathscr N,a)}  (p) \right)_{ p \in \mathscr V} ,
\]
is the collection of all forces at the different vertices of the weighted network $(\mathscr N,a)$. 
When there is no ambiguity, we will drop the index ${(\mathscr N,a)}$ and simply write ${\bf F} (p)$ or ${\bf F}$ instead of ${\bf F}_{(\mathscr N,a)} (p)$ and ${\bf F}_{(\mathscr N,a)}$. Observe that the force is homogeneous of degree $0$ as a function of the coordinates of the vertices and homogeneous of degree $1$ as a function of the weights of the edges. 

The following simple result will be crucial in our analysis. It is a consequence of the definition of the forces of a network.
\begin{lemma}
The following two identities hold~:
\begin{equation}
\sum_{p \in \mathscr V} {\bf F}_{(\mathscr N,a)}  (p) = 0,
\label{eq:net-1}
\end{equation}
and
\begin{equation}
\sum_{p \in \mathscr V} {\bf F}_{(\mathscr N,a)}  (p) \wedge p = 0.
\label{eq:net-2}
\end{equation}
\label{le:net-1.1}
\end{lemma}
\begin{proof}
The proofs of both identities make use of the fact that $a_{[p,q]} = a_{[q,p]}$. For example, to prove the first equality, we just compute
\[
\sum_{p \in \mathscr V} \left( \sum_{q\in \mathscr V_p} a_{[p,q]} \, \frac{q-p}{|q-p|}\right) = \sum_{[p,q] \in \mathscr E} \left( a_{[p,q]} \, \frac{q-p}{|q-p|} + a_{[p,q]} \, \frac{p-q}{|p-q|}\right)  =0.
\]
Similarly, we have
\[
\begin{array}{rllll} 
\displaystyle \sum_{p \in \mathscr V} \left( \sum_{q\in \mathscr V_p} a_{[p,q]} \, \frac{q-p}{|q-p|}\right) \wedge p & = & \displaystyle  \sum_{[p,q] \in \mathscr E} \left( a_{[p,q]} \, \frac{q-p}{|p-q|} \wedge p + a_{[p,q]} \, \frac{p-q}{|p-q|} \wedge q \right) \\[3mm]
&=  &   \displaystyle  \sum_{[p,q] \in \mathscr E} \left( a_{[p,q]} \, \frac{q-p}{|q-p|} \wedge (p - q )\right) = 0 ,
\end{array}
\]
which completes the proof of the result.
\end{proof}

We end this section by a last definition~:
\begin{definition}
A weighted network $(\mathscr N,a)$ is said to be {\em balanced} if ${\bf F}_{(\mathscr N,a)} =0$. Otherwise, we say that the weighted network $ (\mathscr N,a)$ is {\em unbalanced}. 
\label{de:net-1.4}
\end{definition}  

\subsection{Perturbed networks} Assuming that we are given a network $\mathscr N$, we would like to describe the possible {\em perturbations} of $\mathscr N$. Obviously, to describe nearby networks it is enough to describe how the vertices of $\mathscr N$ are perturbed. More precisely, we have the~: 
\begin{definition}
Given a function $\Phi : \mathscr V \to {\bf C}$, we define the {\em perturbed network} $\mathscr N_\Phi : = (\mathscr V_\Phi, \mathscr E_\Phi)$ to be the network whose set of vertices is given by
\[
\mathscr V_\Phi:=\{ \Phi_p \, : \, p \in \mathscr V\},
\]
and whose set of edges is given by 
\[
\mathscr E_\Phi : = \left\{ [\Phi_p,  \Phi_q] \, : \, [p,q] \in \mathscr E \right\},
\] 
where we adopt the notation $\Phi_p : =  \Phi (p)$.
\label{de:net-2.1}
\end{definition}
It will be convenient to label the vertices and edges of the perturbed network by the vertices and edges of the original network.  Observe that the notion of connected network is preserved under perturbation and, if a network is embedded, any small perturbation of the network is again an embedded network. 

Next we define the notion of homotopy between networks. 
\begin{definition}
We will say that two networks ${\mathscr N}_0$ and ${\mathscr N}_1$ are homotopic (respectively, unitary homothopic) if, for each $p \in \mathscr V_0$, vertex of $\mathscr N_0$, there exists a continuous function 
$$
\begin{array}{ccccllll}
[0,1] & \to &  {\bf C}\\[3mm]
s & \mapsto & \Phi_p(s) ,
\end{array}
$$
such that, for $s=0$ and $s=1$~:
\begin{itemize}
\item[(i)] the set of vertices of ${\mathscr N}_s$ is given by 
$$ 
\mathscr V_s : = \{ \Phi_p (s) \, : \, p \in \mathscr V_0\} ;
$$
\item[(ii)] the set of edges of ${\mathscr N}_s$ is given by 
$$
\mathscr E_s : = \{ [\Phi_p (s), \Phi_q (s)] \, : \, [p,q] \in \mathscr E_0\}.
$$
\end{itemize}
\end{definition}
This definition being understood, we then have a natural notion of homotopy class in the set of networks as the set of networks which are homothopic to a given network. 

Given a function $\Phi : \mathscr V \to {\bf C}$, we can define (with slight abuse of notation)
\[
{\bf L}_\Phi  : =  {\bf L}_{\mathscr  N_\Phi}  ,
\]
which is the collection of lengths of the edges of the perturbed network $\mathscr N_\Phi$.  The components of ${\bf L}_\Phi$ will be denoted by ${\bf L}_\Phi ([p,q])$ so that 
\[
{\bf L}_\Phi = \left( {\bf L}_\Phi ([p,q]) \right)_{[p,q] \in \mathscr E} .
\]
It should be clear that $\Phi \mapsto {\bf L}_\Phi$ is smooth and, if $s \mapsto \Phi (s)$ is a smooth one parameter family of maps $\Phi (s) : \mathscr V \to {\bf C}$ such that $\Phi (0) = {\rm Id}$, we can identify 
$$
\dot \Phi : = \partial_s \Phi  |_{s=0},
$$ 
with a vector $(\dot \Phi_p )_{p \in \mathscr V} \in {\bf C}^n$  and, with this identification in mind, we can view ${\rm D} {\bf L}_{\rm Id}$, the differential of ${\bf L}$ at $\Phi ={\rm Id}$, as a linear map
\[
\begin{array}{cccllll}
{\bf C}^n &  \to & {\bf R}^m \\[3mm]
\dot \Phi & \mapsto & {\rm D} {\bf L}_{\rm Id} (\dot \Phi)Ê.
\end{array}
\] 

These notations will be illustrated in the proof of the following Lemma which is straightforward and follows at once from the observation that, if $\Phi$ is the restriction  to $\mathscr V$ of an isometry of $\bf C$, then  ${\bf L}_\Phi = {\bf L}_{\rm Id}$. 
\begin{lemma}
The vectors $({\bf e})_{p \in \mathscr V}$, for ${\bf e} \in {\bf C}$, and the vector $(i \, p)_{p \in \mathscr V}$ belong to the kernel of  ${\rm D} {\bf L}_{{\rm Id}}$. 
\label{le:net-2.1}
\end{lemma}
\begin{proof}
The proof follows from the invariance of ${\bf L}$ under the action of translations and rotations in the plane. Indeed, for $s \in {\bf R}$, we define 
$$
\Phi_p (s) : = p+ s \, {\bf e},
$$ 
where ${\bf e}$ is a fixed vector of ${\bf C}$ to be the translation by $s \, {\bf e}$ or we define 
$$
\Phi_p (s) := e^{is} \, p,
$$ 
to be the restriction of the rotation of angle $s$ and center the origin in ${\bf C}$. In both cases ${\bf L}_{\Phi (s)} = {\bf L}_\Phi$, for all $s \in {\bf R}$ and differentiation with respect to $s$ at $s=0$ yields
\[
{\rm D} {\bf L}_{\rm Id}  ( \dot \Phi ) =0,
\]
where, in the former case, $\dot\Phi_p = {\bf e}$, for all $p \in \mathscr V$, while in the latter case $\dot \Phi_p = i \, p$,  for all $p \in \mathscr V$. 
\end{proof}

Similarly, we can define (with slight abuse of notation)  
\[
{\bf F}_{(\Phi , a)}  : =  {\bf F}_{(\mathscr N_\Phi , a)} ,
\]
which is the collection of forces of the weighted network $(\mathscr N_\Phi, a)$. The components of ${\bf F}_{(\Phi, a)}$ will be denoted by ${\bf F}_{( \Phi , a)} (p)$ so that 
\[
{\bf F}_{(\Phi , a)} = \left( {\bf F}_{( \Phi , a)} ( p) \right)_{p \in \mathscr V} .
\]
Again, it should be clear that $(\Phi , a)\mapsto {\bf F}_{(\Phi ,a)}$ is a smooth map and, if $s \mapsto a_s$ is a smooth one parameter family of maps $ a(s) : \mathscr E \to {\bf R}-\{0\}$ satisfying $a(0)=a$, we can identify 
$$
\dot a : = \partial_s a_s |_{s=0},
$$ 
with the vector $(\dot a_{[p,q]} )_{[p,q] \in \mathscr E} \in {\bf R}^m$  and, with this identification together with the identification we have just used in the study of ${\rm D} {\bf L}_{\rm Id}$, we can view ${\rm D} {\bf F}_{({\rm Id}, a)}$ as a linear map
\[
\begin{array}{cccllll}
{\bf C}^n  \times {\bf R}^m &  \to & {\bf C}^n \\[3mm]
 (\dot \Phi , \dot a )  & \mapsto & {\rm D}_\Phi {\bf F}_{({\rm Id}, a)} (\dot \Phi)Ê+ {\rm D}_a {\bf F}_{({\rm Id}, a)} (\dot a) ,
\end{array}
\] 
where ${\rm D}_\Phi {\bf F}_{({\rm Id}, a)}$ and ${\rm D}_a {\bf F}_{({\rm Id}, a)}$ denote the partial differentials of ${\bf F}$ with respect to $\Phi$ and $a$.

Again, the following Lemma is straightforward and follows from the observations that, if $\Phi$ is the restriction of a translation in ${\bf C}$, then ${\bf F}_{(\Phi, a)} = {\bf F}_{({\rm Id}, a)}$. 
\begin{lemma}
The following statements hold~:
\begin{itemize}
\item[(i)] The vectors $({\bf e})_{p \in \mathscr V}$, for any ${\bf e} \in {\bf C}$, and the vector $(p)_{p \in \mathscr V}$ belong to the kernel of  ${\rm D}_{\Phi} {\bf F}_{({\rm Id}, a)}$. 
\item[(ii)]  The image of ${\rm D} {\bf F}_{({\rm Id}, a)}$ is orthogonal to the space spanned by the vectors $({\bf e})_{p \in \mathscr V}$, for all ${\bf e} \in {\bf C}$.
\end{itemize}
\label{le:net-2.2}
\end{lemma}
\begin{proof}
The statement about the kernel follows as in the proof of Lemma~\ref{le:net-2.1} and also from the fact that the force is homogeneous of degree $0$ as a function of the coordinates of the vertices. While the statement about the image follows from differentiating  (\ref{eq:net-1}).
\end{proof}

By definition, when a weighted network $(\mathscr N, a)$ is balanced,  we have  ${\bf F}_{({\rm Id} , a)} = 0$. Going back to the definition of the forces, we see that ${\bf F}_{(\Phi , a)} = 0$ if $\Phi$ is the restriction to $\mathscr V$ of a rotation of ${\bf C}$ and we check that ${\bf F} (\Phi, \lambda a) = 0$ for any $\lambda\in {\bf R}$. Let us emphasize that these two invariance only hold when the network is balanced. This, together with the previous Lemma, implies the~:
\begin{lemma}
Assume that the weighted network $(\mathscr N, a)$ is balanced, then the following statements hold~:
\begin{itemize}
\item[(i)]  The vectors $({\bf e})_{p \in \mathscr V}$, for any ${\bf e} \in {\bf C}$, the vector $(p)_{p \in \mathscr V}$ and the vector $(i \, p)_{p \in \mathscr V}$ belong to the kernel of ${\rm D}_{\Phi} {\bf F}_{({\rm Id}, a)}$.
\item[(ii)] The vector $(a_{[p,q]})_{[p,q] \in \mathscr E}$ belongs to the kernel of ${\rm D}_{a} {\bf F}_{({\rm Id}, a)}$. 
\item[(iii)] The image of ${\rm D} {\bf F}_{({\rm Id}, a)}$ is orthogonal to the space spanned by the vectors $({\bf e})_{p \in \mathscr V}$, for ${\bf e} \in {\bf C}$, and the vector $(i \, p)_{p \in \mathscr V}$.  
\end{itemize}
\label{le:net-2.3} 
\end{lemma}
\begin{proof}
The statement about the kernel of these linear operators follows the proof of Lemma~\ref{le:net-2.1} and is left to the reader. The statement about the image of the operator  ${\rm D} {\bf F}_{({\rm Id}, a)}$follows from differentiating (\ref{eq:net-1}) and (\ref{eq:net-2}). 
\end{proof}

To summarize the above analysis, we assume that we are given a weighted network $({\mathscr N}, a)$ and we define the linear map
\begin{equation}
\begin{array}{ccccccllll}
\Lambda & :  & {\bf C}^n  \times {\bf R}^m &  \to & {\bf C}^n \times {\bf R}^m  \\[3mm]
&  & (\dot \Phi, \dot a) & \mapsto  & \left( {\rm D} {\bf F}_{({\rm Id}, a)} (\dot \Phi ,\dot a) \, , \, {\rm D} {\bf L}_{\rm Id} (\dot \Phi)  \right).
\end{array}
\label{eq:net-3}
\end{equation}
If the weighted network $({\mathscr N}, a)$ is unbalanced we have proved that $\Lambda$ has kernel of dimension at least $2$ and cokernel of dimension at least $2$, while, if the network $({\mathscr N}, a)$ is balanced, then $\Lambda$ has kernel of dimension at least $4$ and cokernel of dimension at least $3$.

We complete this section by the proof of the following result which will simplify some of the statements to come~:
\begin{proposition}
The following identity holds
\[
\langle  {\rm D}_{a} {\bf F}_{({\rm Id}, a)} (\dot a)  , \dot \Phi  \rangle_{{\bf C}^n}  =  \langle  \dot a , {\rm D} {\bf L}_{{\rm Id}} (\dot \Phi) \rangle_{{{\bf R}}^m}  . 
\]
In other words, the linear maps ${\rm D} {\bf L}_{{\rm Id}}$ and ${\rm D}_{a} {\bf F}_{({\rm Id}, a)}$  are adjoint of each other. 
\label{pr:net-2.1}
\end{proposition}
\begin{proof}
We have 
\[
{\rm D}_{a} {\bf F}_{({\rm Id}, a)} (\dot a)  = \displaystyle  \left( \sum_{q \in \mathscr V_p} \dot a_{[p,q]} \frac{p-q}{|p-q|} \right)_{p \in \mathscr V} ,
\]
and 
\[
{\rm D} {\bf L}_{{\rm Id}} (\dot \Phi)   = \left( \frac{ \langle p-q , \dot \Phi(p) - \dot \Phi (q)\rangle_{\bf C}}{|p-q|} \right)_{[p,q] \in \mathscr E}.
\]
The result then follows from the observation that 
\[
\begin{array}{rlllll}
\langle {\rm D}_{a} {\bf F}_{({\rm Id}, a)} (\dot a) , \dot \Phi \rangle_{{\bf C}^n} & = & \displaystyle \sum_{p \in \mathscr V}  \left\langle  \left( \sum_{q \in \mathscr V_p} \dot a_{[p,q]} \frac{p-q}{|p-q|} \right) , \dot \Phi (p) \right\rangle_{\bf C}  \\[3mm]
& = & \displaystyle \sum_{[p,q] \in \mathscr E} \dot a_{[p,q]} \,  \frac{ \langle p-q , \dot \Phi(p) - \dot \Phi (q)\rangle_{\bf C}}{|p-q|} \\[3mm]
& = & \langle  \dot a , {\rm D} {\bf L}_{{\rm Id}} (\dot \Phi) \rangle_{{{\bf R}}^m} ,
\end{array}
\]
and the proof is complete.
\end{proof}

\subsection{Flexible unbalanced networks} There are two different notions of flexible networks which will be needed in our construction depending whether they apply to balanced or unbalanced networks. We first introduce the notion of flexibility for unbalanced network since it is the easiest to understand. We then give examples of networks which are unbalanced, flexible and also unitary, since these are the (only) ones which are useful in applications. We keep the notations introduced in the previous section.

As mentioned above, the first notion of flexibility applies to {\em unbalanced networks}~:
\begin{definition}
An unbalanced network $(\mathscr N, a)$ is said to be {\em flexible} if the mapping $\Lambda$ defined in (\ref{eq:net-3}) has rank $2n + m -2$.
\label{de:net-3.1}
\end{definition}
According to Lemma~\ref{le:net-2.1} and Lemma~\ref{le:net-2.2}, the linear map $\Lambda$ introduced in (\ref{eq:net-3}) has kernel whose dimension is at least $2$ and image whose codimension is at least $2$. Therefore, asking that the unbalanced network is flexible is nothing but asking that the rank of $\Lambda$ is as large as allowed by these lemmas.

For an unbalanced network to be flexible, it is necessary that 
\[
m \leq 2n-3.
\]
That is, the number of edges should not be too large compared to the number of vertices of the network. 
Indeed, the dimension of the image of ${\rm D} {\bf L}_{\rm Id}$ is necessarily less than $2n-3$ since this mapping has at least a $3$-dimensional kernel (see Lemma~\ref{le:net-2.1}). Moreover, the dimension of the image of ${\rm D} {\bf F}_{({\rm Id}, a)}$ is at most $2n-2$ (see (ii) in Lemma~\ref{le:net-2.2}). And hence, the rank of $\Lambda$ is at most equal to $4n-5$.  So, in order for the rank of $\Lambda$ to be equal to $2n+m-2$, it is necessary that $m\leq 2n-3$. 

Before we proceed with examples, let us observe that the property of being flexible is an open property among unbalanced networks. More precisely, we have the~:
\begin{proposition}
\label{pr:Zd}
The set of flexible unbalanced networks in a given homotopy class (or in a given unitary homotopy class) is  Zariski open.
\end{proposition}
\begin{proof}
Indeed, requiring that a network (or a unitary network) is flexible amounts to say that the rank of $\Lambda$ is equal to $2n+m-3$ and, in coordinates, this can be translated into the fact that one of the square sub-matrix of $\Lambda$ of size $2n+m-3$ has non-zero determinant. Computing the sum of the squares of all square sub-matrices of $\Lambda$ of size $2n+m-3$ give an algebraic functions of the coordinates of the points of the network and  the coefficients of the weights.  The set of weighted networks (or unitary networks) which are not flexible correspond to the zero set of this algebraic function and hence its complement is (by definition) Zariski open. \end{proof}

As a consequence, we see that the set of unitary networks which are unitary homotopic to a given unitary, flexible network, is non empty and Zariski open. 

We now give a series of examples of flexible, unbalanced networks which have in addition the property of being unitary. 

\medskip

\noindent {\bf Example 3.1 :}  \label{ex:3.1}
Given $n\geq 2$, the simplest unbalanced,flexible network one can imagine is the network $\mathscr N_I$ whose set of vertices is given by 
$$
\mathscr V_I : =\{ z_0, \ldots, z_{n-1} \}\subset {\bf C},
$$
where we assume that $|z_{j+1}-z_j|=1$ for all $j=0, \ldots, n-2$ and whose set of edges is defined by
\[
\mathscr E_I := \{ [z_j, z_{j+1}] \, : \, j=0, \ldots, n-2\}.
\]

\begin{center}
\includegraphics[width=10cm]{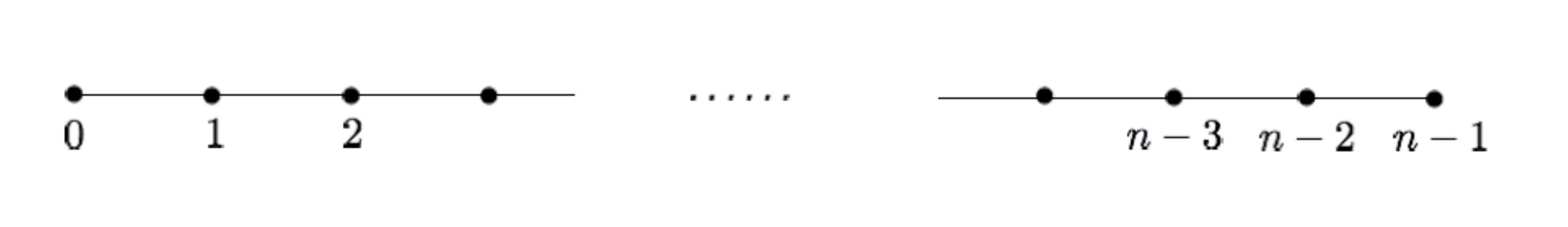}\\
Fig 2 : Example of a network $\mathscr N_I$ (here $z_j = j\in {\bf C}$, for $j=0, \ldots, n-1$).
\end{center}

This is by definition a unitary network and, if we are given $a : \mathscr E_I \to {\bf R}-\{0\}$, this provides an example of an unbalanced network. Indeed, the force at the point $z_0 \in {\mathscr V}_I$ or at the point $z_{n-1}\in {\mathscr V}_I$ are given respectively by
\[
{\bf F}_{({\mathscr N}_I, a )} (z_0) = a_{[0,1]} \, (z_1-z_0)  \qquad \text{and} \qquad  {\bf F}_{({\mathscr N}_I, a)} (z_{n-1}) = a_{[n-2,n-1]} \, (z_{n-1} -z_{z-2}) ,
\] 
and they are not equal to $0$ by definition of the weight function $a$, hence, the network $({\mathscr N}_I , a)$ is always unbalanced. 

\begin{lemma}
We claim that the weighted unbalanced network $({\mathscr N}_I , a)$ defined above is flexible in the sense of Definition~\ref{de:net-3.1}.
\end{lemma}
\begin{proof}
In this example, $m =n-1$ and hence, we need to check that the mapping $\Lambda$ defined in (\ref{eq:net-3}) has rank equal to $2n+m-2 = 3n-3$. To keep the notations short, it is convenient to write
\[
\dot \Phi_{z_{j+1}} - \dot \Phi_{z_j}  := (z_{j+1}-z_j) \, \dot w_j ,
\] 
where $\dot w_j \in {\bf C} $ for $j=0, \ldots, n-2$ and we agree that $\dot w_{-1} =  \dot w_{n-1}  = 0$.  Also, we agree that $\dot a_{[z_{-1},z_0]}$ and $\dot a_{[z_{n-1}, z_n]}$ are both equal to $0$.  

With these notations, we find that
\begin{equation}
{\rm D} {\bf L}_{{\rm Id}} (\dot \Phi) = \left(  \Re \, \dot w_{j} \right)_{j=0, \ldots, n-2}.
\label{eq:dl}
\end{equation}
Also, we have a nice expression for 
\begin{equation}
{\rm D}_\Phi{\bf F}_{({\rm Id}, a)} ( \dot \Phi ) = \left( i \,  \left(a_{[z_{j}, z_{j+1}]} \, (z_{j+1}- z_j) \,  \Im \, \dot w_j- a_{[z_{j-1}, z_{j}]} \, (z_{j} -z_{j-1})) \, \Im \, \dot w_{j-1} \right) \right)_{j=0, \ldots, n-1},
\label{eq:df1}
\end{equation}
and
\begin{equation}
{\rm D}_a{\bf F}_{({\rm Id}, a)} ( \dot a ) = \left(   \dot a_{[z_j, z_{j+1}]} \, (z_{j+1}- z_j) -  \dot a_{[z_{j-1}, z_{j}]} \, (z_{j} - z_{j-1})  \right)_{j=0, \ldots, n-1}.
\label{eq:df2}
\end{equation}
Now, if ${\rm D}{\bf L}_{{\rm Id}} (\dot \Phi) =0$, then $\Re w_j=0$ for all $j=0, \ldots, n-2$. Next, if ${\rm D}{\bf F}_{({\rm Id}, a)} (\dot \Phi ,  \dot a ) =0$, looking at the component at the vertex $z_0$, we get $\dot a_{[z_0, z_1]} =0$ and $\Im \dot w_{0}=0$. Arguing recursively, one concludes that $\dot a_{[z_{j+1}, z_j]} =0$ for all $j=0, \ldots, n-2$ and $\dot w_j=0$ for $j=0, n-2$. Therefore, as a function of $\dot w_j$ and $\dot a_{[z_{j+1},z_j]}$, the mapping $\Lambda$ is injective and this implies that $\Lambda$ has rank $3n-3$. 
\end{proof}

\medskip

\noindent {\bf Example 3.2 :}
\label{ex:3.2}
Given $n\geq 3$, we consider the network $\mathscr N_{Pol}$ defined by an embedded polygon with $n$ sides of size $1$. Hence, the set of vertices of this network is given by
\[
\mathscr V_{Pol} :=\left\{ z_j  \, : \, j=0, \ldots, n-1 \right \}.
\]
We agree to extend the sequence $z_0, \ldots, z_{n-1}$ as a $n$ periodic sequence $(z_j)_{j\in {\bf Z}}$. The set of edges of this network is defined to be 
\[
\mathscr E_{Pol} : = \{[z_j, z_{j+1}] \, : \, j=0, \ldots, n-1\}.
\]
being understood that $[z_{n-1}, z_n] = [z_{n-1}, z_0]$ in agreement with the fact that we have extended periodically the sequence $z_0, \ldots, z_{n-1}$.

\begin{center}
\includegraphics[width=4cm]{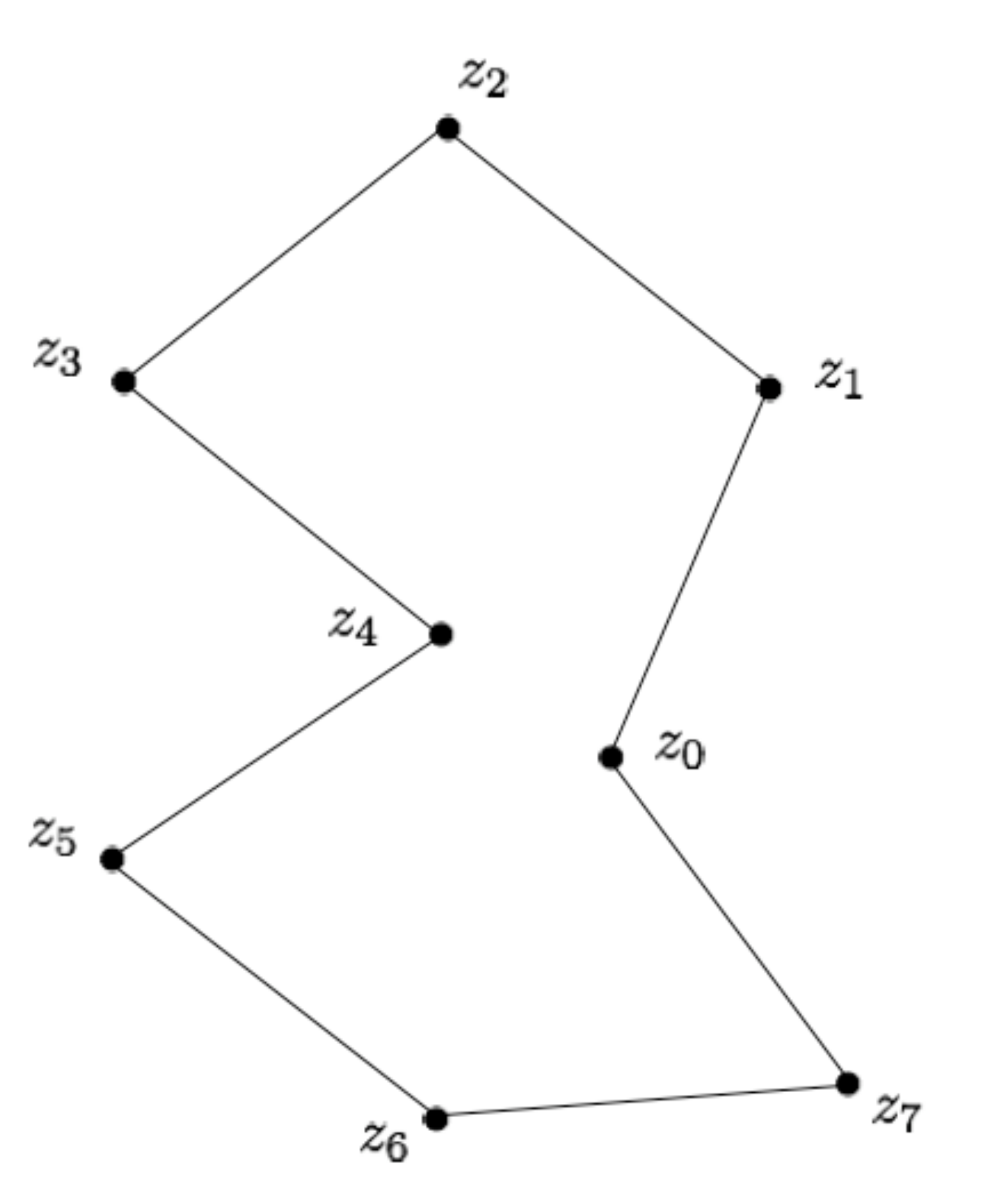}\\
Fig 3 : Example of a network $\mathscr N_{Pol}$ with $n=8$. All edges have length $1$.
\end{center}

This is clearly a unitary network and, in this example, the number of vertices  and the number of edges are both equal to $n$. If we are given $a : \mathscr E_{Pol} \to {\bf R}- \{0\}$, this provides another example of an unbalanced network. Indeed, the force at the point $z_j \in {\mathscr V}$ is given by
\[
{\bf F}_{({\mathscr N}_{Pol}, a )} (z_j) =  a_{[z_j, z_{j+1}]} \,  (z_{j+1}-z_j) - a_{[z_{j-1}, z_{j}]} \,  (z_{j}-z_{j-1}) ,
\] 
which cannot all be equal to $0$ and this implies that the network is not balanced. In this case, we have the~:
\begin{lemma}
The unbalanced network $(\mathscr N_{Pol}, a)$ is flexible in the sense of Definition~\ref{de:net-3.1} if and only if 
\[
A : = \sum_{j=0}^{n-1} \frac{\Re \, (z_{j+1}-z_j)}{a_{[z_j, z_{j+1}]}} \, (z_{j+1} -z_j)  \qquad \text{and} \qquad  B : = \sum_{j=0}^{n-1} \frac{\Im \, (z_{j+1}-z_j)}{a_{[z_j, z_{j+1}]}} \, (z_{j+1} -z_j) .
\]
are ${\bf R}$ linearly independent. 
\label{lemm:degeneral}
\end{lemma}
\begin{proof}
In this example, $n=m$ and hence we need to check that the mapping $\Lambda$ has rank $3m-2$. When studying the rank of $\Lambda$, it is convenient to write
\[
\dot \Phi_{z_{j+1}}-  \dot \Phi_{z_j} =  (z_{j+1} - z_j) \,  \dot w_j ,
\]
where $\dot w_j  \in {\bf C}$. We agree that we extend $\dot w_{j}$ and $\dot a_{[z_j, z_{j+1}]}$ periodically to all indices $j \in {\bf Z}$. Observe that 
\begin{equation}
\sum_{j=0}^{n-1} (z_{j+1} -z_j) \, \dot w_j  =0.
\label{eq:fg}
\end{equation}
Therefore, to show that $\Lambda$ has rank $3m-2$, it is enough to prove that $\Lambda$, as a function of $\dot w_j$ and $\dot a_{[z_j, z_{j+1}]}$, is injective.  So let us assume that 
\[
{\rm D}{\bf L}_{\rm Id} (\dot \Phi) =0 \qquad \text{and} \qquad {\rm D}{\bf F}_{({\rm Id}, a)} (\dot \Phi ,  \dot a ) =0.
\]

With the above notations, (\ref{eq:dl}), (\ref{eq:df1}) and (\ref{eq:df2}) still hold. Now, if ${\rm D}{\bf F}_{({\rm Id}, a)} (\dot \Phi ,  \dot a ) =0$, we see from the above expression that  
\[
\left( \dot a_{[z_j, z_{j+1}]}  +  i \, a_{[z_{j}, z_{j+1}]} \, \Im \, \dot w_j  \right) \, (z_{j+1} -z_j) ,
\]
does not depend on $j$.  This implies that there exists $\eta \in {\bf C}$ such that
\begin{equation}
 a_{[z_{j}, z_{j+1}]}  \, \Im \, \dot w_j =  \Re \left(  \frac{\eta}{z_{j+1}-z_j}  \right) ,
\label{les}
\end{equation}
for $j=1, \ldots, n-1$. 

Now, using these expression into (\ref{eq:fg}), yields 
\[
 A \, \Re \, \eta + B \, \Im \, \eta = 0,
\]
where
\[
A : = \sum_{j=0}^{n-1} \frac{\Re \, (z_{j+1}- z_j)}{a_{[z_j, z_{j+1}]}} \, (z_{j+1} -z_j) \qquad \text{and} \qquad  B : = \sum_{j=0}^{n-1} \frac{\Im  \, (z_{j+1}- z_j)}{a_{[z_j, z_{j+1}]}} \, (z_{j+1} -z_j).
\]
If $A$ and $B$ are ${\bf R}$ linearly independent, we conclude that $\eta =0$ and this proves that the rank of $\Lambda$ is equal to $3m-2$. 
\end{proof}

\medskip

Let us consider the special case where  the network is a {\em regular} polygon with $n$ edges of length $1$. Hence, the set of vertices of the network $\mathscr N_{Reg Pol}$ is given by
\[
\mathscr V_{Reg Pol} :=\left\{ z_j : = \frac{\xi^j}{|1-\xi|}  \in {\bf C} \, : \, j=0, \ldots, n-1 \right \},
\]
where  $\xi := e^{2 i\pi/n}$. The set of edges of this network is defined to be 
\[
\mathscr E_{Reg Pol} : = \{[z_j, z_{j+1}] \, : \, j=0, \ldots, n-1\},
\]
where as usual $z_n:=z_0$. We choose the weight function to be given by
\[
a_{[z_j, z_{j+1}]} =1,
\]
for all $j=0, \ldots, n-1$.

\begin{center}
\includegraphics[width=5cm]{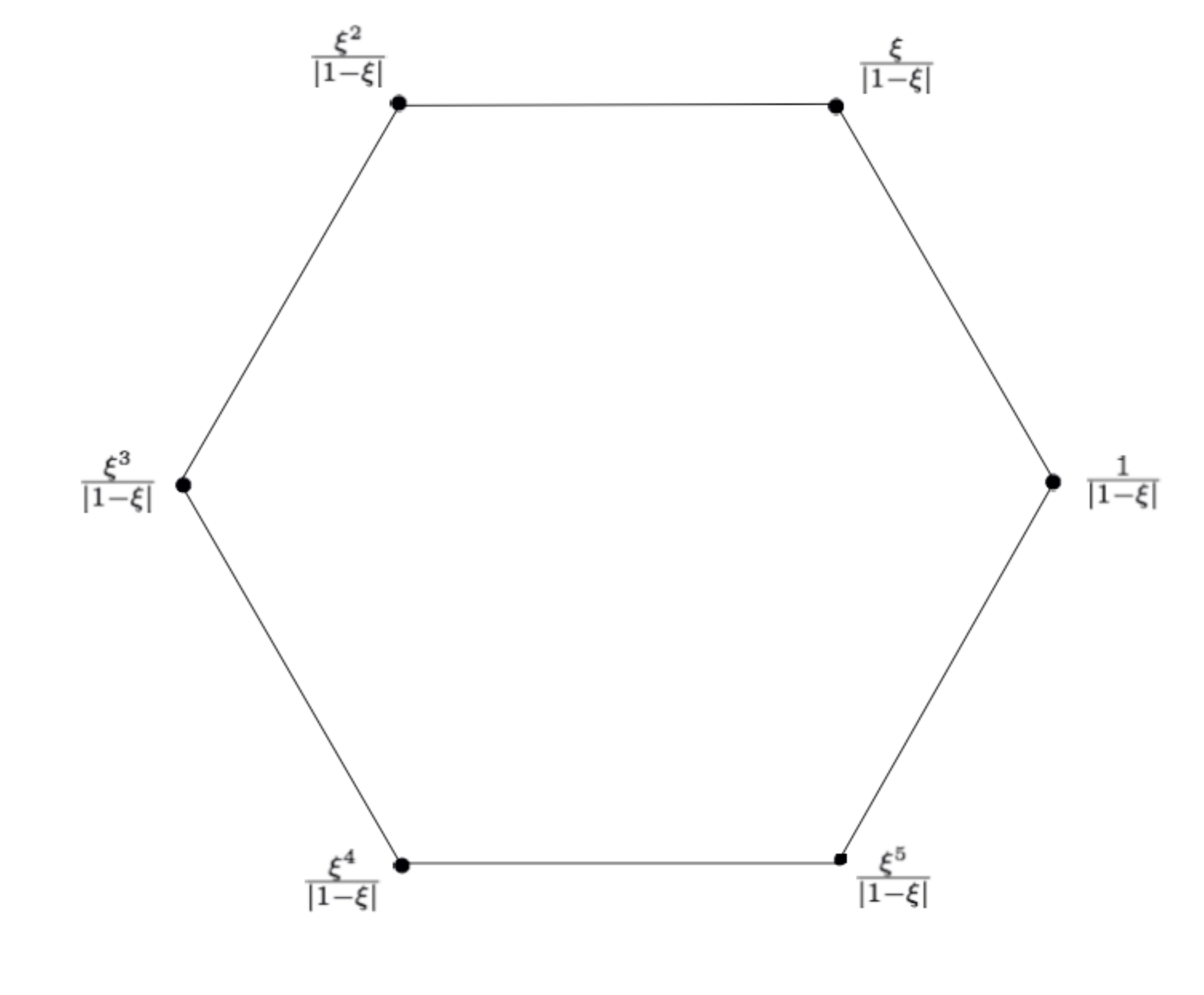}\\
Fig 4 : Example of a network $\mathscr N_{Reg Pol}$ when $n=6$ (i.e. a regular hexagon with edges of length $1$).
\end{center}

In this case, Lemma~\ref{lemm:degeneral} reads
\begin{corollary}
The unbalanced network $(\mathscr N_{RegPol}, a)$, when the weight function $a$ is constant, is flexible in the sense of Definition~\ref{de:net-3.1}.
\label{lemm:de}
\end{corollary}
\begin{proof}
In this special case where the network is a regular polygon and where the weight function is constant, we have 
\[
A : = \sum_{j=0}^{n-1} \Re \, \xi^{n-j} \, \xi^j \qquad \text{and} \qquad  B : = \sum_{j=0}^{n-1} \Im \, \xi^{n-j}  \, \xi^j,
\]
and is easy to check that $A$ and $B$ are ${\bf R}$ linearly independent. According to Lemma~\ref{lemm:degeneral}, this shows that the corresponding unbalanced network is flexible. 
\end{proof}

Another interesting application is the one where,  given $n\geq 3$ and $k \geq 1$, we consider the network $\mathscr N_{RegPol, k}$ defined to be a regular regular polygon with $n$ edges of length $k$. Observe that $\mathscr N_{RegPol, 1}$ corresponds to $\mathscr N_{RegPol}$. The set of vertices of this network is given by
\[
\begin{array}{rlll}
\mathscr V_{RegPol,k} & := & \Big\{ z_{j, j'} : = \displaystyle \frac{1}{|1-\xi|}Ê(k \, \xi^j + j' \, (\xi^{j+1} - \xi^j))  \in {\bf C} \, : \, j=0, \ldots, n-1, \\[3mm]
&  & \hspace{68mm}  j'=0, \ldots, k-1 \Big\},
\end{array}
\]
where  $\xi := e^{2 i\pi/n}$. The set of edges of this network is defined to be
\[
\mathscr E_{RegPol,k} : = \{[z_{j,j'} , z_{j,j'+1}] \, : \, j=0, \ldots, n-1 , \, j'=0, \ldots, k-1 \}.
\]

\begin{center}
\includegraphics[width=5cm]{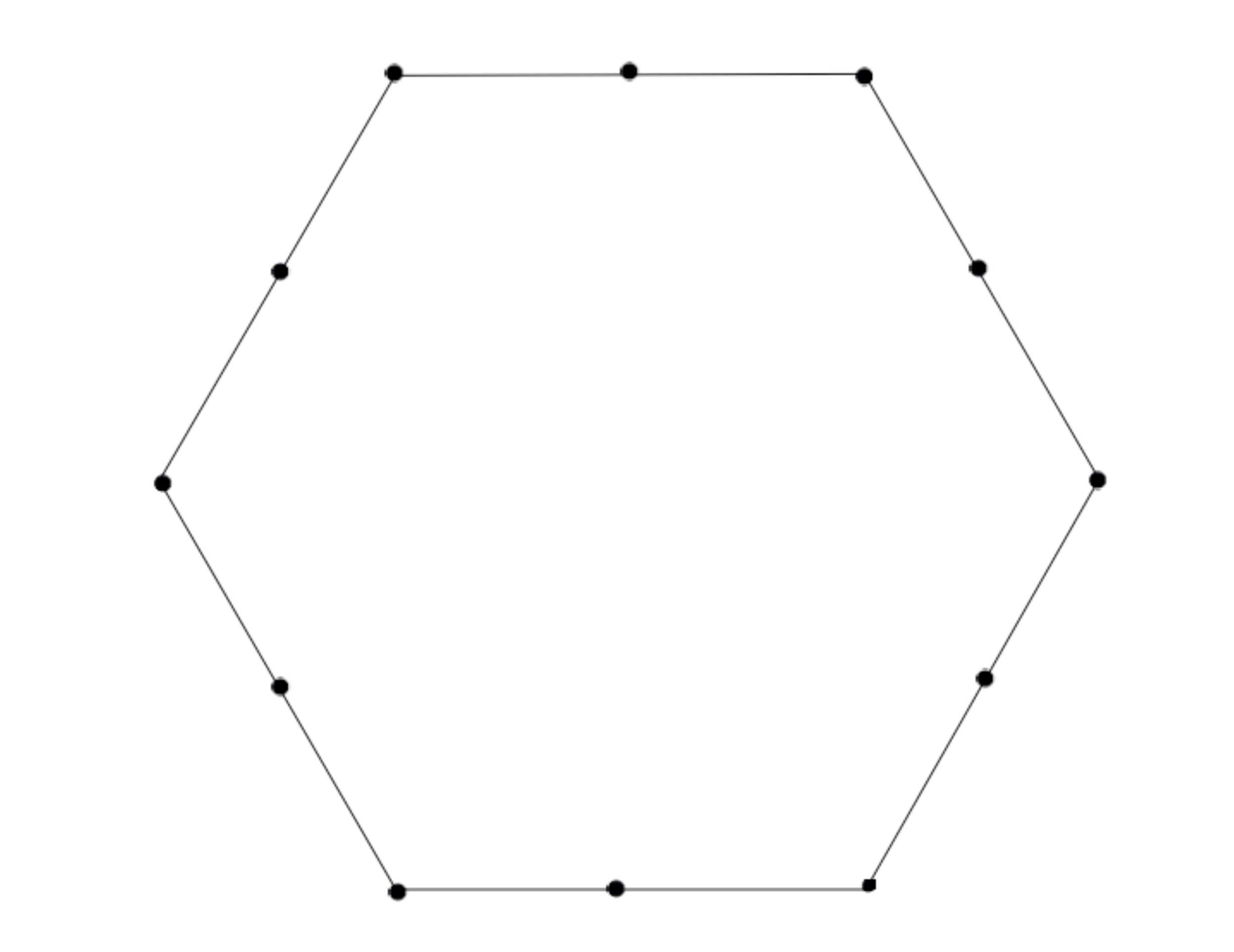}\\
Fig 5 : When $n=6$ and $k=2$, we get a unitary network which is an hexagon whose edges have length $2$.
\end{center}

This is clearly a unitary network and, in this example, the number of vertices  and the number of edges  are both equal to $k \, n$. If we are given $a : \mathscr E_{RegPol,k} \to {\bf R}- \{0\}$, this provides another example of an unbalanced network.  To simplify the discussion, let us assume that the weight function is chosen to be
\[
a_{[z_{j,j'} , z_{j,j'+1}] } = 1,
\]
for all $j=1, \ldots, n-1$ and for all $j'=0, \ldots, k-1$. Then, we have the following result whose proof is left to the reader~:
\begin{corollary}
The unbalanced network $(\mathscr N_{RegPol,k}, a)$ is flexible in the sense of Definition~\ref{de:net-3.1}.
\end{corollary}

Again, Proposition~\ref{pr:Zd}  implies that the set of unitary, unbalanced networks which are unitary homotopic to the network $\mathscr N_{RegPol,k}$ is non empty and Zariski open and this is in agreement with the result of Lemma~\ref{lemm:degeneral} which gives a general condition to ensure the flexibility of such weighted networks.

\medskip

\noindent {\bf Example 3.3 :} \label{ex:3.3} To illustrate further the result of Proposition~\ref{pr:Zd}, let us focus our attention on the network ${\mathscr N}_{RegPol}$ when $n=3$. In this case there is only one homotopy class of unitary networks corresponding to the equilateral triangle. Therefore, the question which remains is the following~: for which weight function $a$ is the equilateral triangle  flexible ? 

We consider the unitary network $\mathscr N_{Tri}$ defined by an equilateral triangle with vertices 
\[
\mathscr V_{Tri} :=\left\{ z_0: = \frac{1}{\sqrt 3} , z_1 : = \frac{\zeta}{\sqrt 3}, z_2 : = \frac{\zeta^2}{\sqrt 3} \right\},
\]
where  $\zeta := e^{2 i\pi/3}$. The set of edges of this network is defined to be 
\[
\mathscr E_{Tri} : = \{[z_0, z_1] , [z_1 , z_2] , [z_2, z_0] \}.
\]
Given $a : \mathscr E_{Tri} \to {\bf R}- \{0\}$, this provides an unbalanced network. We claim that~:
\begin{lemma}
The unbalanced network $(\mathscr N_{Tri}, a)$ is flexible, in the sense of Definition~\ref{de:net-3.1} if and only if 
\[
a_{[z_0, z_1]} + a_{[z_1, z_2]} + a_{[z_2, z_0]} \neq 0.
\]
\label{le:refd}
\end{lemma}
\begin{proof} We keep the notations of the proof of Lemma~\ref{lemm:degeneral}. Starting from the fact that 
\[
\left( \dot a_{[z_j, z_{j+1}]}  +  i \, a_{[z_{j}, z_{j+1}]} \, \Im \, \dot w_j  \right) \, (z_{j+1} -z_j) ,
\]
does not depend on $j$,  we get
\[
\left\{
\begin{array}{rllll}
 a_{[z_1, z_2]} \, \Im \, \dot w_1 & =  & \Im \, \zeta^2 \, \dot a_{[z_0, z_1]}  + a_{[z_0, z_1]} \, \Re \, \zeta^2 \, \Im \, \dot w_0\\[3mm]
 a_{[z_2, z_0]} \, \Im \, \dot w_2 & =  & \Im \, \zeta \, \dot a_{[z_0, z_1]} +  a_{[z_0, z_1]} \, \Re \, \zeta \, \Im \, \dot w_0.
\end{array}
\right.
\]
and taking the sum of these two identifies and using the fact that $\Im (\zeta  + \zeta^2)  =0$, we get 
\[
a_{[z_1 , z_2]} \, \Im \, \dot w_1+ a_{[z_2,z_0]} \, \Im \, \dot w_2  =  - a_{[z_0, z_1]} \,  \Im \, \dot w_0 ,
\]
since $\Re \, \zeta^2 = \Re \, \zeta = -1/2$.  Moreover, (\ref{eq:fg}) implies that 
\[
\Im \, \dot w_0 + \zeta \, \Im \, \dot w_1  + \zeta^2 \,  \Im \, \dot w_2 = 0.
\]
Taking the real part and imaginary part of this last equation, we conclude that $\Im \,\dot w_0 =\Im \, \dot w_1 =\Im \, \dot w_2$. Hence, we have
\[
 \left( a_{[z_0, z_1]}  + a_{[z_1, z_2]} + a_{[z_2, z_0]}  \right) \, \Im \, \zeta \, \dot w_0 =0.
\] 
 Therefore, we have proven that $\Lambda$ has rank $7$ if and only if $a_{[z_0, z_1]} + a_{[z_1, z_2]} + a_{[z_2, z_0]} \neq 0$. 
\end{proof}

Given $\theta \in {\bf R}$, we define the network $\mathscr N_{Tri ,\theta}$ which is obtained from $\mathscr N_{Tri}$ after a rotation of angle $\theta \in {\bf R}$. Hence, the vertices of $\mathscr N_{Tri, \theta}$ are given by $e^{i\theta} z_j$ where $z_j$ are the vertices of $\mathscr N_{Tri}$.  A natural question is the following~: {\em Given ${\bf f}_0, {\bf f}_1$ and ${\bf f}_2 \in {\bf C}$, is it possible to find an angle $\theta \in {\bf R}$ and a weight function $a :  \mathscr E_{Tri, \theta} \to {\bf R}-\{0\}$ such that 
\[
F_{(\mathscr N_{Tri ,\theta}, a)} (e^{i\theta} z_j) = {\bf f}_j , 
\]
for all $j=0,1,2$ ?} 

The answer to this question is given by the~:
\begin{proposition}
Assume that ${\bf f}_0+ {\bf f}_1 + {\bf f}_2 =0$. Then, there exit $\theta \in {\bf R}$ and a weight function $a : \mathscr E_{Tri} \to {\bf R}-\{0\}$ such that 
\[
F_{(\mathscr N_{Tri, \theta }, a)} (e^{i\theta} \, z_j) = {\bf f}_j ,
\]
for $j=0,1$ and $2$. Moreover, the choice of $\theta$ and $a$ is unique if and only if
\[
{\bf f}_j \neq \zeta^2 \, {\bf f}_{j-1},
\]
for  $j=0,1,2$ (observe that inequality for some $j$ implies the inequality for all $j$). 
\label{pr:rfd}
\end{proposition}
\begin{proof}
We have to find $\theta$ and $a$ such that 
\[
\left\{ \begin{array}{rllll}
e^{i\theta} \left(  a_{[1,\zeta^2]} \, (\zeta^2-1) + a_{[\zeta, 1]} \, (\zeta -1) \right) & = & |1-\zeta| \, {\bf f}_0,\\[3mm]
e^{i\theta} \left(  a_{[\zeta, 1]} \, (1 -\zeta) + a_{[\zeta^2, \zeta]} \, (\zeta ^2-\zeta) \right) & = & |1-\zeta| \, {\bf f}_1,\\[3mm]
e^{i\theta} \left(  a_{[1, \zeta^2]} \, (1 -\zeta^2) + a_{[\zeta^2, \zeta]} \, (\zeta -\zeta^2) \right) & =&  |1-\zeta| \, {\bf f}_2.
\end{array}
\right.
\]
Using the second and third equations, we get
\[
\left\{ \begin{array}{rllll}
 a_{[\zeta, 1]}  -  a_{[\zeta^2, \zeta]} \, \zeta   & =  &\displaystyle   |1-\zeta| \, {\bf f}_1\, \frac{e^{-i\theta}}{1-\zeta} ,\\[3mm]
 \displaystyle a_{[1, \zeta^2]}  + a_{[\zeta^2, \zeta]} \, \frac{\zeta}{1+ \zeta} & = &\displaystyle  |1-\zeta| \, {\bf f}_2  \, \frac{e^{-i\theta}}{1-\zeta^2} .
\end{array}
\right.
\]
Taking the real part of each equation gives the formula for $a_{[\zeta, 1]}$ and $a_{[1, \zeta^2]}$ in terms of $a_{[\zeta^2, \zeta]}$ and $\theta$. Next, taking the imaginary part of both equations we get
\[
 a_{[\zeta^2, \zeta]} \, \Im \, \zeta   = - \displaystyle   |1-\zeta| \, \Im \left( {\bf f}_1\, \frac{e^{-i\theta}}{1-\zeta}\right) ,
\]
which gives $a_{[\zeta^2, \zeta]}$ as a function of $\theta$. But we also get 
\[
\Im \left( e^{-i\theta} \, \frac{\zeta^2 \, {\bf f}_1 - {\bf f}_2}{1-\zeta^2} \right) =0,
\]
which determines the value of $\theta$. Observe that this last equation is uniquely solvable if and only if ${\bf f}_2\neq \zeta^2 \, {\bf f}_1$.  \end{proof}

\begin{remark}
It is interesting to compare the result of Lemma~\ref{le:refd} and the result of Proposition~\ref{pr:rfd}. In the above Proposition, one can check that, if ${\bf f}_j = \zeta^2 \, {\bf f}_{j-1}$ for  $j=0,1,2$, then
\[
a_{[e^{i\theta} z_0, e^{i\theta}z_1]}  + a_{[e^{i\theta}z_1, e^{i\theta}z_2]} + a_{[e^{i\theta}z_2,e^{i\theta} z_0]} =0,
\]
and the non uniqueness of $\theta$ and $a$ in Proposition~\ref{pr:rfd} is in agreement with the result of Lemma~\ref{le:refd}.
\end{remark}
 
\subsection{Flexible balanced networks} Let us now focus on {\em balanced weighted networks} for which we also introduce the notion of flexibility. 
\begin{definition}
A balanced weighted network $(\mathscr N, a)$ is said to be flexible if the mapping $\Lambda$, defined in (\ref{eq:net-3}), has rank $2n + m -4$.
\label{de:net-3.2}
\end{definition}
Again, according to Lemma~\ref{le:net-2.1} and Lemma~\ref{le:net-2.3}, the linear map $\Lambda$ has kernel whose dimension is at least $4$ and image whose codimension is at least $3$. Therefore, asking that the weighted balanced network is flexible amounts to require that the rank of $\Lambda$ is as large as allowed by these Lemma.

For a balanced weighted network to be flexible, it is necessary that 
\[
m \leq 2n-2.
\]
Indeed, the dimension of the image of ${\rm D} {\bf L}_{\rm Id}$ is necessarily less than $2n-3$ since this mapping has at least a $3$-dimensional kernel (see Lemma~\ref{le:net-2.1}). Moreover, the dimension of the image of ${\rm D} {\bf F}_{({\rm Id}, a)}$ is at most $2n-3$ (see (iii) in Lemma~\ref{le:net-2.2}). Therefore, the rank of $\Lambda$ is at most $4n-6$. We conclude that, in order for the rank of $\Lambda$ to be equal to $2n+m-4$, it is necessary that $m \leq 2n-2$. 

\begin{center}
\includegraphics[width=7cm]{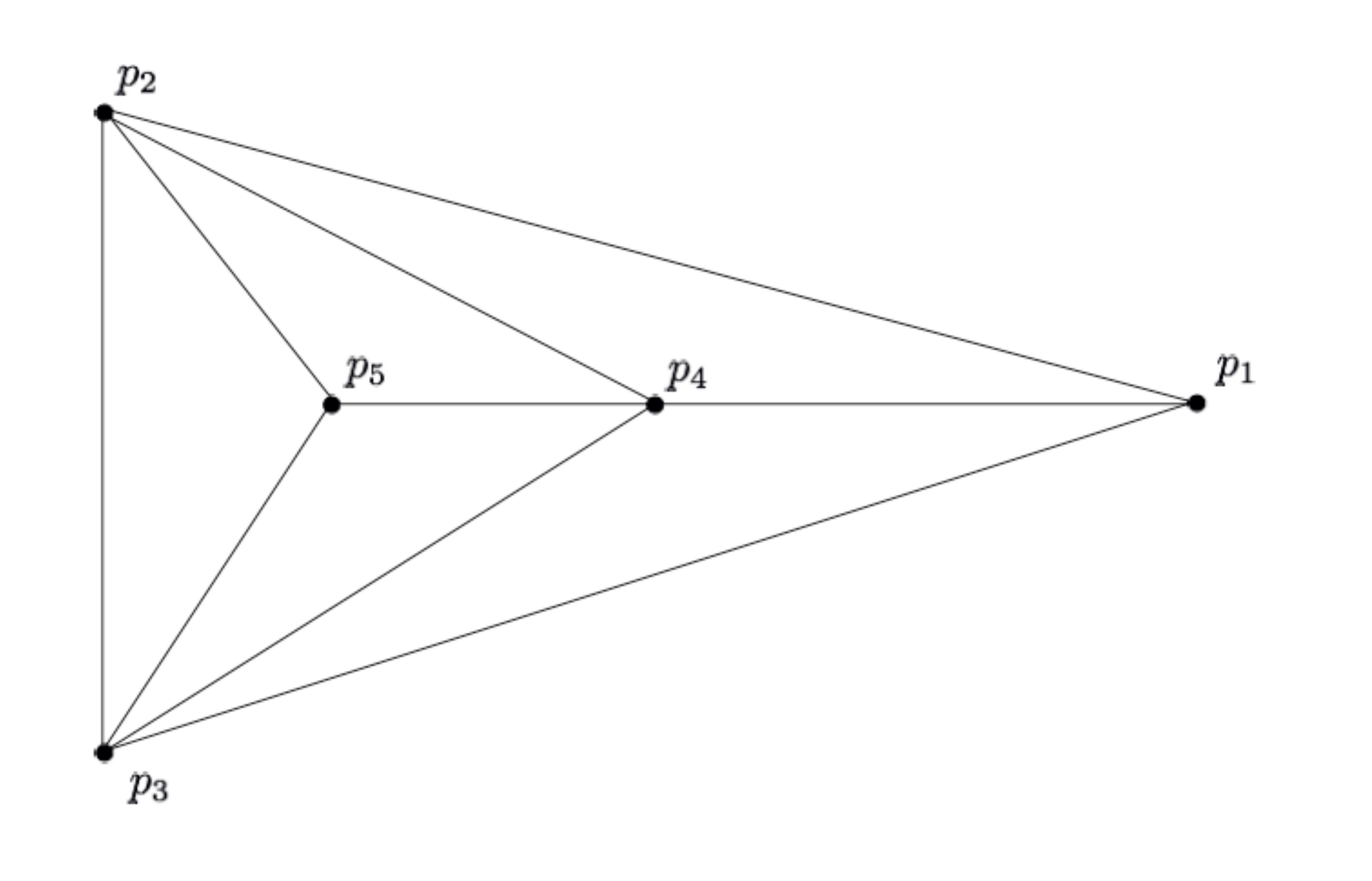}\\
Fig 6 : Example of a network which is not flexible. Here $n=5$ and $m=10$ and hence $m > 2n-2$.
\end{center}

Now, the key observation is that, if we have a balanced, weighted network $(\mathscr N, a)$ which is flexible in the sense of Definition~\ref{de:net-3.2}, then $\Lambda$ has a $4$ dimensional kernel and hence, the image of $\Lambda$ has codimension $4$. But, according to the result of Lemma~\ref{le:net-2.3}, we know that the image of $\Lambda$ is orthogonal to the three vectors which appear in (iii) of Lemma~\ref{le:net-2.3}. So, if the balanced network is flexible, then the image of $\Lambda$ will have codimension $1$ in the orthogonal complement of the space spanned by the vectors $({\bf e})_{p \in \mathscr V}$, for ${\bf e} \in {\bf C}$, and the vector $(i \, p)_{p \in \mathscr V}$.  

In the applications, one of the important cases are the ones where we have a balanced network for which $m= 2n-2$. In this case, we show the~:
\begin{proposition}
Assume that $m=2n-2$. Then, the balanced network $(\mathscr N, a)$ is flexible if and only if ${\rm D}_a {\bf F}_{({\rm Id}, a)}$ (or equivalently ${\rm D} {\bf L}_{\rm Id}$) has rank $m-1$. 
\label{pr:net-3.1}
\end{proposition}
\begin{proof} Simple linear algebra together with the result of Proposition~\ref{pr:net-2.1}.
\end{proof}

Let us insist on the fact that, thanks to Proposition~\ref{pr:net-2.1}, the linear maps ${\rm D}_a {\bf F}_{({\rm Id}, a)}$ and ${\rm D} {\bf L}_{\rm Id}$ have the same rank and hence, in the case where $m=2n-2$, it is enough to check that one of them has the desired rank to check flexibility of the network.

We now describe some interesting flexible balanced networks. Further examples will be given in section 10. 

\medskip

\noindent {\bf Example 3.4 :} 
\label{ex:3.4}
Given $k\geq 3$, we consider the network $\mathring{\mathscr N}_{Pol}$ defined by a regular polygon with $k$ sides, whose vertices are linked to the origin. Hence, the set of vertices of this network is given by
\[
\mathring{\mathscr V}_{Pol} :=\{0\} \cup \{\xi^j \in {\bf C}  \, : \, j=1, \ldots, k\},
\]
where $\xi : = e^{2i\pi/k}$. The set of edges of this network is defined to be 
\[
\mathring{\mathscr E}_{Pol} : = \{ [0, \xi^j] \, : \, j=1, \ldots, k\} \cup \{[\xi^{j+1}, \xi^j] \, : \, j=1, \ldots, k\}
\]
In this example, the number of vertices is $n=k+1$ and the number of edges is $m= 2k$. Hence we have 
\[
m = 2n -2.
\]
If we define $a : \mathring{\mathscr E}_{Pol} \to {\bf R} - \{0\}$ by 
\[
a_{[\xi^{j}, \xi^{j+1}]} =  1, \qquad \mbox{and} \qquad  a_{[0,\xi^j]} = - 2 \, \sin (\pi/k).
\]
we obtain a balanced network $(\mathring{\mathscr N}_{Pol}, a)$. 

\begin{center}
\includegraphics[width=6.5cm]{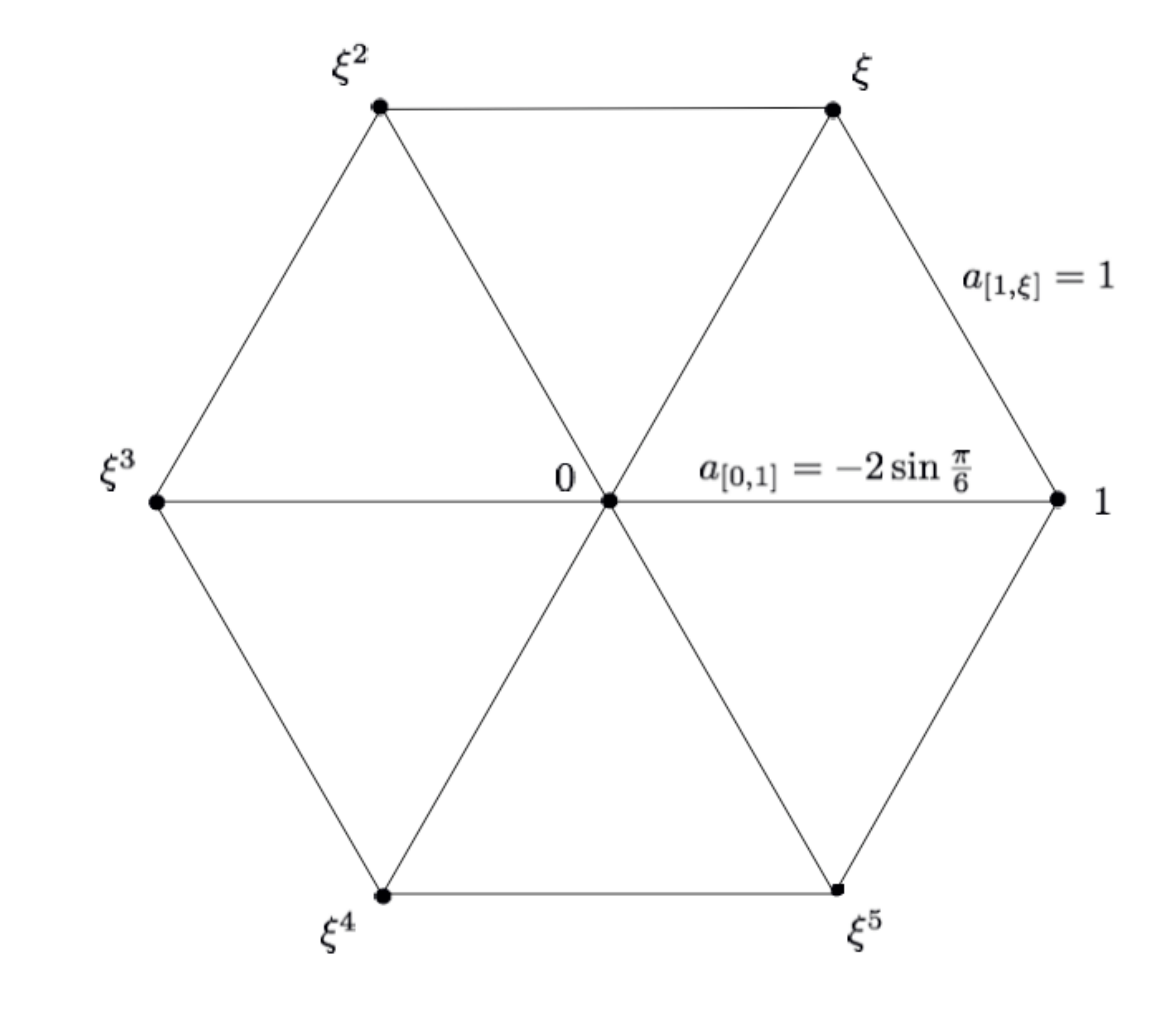}\\
Fig 7 : Example of a network $\mathring{\mathscr N}_{Pol}$. Because of dihedral symmetry the weights along all edges can be determined from the weights along $[0, 1]$ and $[1, \xi]$.
\end{center}

We claim that~:
\begin{lemma}
The map ${\rm D}_{a} {\bf F}_{({\rm Id}, a)}$ has rank $2k-1$. 
\label{le:net-5.1}
\end{lemma}
\begin{proof}
We have
\[
{\rm D}_a {\bf F}_{({\rm Id}, a)} (\dot a) = \left(    \sum_{q \in \mathring{\mathscr V}_p}  \dot a_{[p,q]} \, \frac{q-p}{|q-p|}\right)_{p \in \mathring{\mathscr V}}.
\]

Assume that ${\rm D}_a {\bf F}_{({\rm Id}, a)} (\dot a) = 0$ and also that $\dot a_{[1, \xi ]}=0$. Then, looking at the component of ${\rm D}_a {\bf F}_{({\rm Id}, a)} (\dot a)$ at the vertex $\xi$, we find that 
\[
\dot  a_{[\xi ,0]} \, \xi + \dot  a_{[\xi , \xi^2]} \, \frac{\xi - \xi^2}{|1 - \xi |} = 0,
\]
since the vectors $\xi$ and $\xi - \xi^2$ are not $\bf R$-collinear, we conclude that $\dot  a_{[\xi, 0]} = \dot a_{[ \xi, \xi^2]}=0$. 

To proceed, one looks at the component of ${\rm D}_a {\bf F}_{({\rm Id}, a)} (\dot a)$ at the point $\xi^j$, which gives (after simplification by $\xi^{j-1}$)
\[
\dot a_{[\xi^j,0]} \, \xi + \dot  a_{[\xi^j , \xi^{j-1}]}  \, \frac{\xi - 1}{|1 - \xi|} + \dot  a_{[\xi^j , \xi^{j+1}]} \, \frac{\xi - \xi^2}{|1 - \xi |} = 0.
\]
One then proves by induction that  $\dot  a_{[\xi^j,0]} = \dot a_{[\xi^j,\xi^{j+1}]}=0$ for all $j=1, \ldots, k$, following the arguments given in the case where $j=1$.  

We conclude that ${\rm D}_{a} {\bf F}_{({\rm Id}, a)}$, restricted to the hyperplane $\dot a_{[1, \xi ]}=0$ is injective, and hence we have proven that this map has rank at least $2k-1$. \end{proof}

As a consequence, we have the~:
\begin{corollary}
The balanced network $(\mathring{\mathscr N}_{Pol}, a)$ is flexible in the sense of Definition~\ref{de:net-3.2}.
\label{co:net-5.1}
\end{corollary}
\begin{proof}
According to Proposition~\ref{pr:net-3.1} and Proposition~\ref{pr:net-2.1}, it is enough to prove that ${\rm D}_{a} {\bf F}_{({\rm Id}, a)}$ has rank $2k-1$ and this is just what we have proven in the previous Lemma. \end{proof}

\subsection{Construction of non symmetric balanced networks} All the examples of balanced networks we have seen so far are invariant under the action of a non trivial group of isometries in the plane. More generally, constructing balanced networks can be quite a difficult task since the equation 
\[
{\bf F}_{(\mathscr N, a)} =0,
\]
is highly nonlinear,  specially when one is looking for balanced networks which have no symmetry. Hopefully,  the implicit function theorem comes to the rescue and allows one to deform a given network keeping it balanced. More precisely, we have the~:
\begin{proposition}
Assume that $(\mathscr N, a)$ is a balanced network and further assume that $m = 2n-2$ and that ${\rm D}_a {\bf F}_{({\rm Id}, a)}$ has rank $2n-3$. Then, for all $\Phi :\mathscr V \to {\bf C}$ close enough to ${\rm Id}$, there exists $a_\Phi : \mathscr E \to {\bf R} - \{0\}$ such that the network $(\mathscr N_\Phi, a_\Phi)$ is balanced.
\label{pr:net-6.1}
\end{proposition}
\begin{proof}
By assumption, ${\rm D}_a {\bf F}_{({\rm Id}, a)}$ has rank $2n-3$ and, according to Lemma~\ref{le:net-2.3}, the image of ${\rm D} {\bf F}_{({\rm Id}, a)}$ is orthogonal to $({\bf e} +  i \, t\,  p )_{p\in \mathscr V}$ for all ${\bf e} \in {\bf C}$ and $t \in {\bf R}$. 

We define the mapping
\[
{\bf G} (\Phi, a, {\bf e}, t): =  {\bf F}_{(\Phi, a)} + ({\bf e} + i \, t \, \Phi(p))_{p\in \mathscr V} ,
\]
where ${\bf e} \in {\bf C}$ and $t \in {\bf R}$. By assumption the differential of this mapping with respect to $a$, ${\bf e}$ and $t$, computed at $({\rm Id}, a, 0, 0)$, is onto and hence, the implicit function theorem implies that, for all $\Phi$ close to ${\rm Id}$, there exists $a_\Phi$, ${\bf e}_\Phi$ and $t_\Phi$ such that 
$$
{\bf G} (\Phi, a_{\Phi} , {\bf e}_\Phi , t_\Phi ) =0.
$$ 
In other words
\[
{\bf F}_{(\Phi, a_\Phi)} + ({\bf e}_\Phi + i \, t_\Phi  \, \Phi(p))_{p\in \mathscr V} =0.
\]
In particular, 
\[
\langle {\bf F}_{(\Phi, a_\Phi)} + ({\bf e}_\Phi + i \, t_\Phi \, \Phi(p))_{p\in \mathscr V} , ({\bf e}_\Phi + i \, t_\Phi \, \Phi(p) )_{p\in \mathscr V} \rangle_{{\bf C}^n} =0.
\]
But, using (\ref{eq:net-1}) and (\ref{eq:net-2}), one gets
\[
\langle {\bf F}_{(\Phi, a_\Phi)} , ({\bf e}_\Phi + i \, t_\Phi \, \Phi(p))_{p\in \mathscr V} \rangle_{{\bf C}^n} =0,
\]
hence ${\bf e}_\Phi + i \, t_\Phi  \, \Phi(p)=0$ for all $p \in \mathscr V$. This implies that ${\bf e}_\Phi =0$ and $t_\Phi =0$ and hence ${\bf F}_{(\Phi, a_\Phi)} =0$. This completes the proof of the result.
\end{proof}

This last result, combined with the result of Lemma~\ref{le:net-5.1}, implies that~:
\begin{corollary}
Any network which is close to the network  $(\mathring{\mathscr N}_{Pol}, a)$ defined in Example~\ref{ex:3.4}, can be balanced and gives rise to a flexible balanced network.
\label{co:net-6.1}
\end{corollary}

In particular, there exists balanced networks which are flexible and which have no symmetry. In fact, more is true and, in the spirit of the result of Proposition~\ref{pr:Zd}, we have the~:
\begin{proposition}
The set of flexible balanced networks in a given homotopy class is Zariski open. 
\end{proposition}

We will see in section 10, some explicit networks which are balanced and which have no symmetry. 

\section{Applications}

We now explain how the previous framework can be used to construct approximate solutions to (\ref{eq:nls}). 

\subsection{The interaction function}

We define the {\em interaction function} $\Upsilon$ by
\begin{equation}\label{martin}
\Upsilon (s) := - \iint_{{\bf C}} u_0 (z -s {\bf e}) \, \mbox{\rm div} \, \left( u_0^3 (z) \, {\bf e} \right) \, dx\, dy,
\end{equation}
where ${\bf e} \in {\bf C}$ is {\em any} unit vector (and $z=x+iy$). Since the function $u_0$ is radial, it is easy to check that $\Upsilon (s)$ does not depend on the choice of ${\bf e}$. If the exact formula for $\Upsilon$ is not known, its asymptotic behavior as $s$ tends to infinity is well understood and, for example, we know that there exists a constant $C \in {\bf R}$ such that
\begin{equation}
- \ln \Upsilon (t) = t + \frac{1}{2} \ln t + C + \mathcal O \left( \frac{1}{t}\right) ,
\label{eq:upsilon1}
\end{equation}
at infinity and also that
\begin{equation}
- \frac{\Upsilon (t)}{\Upsilon'(t)} = 1 - \frac{1}{2t}  + \mathcal O \left( \frac{1}{t^2}\right)  ,
\label{eq:upsilon2}
\end{equation}
at infinity.

For all $\ell >0$ and $a\in {\bf R}-\{0\}$, we define, if it exists, $\alpha : = \alpha_\ell (a)  \in {\bf R}$ by the identity
\begin{equation}
\Upsilon \left( \ell (1 - \alpha ) \right)=  |a| \,  \Upsilon (\ell ).
\label{eq:net-4}
\end{equation}
The asymptotic behavior of the function $\Upsilon$ at infinity implies that the function $a \mapsto \alpha_\ell (a)$ is well defined for all $\ell >0$ large enough. Moreover, we have the expansion
\begin{equation}
\alpha_\ell (a)  =  \frac{\ln |a|}{\ell} + \mathcal O \left( \frac{1}{\ell^2}\right)  ,
\label{eq:net-5}
\end{equation}
which holds for all $a$ in a given compact of ${\bf R}-\{0\}$ and for all $\ell >0$ large enough. Finally, differentiating (\ref{eq:net-4}) yields 
\begin{equation}
\begin{array}{rllll}
 \partial_a  \ln (1 - \alpha_\ell )   = - \displaystyle \left( \frac{2\ell-1}{2\ell^2} +\frac{\alpha_\ell}{\ell} + \mathcal O \left( \frac{1}{\ell^3}\right)  \right) \, \frac{1}{a} ,
\end{array}
\label{eq:net-6}
\end{equation}
when $a$ is in a given compact of ${\bf R}-\{0\}$ and for all $\ell >0$ large enough. In both (\ref{eq:net-5}) and (\ref{eq:net-6}), $\mathcal O (\epsilon^k)$ are smooth functions of $a$ and $\epsilon >0$ such that $ \epsilon^{-k} \, \mathcal O (\epsilon^{k})$ extends smoothly at $\epsilon = 0$.

\subsection{Perturbations of unbalanced networks}

We assume here that we have a flexible, unbalanced unitary network $(\mathscr N, a)$. Recall that the fact that the network is unitary just means that the lengths of the edges are all equal to $1$. Everything applies to networks which are not unitary but it turns out that, in applications, only unitary networks are used.  The results of this section will not be directly used in the paper but should be understood as a warm-up. 

As usual, we agree that $n$ denotes the number of vertices and $m$ the number of edges of the network $\mathscr N$. To begin with, let us prove the following result which states that, modifying slightly the vertices and the weights of the network $(\mathscr N, a)$ it is possible to perturb the forces at the vertices of the network and it is also possible to change the lengths of the edges of the network.  More precisely, we have the~:
\begin{proposition}
There exists $\epsilon_* >0$ such that, for all $\alpha : = (\alpha_{[p,q]})_{[p,q] \in \mathscr E} \in {\bf R}^m$, for all ${\bf f}  : = ({\bf f}_p)_{p \in \mathscr V} \in {\bf C}^n$, satisfying 
\[
|{\bf f}_p| + |\alpha_{[p,q]}| \leq \epsilon_*,
\] 
there exists $\Phi : \mathscr V \to {\bf C}$, $\tilde a : \mathscr E \to {\bf R}-\{0\}$ and ${\bf e} \in {\bf C}$,  all depending smoothly on the ${\bf f}_p$ and the $\alpha_{[p,q]}$, such that
\[
\left\{
\begin{array}{rllll}
{\bf F}_{(\Phi, \tilde a)} (p) & = & {\bf F}_{({\rm Id}, a)} (p)  + {\bf f}_p + {\bf e}, &\qquad \text{for all} \quad p \in \mathscr V ,\\[3mm]
{\bf L}_\Phi ([p,q]) & = & 1 - \alpha_{[p,q]} ,  &  \qquad  \text{for all} \quad  [p,q] \in \mathscr E, \\[3mm]
\displaystyle \sum_{p \in \mathscr V}  (\Phi_p - p) & = & 0 . & 
\end{array}
\right.
\]
Moreover, $\Phi = {\rm Id}$ and $\tilde a =a$ when the ${\bf f}_p =0$ and the $\alpha_{[p,q]}=0$.
\label{pr:net-7.1}
\end{proposition}
\begin{proof}
We define the mapping
\[
\begin{array}{lllll}
\displaystyle {\bf G} \left(\Phi, \tilde a , {\bf e} \, ; {\bf f} , \alpha \right)  : =  \displaystyle \left(  {\bf F}_{ (\Phi, \tilde a)}  - {\bf F}_{ ({\rm Id}, a)}  - \left({\bf e} \right)_{p \in \mathscr V} - {\bf f} \, ;     {\bf L}_\Phi -  \left( 1 \right)_{[p,q] \in \mathscr E}  - \alpha \right), 
\end{array}
\]
where $\Phi : {\mathscr V} \to {\bf C}$, $\tilde a : \mathscr E \to {\bf R}-\{0\}$ and ${\bf e} \in {\bf C}$. Certainly, 
\[
\displaystyle {\bf G} \left( {\rm Id} , a , 0 \,  ; 0, 0\right) = 0.
\]
The fact that the network is flexible in the sense of Definition~\ref{de:net-2.1} implies that 
\[
\Lambda (\dot \Phi, \dot a ) : = \left(  {\rm D} {\bf F}_{({\rm Id}, a)} (\dot \Phi , \dot a)  ;  {\rm D} {\bf L}_{{\rm Id}} (\dot \Phi) \right), 
\]
the differential of ${\bf G}$ with respect to $\Phi$ and $\tilde a$, computed at $\Phi ={\rm Id}$ and $a=\tilde a$,  has rank $2n+m-2$ and this, together with (ii) in Lemma~\ref{le:net-2.2} implies that 
\[
\Lambda^\flat (\dot \Phi, \dot a , \dot {\bf e} ) : = \left(  {\rm D} {\bf F}_{({\rm Id}, a)} (\dot \Phi , \dot a) + (\dot {\bf e})_{p \in \mathscr V} ;  {\rm D} {\bf L}_{{\rm Id}} (\dot \Phi) \right), 
\]
the differential of ${\bf G}$ with respect to $\Phi$, $\tilde a$ and ${\bf e}$, computed at $\Phi ={\rm Id}$, $\tilde a=a$ and ${\bf e} =0$, is onto and has kernel of dimension $2$ spanned by the $( ({\bf e})_{p\in \mathscr V}, 0, 0) \in {\bf C}^n \times {\bf R}^m \times {\bf C}$, for all ${\bf e} \in {\bf C}$. In particular, when trying to solve 
$$ 
{\bf G} \left(\Phi, \tilde a , {\bf e} \, ; {\bf f} , \alpha \right) =0,
$$ 
it is enough to restrict our attention to space of mappings $\Phi$ such that 
\[
\sum_{p \in \mathscr V} (\Phi_p - p ) = 0 .
\]
since $\Lambda^\flat$ is an isomorphism from the space of $(\dot \Phi, \dot a, \dot {\bf e}) \in {\bf C}^n \times {\bf R}^m \times {\bf C}$ such that 
\[
\sum_{p \in \mathscr V} \dot \Phi_p = 0,
\]
into ${\bf C}^n \times {\bf R}^m \times {\bf C}$. The application of the implicit function theorem implies that there exists $\Phi : \mathscr V \to {\bf C}$, $\tilde a : \mathscr E \to {\bf R}-\{0\}$ and ${\bf e} \in {\bf C}$,  all depending smoothly on the ${\bf f}_p$ and the $\alpha_{[p,q]}$, such that $\Phi = {\rm Id}$ and $\tilde a =a$ when the ${\bf f}_p =0$ and the $\alpha_{[p,q]}=0$,
\[
{\bf F}_{(\Phi, \tilde a)} (p) = {\bf F}_{({\rm Id}, a)} (p)  + {\bf e} +  {\bf f}_p,
\]
for all $p \in \mathscr V$, and
\[
{\bf L}_\Phi ([p,q]) = 1 - \alpha_{[p,q]} , 
\]
for all $[p,q] \in \mathscr E$.
\end{proof}

In applications, it turns out that the $\alpha_{[p,q]}$ are parameters which are not independent of the other parameters but rather depend on the $a_{[p,q]}$. More precisely, in applications, we have
\[
\alpha_{[p,q]} : =  \alpha_\ell (a_{[p,q]})  ,
\]
where $a \mapsto \alpha_\ell (a)$ is the function defined in (\ref{eq:net-4}) and hence the $\alpha_{[p,q]}$ are functions of the $a_{[p,q]}$ (and of the parameter $\ell >0$). A straightforward modification of the proof of the previous result, yields~:
\begin{proposition}
There exists $\ell_* >0$ and $\epsilon_* >0$ such that, for all $\ell\geq \ell_*$ and for all ${\bf f} : = ({\bf f}_p)_{p \in \mathscr V} \in {\bf C}^n$, such that 
\[
 |{\bf f}_p|\leq \epsilon_*,
\] 
there exists $\Phi : \mathscr V \to {\bf C}$, $\tilde a : \mathscr E \to {\bf R}-\{0\}$ and ${\bf e} \in {\bf C}$, all depending smoothly on the ${\bf f}_p$ and the $\alpha_{[p,q]}$, such that
\[
\left\{
\begin{array}{rllll}
{\bf F}_{ (\Phi, \tilde a)} (p) & = & {\bf F}_{({\rm Id}, a)} (p)  + {\bf f}_p + {\bf e}, &\qquad \text{for all} \quad p \in \mathscr V ,\\[3mm]
{\bf L}_\Phi ( [p,q]) & = & 1 - \alpha_\ell (\tilde a_{[p,q]}) ,  &  \qquad  \text{for all} \quad  [p,q] \in \mathscr E, \\[3mm]
\displaystyle \sum_{p \in \mathscr V}  (\Phi_p -p) & = & 0 . & 
\end{array}
\right.
\]
Moreover, $\Phi = {\rm Id}$ and $\tilde a =a$ when $\ell =+\infty$ and when all the ${\bf f}_p =0$, for $p \in \mathscr V$.
\label{pr:net-7.111}
\end{proposition}

As explained above, we will not directly make use of this result in this paper. However, this result can, for example, be used to generalize the examples of solutions of (\ref{eq:nls}) which infinite energy constructed by Malchiodi in \cite{M} or it can also be used to construct complete non compact constant mean curvature surfaces in the spirit of \cite{K, K3}.

\subsection{Perturbation of balanced networks} We assume now that we have a flexible, balanced network $(\mathscr N, a)$. As usual, $n$ denotes the number of vertices and $m$ the number of edges of the network $\mathscr N$. Again, as a warm up, we would like to modify slightly the vertices and the weights of the weighted network $(\mathscr N, a)$, in such a way that, as above, the forces at the vertices of the perturbed network are prescribed (small vectors) and also we would like to slightly alter the size of the edges of the network in such a way that, the length of each edge $[p,q]$ of the perturbed network, dilated by a factor $\kappa \gg 1$, is an integer multiple of $1- \alpha_{[p,q]}$. 

More precisely, we assume that we are given 
\[
\alpha : =  (\alpha_{[p,q]})_{[p,q] \in {\mathscr E}} \in {\bf R}^m,
\]
small enough, 
\[
{\bf f} : =  ({\bf f}_p)_{p \in \mathscr V},
\]
small enough and $\kappa \gg 1$. For all $[p,q] \in \mathscr E$, we define the integer $m_{[p,q]} \in {\bf N}$  by
\[
\kappa \, \frac{|p-q|}{1- \alpha_{[p,q]}} \leq   2 \, m_{[p,q]}  <  \kappa \, \frac{|p-q|}{1- \alpha_{[p,q]}} +Ê2.
\]

We would like the perturbed network close to $(\mathscr N, a)$ to satisfy the following properties~:

\begin{itemize}
\item[(i)] the forces at the vertices of the perturbed network are given by
\[
{\bf F}_{(\Phi, \tilde a)} (p)  = {\bf f}_p ;
\]

\item[(ii)] the lengths of the edges of the perturbed network satisfy 
\[
\kappa \, {\bf L}_\Phi  ([p,q]) =  2 \, m_{[p,q]} \,  (1 - \alpha_{[p,q]} ) .
\]
\end{itemize}
In other words, the forces at the vertices are prescribed and the lengths of the edges of the original network, which is dilated by $\kappa$, are integer multiple of a prescribed quantity close to $1$. 

As in the previous section, we start with the definition of a nonlinear map
\[
\begin{array}{rlllll}
\displaystyle \mathring{\bf G} \left(\Phi, \tilde a , {\bf e} , t \, ; {\bf f} , \alpha \right)  & : = &  \displaystyle \Big( {\bf F}_{(\Phi, \tilde a)} - \left({\bf e} + t \, i \, \Phi_p\right)_{p \in \mathscr V} - {\bf f} \, ; \\[3mm]
&  &  \displaystyle \hspace{30mm} {\bf L}_\Phi -  \left(  \frac{2 m_{[p,q]}}{\kappa} \left( 1- \alpha_{[p,q]}  \right)\right)_{[p,q] \in \mathscr E} \Big), 
\end{array}
\]
where $\Phi : {\mathscr V} \to {\bf C}$, $\tilde a : \mathscr E \to {\bf R}-\{0\}$, ${\bf e} \in {\bf C}$ and $t \in {\bf R}$. This time $\displaystyle \mathring {\bf G} \left( {\rm Id} , a , 0 ,0 \, ; 0 , 0\right)$ is not equal to $0$ but is close to $0$ (at least when $\kappa$ is large). Indeed, by definition of $m_{[p,q]}$, we have 
\[
\left|  |p-q| -  \frac{2m_{[p,q]}}{\kappa} \left( 1- \alpha_{[p,q]}  \right) \right| \leq \frac{2}{\kappa},
\]
which is small since we assume that $\kappa \gg 1$. As in the previous section we would like to apply some implicit function theorem or more likely  some fixed points argument for contraction mappings, to solve 
\[
\mathring{\bf G} \left(\Phi, \tilde a , {\bf e} , t \, ; {\bf f} , \alpha \right)  =0,
\]
for all $\bf f$ and $\alpha$ small enough and for all $\kappa$ large enough. Unfortunately, this time, the situation is more complicated since the flexibility of the network $(\mathscr N, a)$ implies that the linear map
\[
\Lambda (\dot \Phi, \dot a  ) : = \left( {\rm D} {\bf F}_{({\rm Id}, a)} (\dot \Phi , \dot a)  \, ;  {\rm D} {\bf L}_{{\rm Id}} (\dot \Phi )\right), 
\]
has rank $2n+m-4$ and it also implies that the linear map
\[
\Lambda^\sharp (\dot \Phi, \dot a , \dot {\bf e}, \dot t ) : = \left( {\rm D} {\bf F}_{({\rm Id}, a)} (\dot \Phi , \dot a) - \left(\dot {\bf e} + \dot t \, i \, p \right)_{p \in \mathscr V} \, ;  {\rm D} {\bf L}_{{\rm Id}} (\dot \Phi )\right), 
\]
which is the differential of $\mathring {\bf G}$ with respect to $\Phi$, $\tilde a$, ${\bf e}$ and $t$, computed at $\Phi ={\rm Id}$, $a=\tilde a$, ${\bf e} =0$ and $t=0$, has rank $2n+m-1$. In particular, $\Lambda^\sharp$ is not onto and this prevents us from applying any fixed point theorem for contraction mappings to solve the above equation. 

In some sense, the fact that $\Lambda^\sharp$ has rank $2n+m-1$ can be interpreted by saying that, by perturbing the weighted network  $({\mathscr N}, a)$ we can ensure that ${\bf f}_p$ is indeed the force at the vertex $\Phi_p$ and we can also ensure that the lengths of the edges of the perturbed network are exactly what we want them to be, except for one of them. Hence, we are missing one extra degree of freedom to ensure that all the lengths of the perturbed network are what we want them to be.
 
The problem seems to be hopeless since we have exhausted all possible parameters to perturb the weighted network. Surprisingly, the solution comes from the fact that, in applications, the parameters $\alpha_{[p,q]}$ are not arbitrary but are functions of the weights $a_{[p,q]}$. Moreover, dilation of the weight is in the kernel of $\Lambda$. These two facts combine and turn out to be the key to our problem. 

To explain this further, we need to introduce the notion of {\em closable network}.  Given a weighted network $(\mathscr N, a)$, we define
\begin{equation}
{\bf T} : = \left( |p-q|  \, \ln |a_{[p,q]}| \right)_{[p,q] \in \mathscr E} \in {\bf R}^m.
\label{eq:TT}
\end{equation}
We have the~:
\begin{definition}
A flexible, balanced network $(\mathscr N, a)$ is said to be {\em closable} if
\[
\mathring \Lambda  (\dot \Phi, \dot a , \dot s ) : = \left( {\rm D} {\bf F}_{({\rm Id}, a)} (\dot \Phi , \dot a)  \, ;  {\rm D} {\bf L}_{{\rm Id}} (\dot \Phi ) + \dot s \, {\bf T}\right), 
\]
has rank $2n+m-3$.
\label{de:clos}
\end{definition}
Observe that the notion of closable network, just like the notion of flexibility, only depends on the network $\mathscr N$ and on the weight function $a$. 

\begin{remark}
While the definition of a closable network is independent of the problem we are looking at, the definition of ${\bf T}$ depends on the problem we are studying. For example, in the study of constant mean curvature surfaces in Euclidean $3$ space, the definition of ${\bf T}$ would rather be ${\bf T} = \left( |p-q|  \, a_{[p,q]} \right)_{[p,q] \in \mathscr E}$ instead of (\ref{eq:TT}).
\end{remark}

\medskip

\noindent {\bf Example 4.1 :} 
\label{ex:4.1}
Given $k\geq 3$, we have already considered the network $\mathring{\mathscr N}_{RegPol}$ defined by a regular polygon with $k$ sides, whose vertices are linked to the origin. Its set of vertices is given by
\[
\mathring{\mathscr V}_{RegPol} :=\{0\} \cup \{\xi^j \in {\bf C}  \, : \, j=1, \ldots, k\}.
\]
where $\xi : = e^{2i\pi/k}$.  Let us now check that this network is also closable in the sense of Definition~\ref{de:clos} provided $k\neq 6$. We need to check that ${\bf T}$ is not in the image of ${\rm D} {\bf L}_{\rm Id}$.  Therefore, we need to check that there does not exist $\dot \Phi \neq 0$ such that 
\[
\langle \xi^{j+1} - \xi^j , \dot \Phi_{\xi^{j+1}} - \dot \Phi_{\xi^{j}}\rangle_{\bf C} = 0 
\]
for $j=0, \ldots, k-1$ and 
\[
\langle \xi^{j} , \dot \Phi_{\xi^{j}} - \dot \Phi_{0}\rangle_{\bf C} =  \ln |2 \sin (\pi/k)|. 
\]
Observe that $(|p-q|)_{[p,q] \in \mathring{\mathscr E}_{RegPol}}$ is always in the image of ${\rm D} {\bf L}_{\rm Id}$ since it is the image of $\dot \Phi$ defined by $\dot \Phi_p =p$ for all $p \in \mathring{\mathscr V}_{RegPol}$.  Therefore, by linearity, it is enough to check that there does not exist $\dot \Phi$ such that 
\[
\langle \xi^{j+1} - \xi^j , \dot \Phi_{\xi^{j+1}} - \dot \Phi_{\xi^{j}}\rangle_{\bf C} =  1,
\]
for $j=0, \ldots, k-1$ and 
\[
\langle \xi^{j} , \dot \Phi_{\xi^{j}} - \dot \Phi_{0}\rangle_{\bf C} =  0. 
\]
Observe that we have implicitly used the fact that $\ln |2 \sin (\pi/k)| \neq 0$ for $k\neq 6$. Now, the second equation implies that 
\[
\dot \Phi_{\xi^{j}} - \dot \Phi_{0} = i \,  x_j \, \xi^j ,
\]
for some $x_j \in {\bf R}$. Inserting this information in the first set of equations, we conclude that 
\[
\sin (\pi/k) \, \left( x_{j+1} - x_j \right) = 1.
\]
Summing these equalities from $j=0$ to $j=k-1$ and remembering that we identify $x_{k}=x_0$, we reach a contradiction. Therefore, the network is closable and so are all nearby networks. 

\begin{remark}
The above example corresponds to the construction in \cite{MPW} where the condition $k\neq 6$ also appears in to be a necessary condition for the construction to be successful. Indeed, at the end of section 5 in \cite{MPW}, one needs $D_0$, the determinant of some $2$ by $2$ system, not to be too close to zero. It is also shown that the leading order in the expression of $D_0$ is equivalent to $\ln | 2 \sin (\pi/k)| \neq 0$ and hence one concludes that  $D_0$ is not too close to $0$ precisely when $k\neq 6$. 
\end{remark}

To see how the notion of closable network enters in our analysis, let us recall that, in applications, the parameters $\alpha_{[p,q]}$ are not independent of the other parameters but rather depend on the $a_{[p,q]}$, namely
\[
\alpha_{[p,q]} = \alpha_\ell (a_{[p,q]}) ,
\]
where $a \to \alpha_\ell (a)$ is the function defined in (\ref{eq:net-4}). This time, given $\kappa >0$ very large, for all $[p,q] \in {\mathscr E}$, we define $m_{[p,q]} \in {\bf N}$ by
\[
\kappa \, \frac{| q -p |}{ 1 - \alpha_\ell (a_{[p,q]}) } \leq  2 m_{[p,q]}  < \kappa \, \frac{| q -p |}{ 1 - \alpha_\ell (a_{[p,q]}) } + 2.
\]

We consider the nonlinear map 
\[
\begin{array}{rlllll}
\displaystyle \hat {\bf G} \left(\Phi, \tilde a , {\bf e} , t \, ; {\bf f}  \right)  & : = &  \displaystyle \Big( {\bf F}_{(\Phi, \tilde a)} - \left({\bf e} + t \, i \, \Phi_p\right)_{p \in \mathscr V} - {\bf f} \, ; \\[3mm]
&  & \hspace{30mm} \displaystyle {\bf L}_\Phi -  \left(  \frac{2m_{[p,q]}}{\kappa} \left( 1- \alpha_\ell (a_{[p,q]}) \right)\right)_{[p,q] \in \mathscr E}  \Big), 
\end{array}
\]
where $\Phi : {\mathscr V} \to {\bf C}$, $\tilde a : \mathscr E \to {\bf R}-\{0\}$, ${\bf e} \in {\bf C}$ and $t \in {\bf R}$.  Observe that $\mathring G$ which was introduced above and $\hat G$ are related by the identity 
\[
\hat {\bf G} \left(\Phi, \tilde a , {\bf e} , t \, ; {\bf f}  \right) = \mathring  {\bf G} \left(\Phi, \tilde a , {\bf e} , t \, ; {\bf f} , \alpha \right)
\]
where, on the right hand side, $\alpha_{[p,q]} = \alpha_\ell (a_{[p,q]})$.

Again, $\hat G ({\rm Id}, a, 0,0 \, ; 0)$ is not equal to $0$ but it is close to $0$ when $\kappa$ is large and we would like to apply some fixed point theorem for construction mappings to solve $\hat {\bf G} \left(\Phi, \tilde a , {\bf e} , t \, ; {\bf f}  \right)  =0$ at least when ${\bf f}$ is small and when $\kappa$ and $\ell$ are large enough.

The differential of $\hat {\bf G}$ with respect to $\Phi$, $\tilde a$, ${\bf e}$ and $t$, computed at $\Phi ={\rm Id}$, $a=\tilde a$, ${\bf e} =0$ and $t=0$, is given by the formula
\[
\hat \Lambda (\dot \Phi, \dot a , \dot {\bf e}, \dot t ) : = \left( {\rm D} {\bf F}_{({\rm Id}, a)} (\dot \Phi , \dot a) - \left(\dot {\bf e} + \dot t \, i \, p \right)_{p \in \mathscr V} \, ;  {\rm D} {\bf L}_{{\rm Id}} (\dot \Phi )  + {\bf S} (\dot a) \right)  , 
\]
where 
\[
{\bf S} (\dot a) : = \left( S_{[p,q]} \, \dot a_{[p,q]} \right)_{[p,q] \in \mathscr E} ,
\]
and where 
\[
S_{[p,q]} : =  \frac{2m_{[p,q]}}{\kappa} \partial_a \alpha_\ell (a_{[p,q]}) .
\]

It follows from the definition of $m_{[p,q]}$ that 
\[
\frac{2m_{[p,q]}}{\kappa}  = \frac{|p-q|}{1- \alpha_\ell (a_{[p,q]})} + \mathcal O\left( \frac{1}{\kappa}\right) ,
\]
and, using the expansion of $\partial_a \alpha_\ell$ given in (\ref{eq:net-6}) together with (\ref{eq:net-5}), we conclude that 
\[
S_{[p,q]} = -  |p-q| \, \left( \frac{2\ell-1}{2\ell^2} + \frac{\ln |a_{[p,q]}|}{\ell^2} + \mathcal O \left( \frac{1}{\ell^3}\right) + \mathcal O \left( \frac{1}{\kappa \ell}\right) \right)\, \frac{1}{a_{[p,q]}}.
\]

Now, it is convenient to decompose 
\begin{equation}
\dot a_{[p,q]} = \dot a^\perp_{[p,q]} - \dot c \, \ell^2 \, a_{[p,q]} ,
\label{decomp-1}
\end{equation}
where $\dot c \in {\bf R}$ and where $\dot a^\perp$ and $a$ are orthogonal. Similarly, we decompose
\begin{equation}
\dot \Phi_p =  \dot \Phi_p^\perp + \left(\dot d - \frac{2\ell-1}{2} \, \dot c\right)  \, p ,
\label{decomp-2}
\end{equation}
where $\dot d \in {\bf R}$ and  where $\dot \Phi^\perp$ and $(p)_{p \in \mathscr V}$ are orthogonal. With these decompositions at hand, we have 
\[
 {\rm D} {\bf F}_{({\rm Id}, a)} (\dot \Phi , \dot a)  =  {\rm D} {\bf F}_{({\rm Id}, a)} (\dot \Phi^\perp , \dot a^\perp) ,
\]
since $a$ is in the kernel of $D_a{\bf F}_{(\rm Id, a)}$ and $(p)_{p \in {\mathscr V}}$ is in the kernel of $D_\Phi {\bf F}_{(Id, a)}$,  while 
\[
{\rm D} {\bf L}_{{\rm Id}} (\dot \Phi) + {\bf S} (\dot a )  = {\rm D} {\bf L}_{{\rm Id}} (\dot \Phi^\perp) + {\bf S} (\dot a^\perp )  + \dot d \, {\bf L}_{\rm Id} + \dot c \, {\bf T} + \mathcal O \left( \frac{|\dot c|}{\ell} \right) + \mathcal O \left( \frac{\ell \, |\dot c|}{\kappa } \right),
\]
where the vector ${\bf T}$ is the one defined in (\ref{eq:TT}).  Now, by assumption, the mapping
\[
\begin{array}{rllll}
\Lambda_0 (\dot \phi^\perp, \dot a^\perp, \dot c , \dot d, \dot {\bf e}, \dot t) & : = & \Big( {\rm D} {\bf F}_{({\rm Id}, a)} (\dot \Phi^\perp , \dot a^\perp)  - \left(\dot {\bf e} + \dot t \, i \, p \right)_{p \in \mathscr V}  \, ;   \\[3mm]
&& \hspace{30mm} {\rm D} {\bf L}_{{\rm Id}} (\dot \Phi^\perp) + {\bf S} (\dot a^\perp )  + \dot d \, {\bf L}_{\rm Id} + \dot c \, {\bf T} \Big),
\end{array}
\]
has full rank and we are in position to apply some fixed point theorem to solve $\hat {\bf G} \left(\Phi, \tilde a , {\bf e} , t \, ; {\bf f}  \right)  =0$.  This leads to the~:

\begin{proposition}
There exists $\ell_* >0$, $\kappa_* >0$ and $\epsilon_* >0$ such that, for all $\ell\geq \ell_*$, for all $\kappa \geq \kappa_* \, \ell^3$ and for all $({\bf f}_p)_{p \in \mathscr V} \in {\bf C}^n$, such that 
\[
\ell^3 \, |{\bf f}_p|\leq  \epsilon_*,
\] 
there exists $\Phi : \mathscr V \to {\bf C}$, $\tilde a : \mathscr E \to {\bf R}-\{0\}$, ${\bf e} \in {\bf C}$ and $t \in {\bf R}$ all  depending smoothly on the ${\bf f}_p$ such that
\[
\left\{ 
\begin{array}{rlllll}
{\bf F}_{(\Phi, \tilde a)} (p) & = &  {\bf f}_p + {\bf e}Ê+  i \, t \, \Phi_p, &\qquad \text{for all} \quad p \in \mathscr V,\\[3mm]
\kappa \, {\bf L}_{\Phi} ( [p,q]) & =  & 2 \, m_{[p,q]} \, \left(1 - \alpha_\ell  (\tilde a_{[p,q]}) \right), &\qquad \text{for all} \quad [p,q] \in \mathscr E, \\[3mm]
\displaystyle \sum_{p \in \mathscr V} (\Phi_p  -p )&  = & 0 
\qquad \text{and} \qquad 
\displaystyle \sum_{p \in \mathscr V} (\Phi_p - p) \wedge p =0 .
\end{array}
\right.
\]
Moreover
\[
\ell \, \sup_{p \in \mathscr V} |\Phi_p-p| + \sup_{[p,q] \in \mathscr E}| \tilde a_{[p,q]} -a_{[p,q]} | \leq C \, \ell^2 \, \left( \sup_{p \in \mathscr V} |{\bf f}_p| + \frac{1}{\kappa} \right),
\]
for some constant $C >0$. 
\label{pr:net-7.44}
\end{proposition}
\begin{proof}
This time, we apply a fixed point theorem for contraction mapping. The proof does not offer any difficulty but we shall nevertheless comment on a couple of issues. 

Since we start with a balanced network, the kernel of the linearized map $\hat \Lambda$ contains $({\bf e})_{pÊ\in \mathscr V}$ and $(i \, p)_{pÊ\in \mathscr V}$, therefore, it is enough to restrict our attention to the space of mappings $\Phi$ satisfying 
\[
\sum_{p \in \mathscr V} (\Phi_p  -p ) =0 \qquad \mbox{and} \qquad \sum_{p \in \mathscr V} (\Phi_p - p) \wedge p =0 ,
\]
since $\hat \Lambda$ is an isomorphism from the space of $(\dot \Phi, \dot a, \dot {\bf e})$ such that 
\[
\sum_{p \in \mathscr V} \dot \Phi_p =0 \qquad \mbox{and} \qquad \sum_{p \in \mathscr V} \dot \Phi_p \wedge p =0 .
\]
In geometric terms, this amounts to require that we do not translate or rotate the initial configuration before we prescribe the small forces.  

In agreement with the decomposition of $\dot a$ and $\dot \Phi$ given in (\ref{decomp-1}) and (\ref{decomp-2}), we write 
\[
\Phi_p  = \left( 1+ d - \frac{2\ell-1}{2} \, c\right)  \, p + \Phi^\perp_p, 
\]
where $\phi^\perp$ and $(p)_{p\in \mathscr V}$ are orthogonal and $c,d \in {\bf R}$, and
\[
\tilde a_{[p,q]} = (1- c \, \ell^2) \, a_{[p,q]} + a^\perp_{[p,q]}, 
\]
where $a$ and $a^\perp$ are orthogonal. Hence, the unknowns are now $\Phi^\perp$, $a^\perp$, $c$ and $d$. 

We now apply a fixed point theorem for contraction mappings, to obtain a solution of  $\hat {\bf G} \left(\Phi, \tilde a , {\bf e} , t \, ; {\bf f}  \right)  =0$. Inspection of the nonlinearities shows that, in order to obtain a contraction mapping, we need to assume that $\ell^3  \ll \kappa $ and $\ell^3 \, |{\bf f}|\ll 1$, and then we obtain a solution which satisfies
\[
\ell \, |\Phi_p-p| + |\tilde a-a| \leq C \, \ell^2 \, \left( |{\bf f}| + \frac{1}{\kappa} \right),
\]
for some $C >0$.  
\end{proof}

Some important comment is due on the parameters of the construction which are free {\em continuous} parameters. At first glance it might appear that $\ell$ and $\kappa$ are continuous parameters which are free to be specified close to a given value and hence, Proposition~\ref{pr:net-7.44} provides a $2$-dimensional (smooth) family of solutions. This is not the case and to explain this one needs to go back to the decomposition of $\dot \Phi$ and $\dot a$. Indeed, infinitesimal modification of $\kappa$ amounts to apply some dilation (with factor close to $1$) to the set of vertices of the network and close inspection of the expression of $\dot \Phi$ given in (\ref{decomp-2}) shows that we need to allow dilations of the set of points in the fixed point argument. Therefore, in some sense, it is not possible to consider $\kappa$ as a free continuous parameter since a slight change of $\kappa$ will be counterbalanced by a dilation of the vertices of the network. Similarly, to understand why $\ell$ is not a free continuous parameter, we refer to (\ref{decomp-1}) where one can see that an infinitesimal change in the value of $\ell$ will be counterbalanced by a dilation of the weight function. Therefore, the parameters $\ell$ and $\kappa$ are somehow {\em quantized} by the choices of the $m_{[p,q]}$.

In principle, thanks to the above result, we should be close to the end of the construction of the set of points $Z^+$ and $Z^-$ which are mentioned in section 2. 

Indeed, we can now dilate the network $\mathring {\mathscr N}$ by $\kappa \, \ell$ and, since the perturbed network is constructed in such a way that
\[
\kappa \, | \Phi_p - \Phi_q|  = 2m_{[p,q]} \,  \left( 1 - \alpha_\ell (\tilde a_{[p,q]} ) \right),
\]
we can insert exactly  $2m_{[p,q]} -1$ points between $\kappa \, \ell \, \Phi_p$ and $\kappa \, \ell \, \Phi_q$, in such a way that the distances between two consecutive points are exactly equal to $\ell \left( 1 - \alpha_\ell (\tilde a_{[p,q]} ) \right)$. In the case where $\tilde a_{[p,q]} >0$ we decide that these points, together with the end points $\kappa \, \ell \, \Phi_p$ and $\kappa \, \ell \, \Phi_q$ will be points where we center copy of $+ u_0$ and hence these points will belong to $Z^+$. While, if $\tilde a_{[p,q]} < 0$ we decide to put copies of $\pm u_0$ with alternative signs at these points. More precisely, we can label the points we evenly distribute along the edge $\kappa \, \ell \, [\Phi_p, \Phi_q]$ as
\[
z_j^{[p,q]} : = \kappa \, \ell \, \Phi_p + j \,  \ell \, \left( 1 - \alpha_\ell (a_{[p,q]} ) \right) \, \frac{\Phi_q-\Phi_p}{|\Phi_q-\Phi_p|},
\]
for $j=0, \ldots, 2m_{[p,q]}$ (observe that  $ z_{2m_{[p,q]}}^{[p,q]}  : = \kappa \, \ell \, \Phi_q$). Then, we decide to  put copies of $(-1)^j u_0$ centered  at the points $z^{[p,q]}_{j}$, for $j=0, \ldots, 2m_{[p,q]}$.

This is not the end of the story since there is yet another issue we need to take care of. It should be clear that $\kappa \, \ell \, \Phi_p$ has as many closest neighbors, in the sense defined in section $2$ as the number of vertices meeting at $\Phi_p$. In fact, the set of closest neighbors of $\kappa \, \ell \, \Phi_p$ is explicitly given by
\[
N_{\kappa \, \ell \, \Phi_p} = \left\{ \kappa \, \ell \, \Phi_p +  \ell \, \left( 1 - \alpha_\ell (\tilde a_{[p,q]} ) \right) \, \frac{\Phi_q-\Phi_p}{|\Phi_q-\Phi_p|}\, : \,  qÊ\in \mathscr V_p\right\}
\]

We would like to guarantee that the points we evenly distribute along the edges $[\Phi_p, \Phi_q]$, dilated by $\kappa \, \ell$, have exactly $2$ closest neighbors, in the sense described in section 2. Namely, we would like to guarantee that, for $j=1, \ldots , 2m_{[p,q]}-1$,  the only closest neighbors of $z_j^{[p,q]}$ are $z_{j-1}^{[p,q]}$ and $z_{j+1}^{[p,q]}$. It is easy to see that this is only possible if the angles between two different edges meeting at the same vertex is larger than $\pi/3$.  

Unfortunately, given a balanced network, it never happens that all the angles between edges meeting at a common vertex are larger than $\pi/3$, for all the vertices. This is the reason why we need to alter the previous construction by replacing vertices of the network by more complicated structures which turn out to be unbalanced networks. We explain this extra construction in the next section. 

\section{Construction of approximate solutions}

\subsection{Networks and sub-networks} \label{se:dsds} Assume that we are given a closable, flexible balanced network $({\mathscr N} , a)$ and  two parameters $\kappa , \ell \gg 1$. 

For each $p \in {\mathscr V}$, we assume that we are given either a flexible, unitary network $({\mathscr N}^{p}, a^{p})$ or we define a ${\mathscr N}^{p}$ to be the network reduced to $\{0\}$ (in which case we agree that the set of edges is empty). These networks $({\mathscr N}^{p}, a^{p})$ which we call  {\em sub-networks}, should be chosen to satisfy certain properties we now describe carefully. First, for each edge $[p, q] \in {\mathscr E}$, we assume that we have identified a vertex $r^{p}_{q} \in {\mathscr V}^{p}$ (one should not confuse $\mathscr V^p$ which is the set of vertices of $\mathscr N^p$ with $\mathscr V_p$ which is the set of vertices $q \in \mathscr V$ such that $[p,q] \in \mathscr E$ and which has been defined in (\ref{eq:Vp})) and a vertex $r^{q}_{p} \in {\mathscr V}^{q}$. Since a given vertex $r \in {\mathscr V}^p$ might be associated to many edges of ${\mathscr E}$, we define 
\[
{\mathscr V}_{p,r} : = \left\{q \in {\mathscr V} \, : \, r = r^p_q \right\} ,
\]
which can be either empty, in which case we call such a vertex an {\em internal vertex} of the sub-network $({\mathscr N}^p, a^p)$, or can contain only one point or can contain many points, in which case we call such a vertex an {\em external vertex} of the sub-network $({\mathscr N}^p, a^p)$. 

For all  $p \in {\mathscr V}$, we define
\[
{\mathscr E}^p_{ext} : = {\mathscr E}^p \, \cup \{ {\rm R}^p_{r,q}\, : \, \forall r \in {\mathscr V}^p, \quad  \forall  q \in {\mathscr V}_{p,r}\} ,
\]
where the ray ${\rm R}^p_{q,r}$ is defined by
\[
{\rm R}^{p}_{r,q} : = \left\{ r+ t \, \frac{q-p}{|q-p|} \, : \,  t >0 \right\} 
\]

\begin{center}
\includegraphics[width= 7cm]{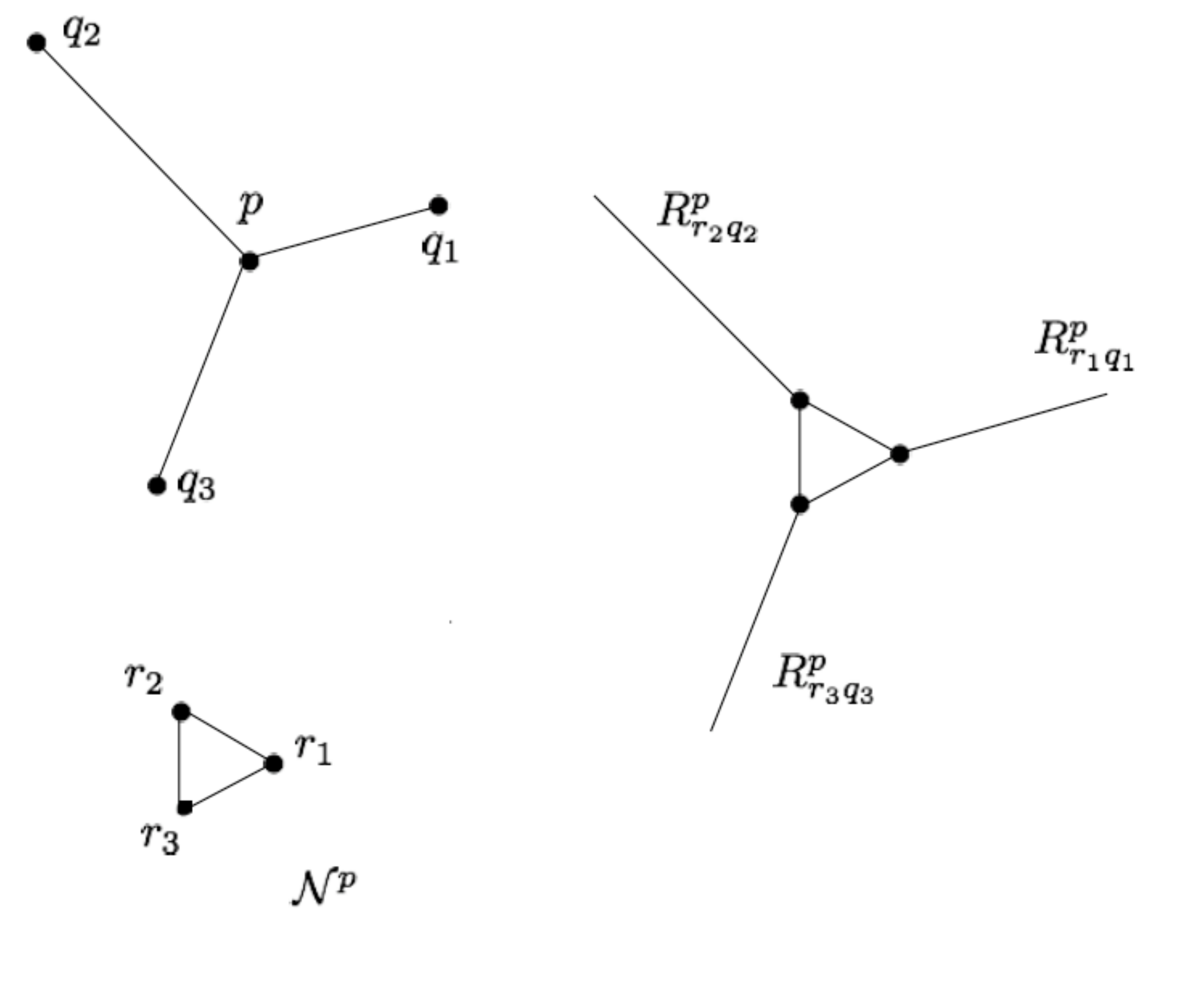}\\
Fig 8 : Example of a vertex $p \in \mathscr V$ with the edges of $\mathscr E$ ending at $p$, a sub-network $\mathscr N^p$ and the same sub-network where the rays are drawn. 
\end{center}

\begin{center}
\includegraphics[width=8cm]{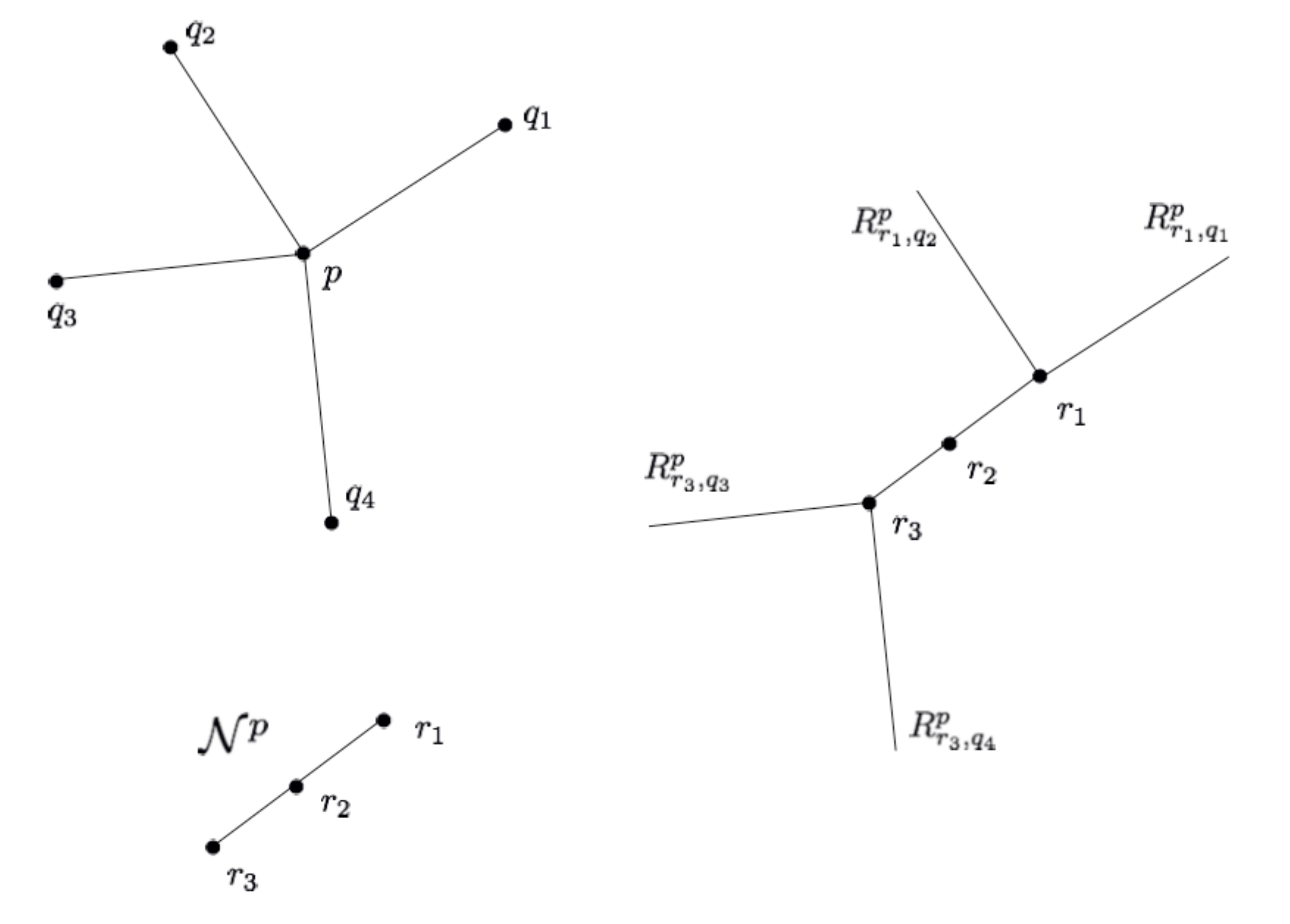}\\
Fig 9 : Example of a vertex $p \in \mathscr V$ with the edges of $\mathscr E$ ending at $p$, a sub-network $\mathscr N^p$ and the same subnetwork where the rays are drawn. 
\end{center}

We require that the following properties hold~:
\begin{itemize}

\item[(i)] The barycenter of the vertices of ${\mathscr N}^{p}$ is equal to $0$, namely
\[
\sum_{r \in {\mathscr V}^{p}} r = 0 .
\]

\item[(ii)] (Embeddedness of the sub-netkorks with rays) For each $p \in {\mathscr V}$, any two distinct elements of ${\mathscr E}^p_{ext}$ (which might be edges or rays) are either disjoint or intersect at their end points. 

\item[(iii)] (Internal vertices are balanced)   If $r \in {\mathscr V}^p$ is an internal vertex, i.e. is not equal to any of the $r^{p}_q$, then 
\[
\sum_{r' \in {\mathscr V}_r^{p}} a^{p}_{[r'r]} \, \frac{r'-r}{|r'-r|} = 0,
\]
where ${\mathscr V}_r^{p}$ is the set of vertices $ r' \in \mathscr V^p$ such that $[r',r] \in \mathscr E^p$. 

\item[(iv)] (Balancing conditions for external vertices)  If $r = r^{p}_q \in {\mathscr V}^{p}$ is an external vertex, then  
\[ 
\sum_{r' \in {\mathscr V}_r^{p}} a^{p}_{[r'r]} \, \frac{r'-r}{|r'-r|}  +  \sum_{q' \in {\mathscr V}_{p,r}} a_{[p,q']} \frac{q'-p}{|q'-p|} =0 .
\]

\item[(v)] (No other closest neighbor conditions) If $r \neq r' \in {\mathscr V}^p$ and if $|r'-r|\leq 1$, then $[r, r'] \in {\mathscr E}^p$ and hence $|r'-r|=1$.

\item[(vi)] (No other closest neighbor conditions for rays)  For all $[p,q] \in {\mathscr E}$, 
\[
\min_{r' \in {\mathscr V}^p, \, r' \neq r} \min_{j\in {\bf N}-\{0\}} \, \left| r'-r- j\frac{q-p}{|q-p|} \right| >1.
\] 
and we also require that, for all $[p, q'] \in {\mathscr E}$ distinct from $[p,q]$, we have
\[
\min_{j, j' \in {\bf N}-\{0\}} \, \left| r^p_{q'} + j' \frac{q'-p}{|q'-p|} - r^p_q - j\, \frac{q-p}{|q-p|} \right| >1.
\] 

\item[(vii)] (Sign compatibility) It is possible to define a function $\eta^p : {\mathscr V}^p \to \{\pm 1\}$ in such a way that 
\[
\eta^p_r \, \eta^p_{r'} = \text{sign} (a_{[r,r']}^p) ,
\]
for all $r,r' \in  {\mathscr V}^p$ and 
\[
\eta^p_{r^p_q} = \eta^q_{r^q_p}  ,
\]
for all $[p,q] \in {\mathscr E}$. 
\end{itemize} 

Let us give a couple of examples of such configurations. 

\medskip

\noindent
{\bf Example 5.1 :} \label{ex:5.1} We assume that the network $({\mathscr N}, a)$ is the one described in Example~\ref{ex:3.4}. Namely, the regular polygon with $k$ sides together with the origin and the edges joining the origin to the vertices of the polygon. The vertices of this network are given by
\[
{\mathscr V} : =\{0\} \cup \left\{ \xi^j \, : \, j = 0, \ldots, k-1\right\},
\]
where $\xi : =  e^{2i\pi/k}$ and the weight function $a$ is chosen to be 
\[
a_{[0, \xi^j]} = 2 \, \sin \pi/k ,
\]
and 
\[
a _{[\xi^j, \xi^{j+1}]} =-1.
\] 
The angle between the edges $[0,1]$ and $[1, \xi]$ is given by $\pi/2-\pi/k$ and hence, when $k \geq 7$ this angle is larger than $\pi/3$. In particular, if we chose the sub-network ${\mathscr N}^{\xi^j}$ to be equal to $\{0\}$,  conditions (vi) will be fulfilled.     

In contrast, at the origin, the angle between the edges $[0,1]$ and $[0, \xi]$ is less than $\pi/3$ when $k\geq 7$ and condition (vi) will not be fulfilled if we chose the sub-network ${\mathscr N}^{0}$ to be equal to $\{0\}$. This is the reason why, we  choose the sub-network $( {\mathscr N}^{0},  a^0)$ to be the polygon described in Example~\ref{ex:3.2}. Namely, the network whose set of vertices is given by
\[
{\mathscr V}^0 : =\left\{ z_j : = \frac{\xi^j}{|1-\xi|}  \, : \, j = 0, \ldots, k-1 \right\},
\]
and where the weight function $a^0$ is chosen to be 
\[
a _{[ z_j, z_{j+1}]} =1.
\] 
This time (vi) is fulfilled. 

\begin{center}
\includegraphics[width=10cm]{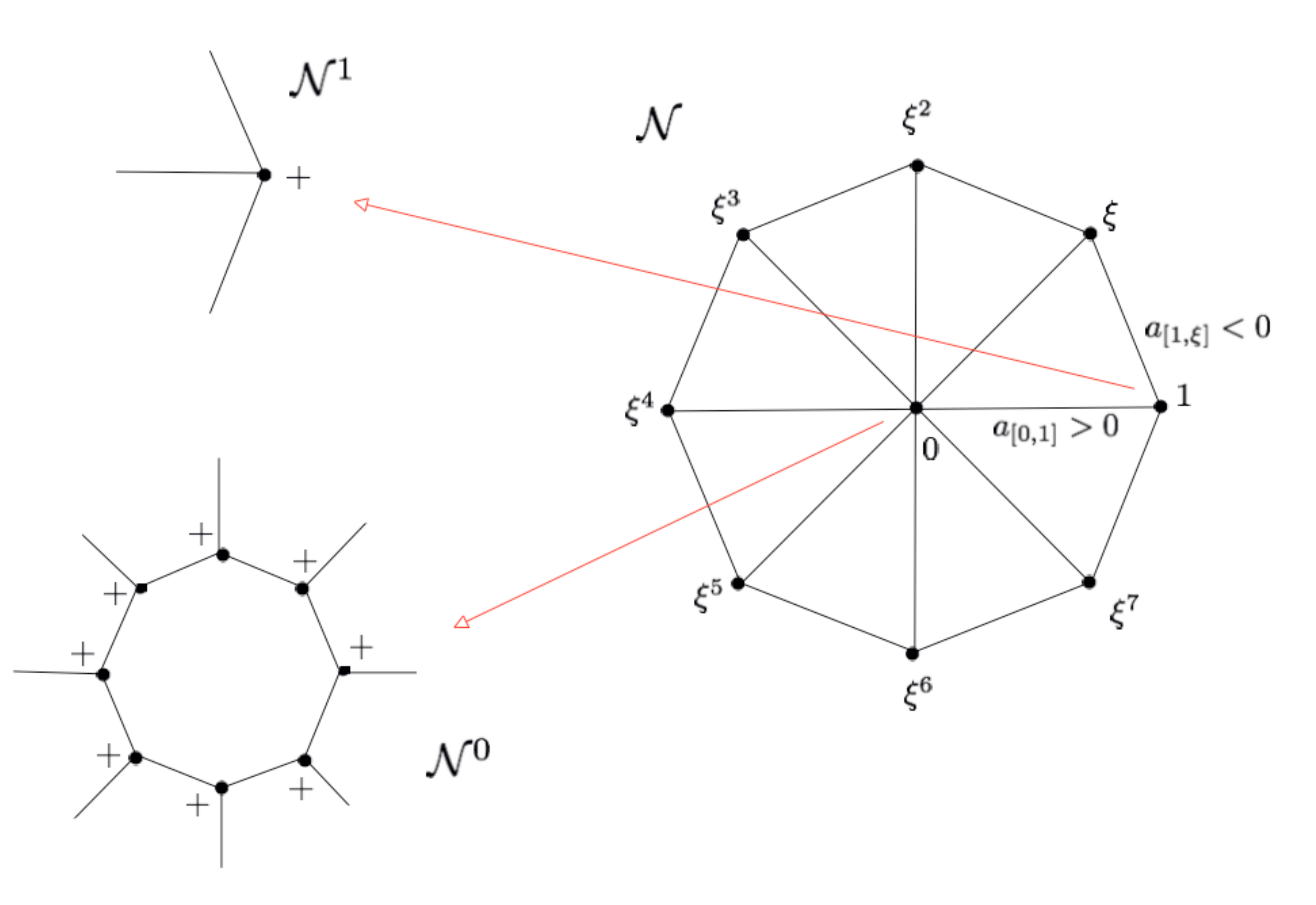}\\
Fig 10 : Example of a network $\mathscr N$ and sub-networks at the points $0$ and $1$. The sign of the weight function is mentioned as well as the signs associated to the vertices of the sub-networks.
\end{center}

We leave to the reader to check that all properties (i) to (vii) are fulfilled with these choices of sub-networks. This is the example which was originally considered in \cite{MPW}.

\medskip

\noindent
{\bf Example 5.2 :} \label{ex:5.2} Again, we start with the network $({\mathscr N}, a)$ which is the regular polygon with $k$ sides together with the origin and the edges joining the origin to the vertices of the polygon. This time  we assume that $k=4$ or $k=5$ to ensure that the angle between the edges $[0,1]$ and $[0, \xi]$ is larger than $\pi/3$. Hence, we can choose  the sub-network $({\mathscr N}^{0}, a^0)$ to be equal to $\{0\}$ and (vi) will be fulfilled with this choice. 

However, since $k\leq 5$, the angle between the edges $[0,1]$ and $[1, \xi]$ is less than $\pi/3$ and we cannot take  ${\mathscr N}^{\xi^j}$ to be equal to $\{0\}$ since (vi) would not be fulfilled. Therefore, this time, to construct a sub-network ${\mathscr N}^{\xi^j}$ satisfying (vi), we consider the example described in Example~\ref{ex:3.3}, with vertices $z_0 : = 1/\sqrt{3}, z_1 : = \zeta/\sqrt{3}, z_2 : = \zeta^2/\sqrt{3}$ where $\zeta$ is the $3$-rd  root of unity. 

\begin{center}
\includegraphics[width=10cm]{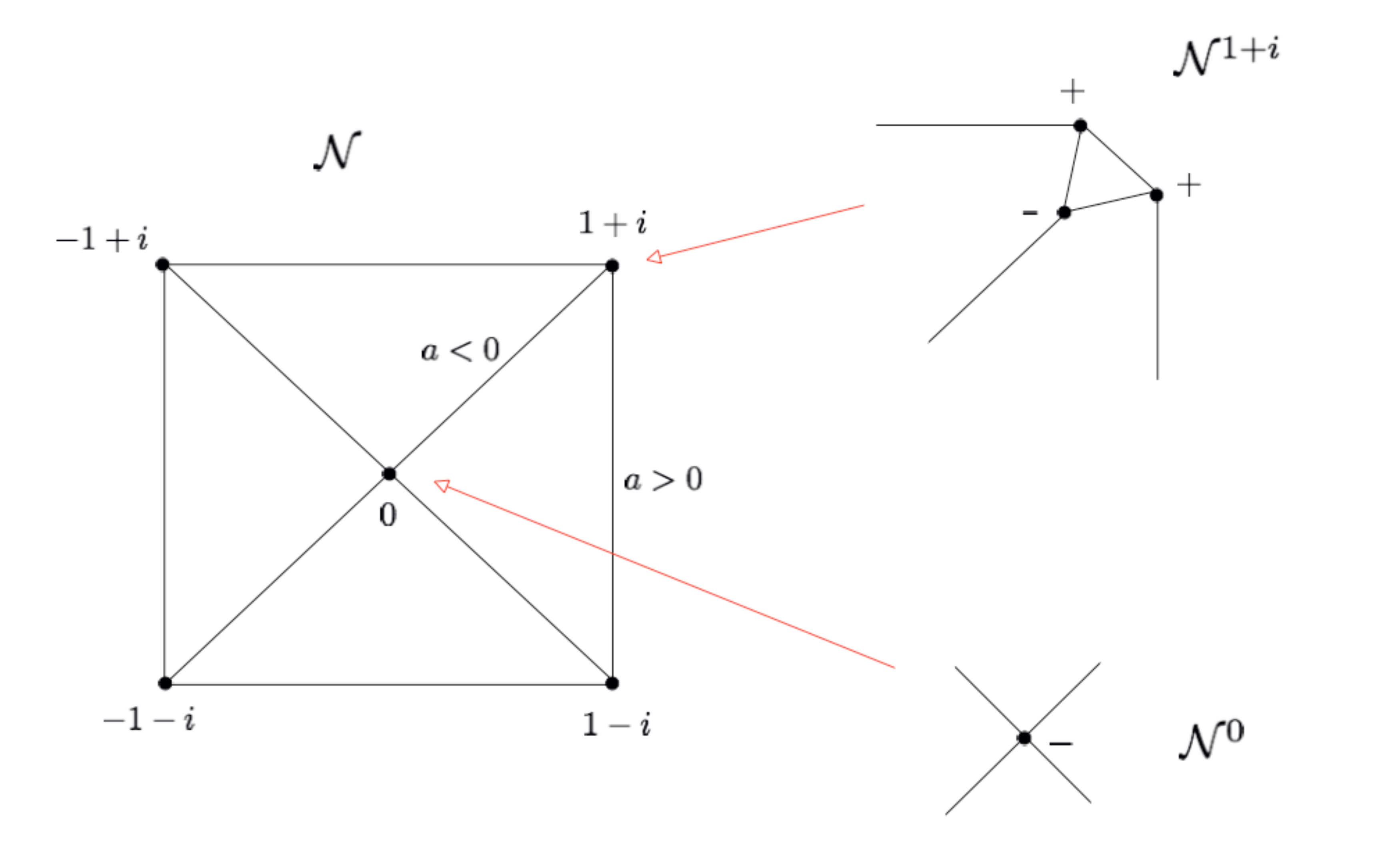}\\
Fig 11 : Example of a network $\mathscr N$ and sub-networks at the points $0$ and $1+i$. The sign of the weight functions are mentioned as well as the signs associated to the vertices of the sub-network.
\end{center}

We define the weight 
\[
a_{[z_0, z_1]} = a_{[z_2,z_0]} = - \frac{1}{\sqrt 3} \, \sin (\pi/k) ,
\]
and
\[
a_{[z_1, z_2]} = \cos (\pi/k) +  \frac{1}{\sqrt 3} \, \sin (\pi/k).
\] 
Then, we define the sub-network $({\mathscr N}^{\xi^j}, a^{\xi^j})$ to be the network $({\mathscr N}_{Tri}, a)$ which is rotated by $\pi+ j\, 2\pi/k$.  Observe that (vi) requires that $k\geq 4$ and flexibility of the unbalanced triangle requires that $k\neq 3$. We leave to the reader to check that all properties (i) to (vii) are fulfilled with these choices of sub-networks. 

\begin{remark}
These  two examples are particularly interesting because, according to the result of Corollary~\ref{co:net-6.1}, any small perturbation of $({\mathscr N}, a)$ can also be balanced and it is easy to see that for small perturbations, one can deform $({\mathscr N}^{\xi^j}, a^{\xi^j})$ in such a way that the networks still fulfill assumptions (i) to (vii). In particular, this leads to configurations which have fewer or even which have absolutely no symmetry (for example, one can just move the vertex $0$ to $\epsilonÊ\, e^{i\theta}$ for some $\theta \in (0, \pi/k)$ and some $\epsilon >0$ small, to produce networks which have no symmetry). 
\end{remark}

Given $\kappa \gg 1$ and $[p,q] \in {\mathscr E}$, we define $m_{[p,q]} \in {\bf N}$ by
\begin{equation}
\kappa \, \frac{| q -p |}{ 1 - \alpha_\ell (a_{[p,q]}) } \leq  2 m_{[p,q]}  < \kappa \, \frac{| q -p |}{ 1 - \alpha_\ell (a_{[p,q]}) } + 2.
\label{eq:mpq}
\end{equation}

We have the~:
\begin{proposition}
Assume that $({\mathscr N} , a )$ is a closable, flexible network, and, for each $p \in {\mathscr V}$,  assume that $({\mathscr N}^p , a^p )$  is a flexible unitary network, such that properties (i)-(vii) in \S~\ref{se:dsds} are fulfilled. Then, there exists $\ell_* >0$, $\kappa_* >0$ and $\epsilon_* >0$ such that, for all $\ell\geq \ell_*$, for all $\kappa \geq \kappa_* \, \ell^3$ and for all sets of forces  $({\bf f}_r^p)_{ r \in {\mathscr V}^p} \in {\bf C}^{n_p}$,  where $n_p$ is the number of vertices of $\mathscr V^p$, such that 
\[
\ell^3 \, |{\bf f}_r^p|\leq  \epsilon_*,
\] 
there exists~:
\begin{enumerate}
\item[(i)] $\Phi : \mathscr V \to {\bf C}$ and  $\tilde a : \mathscr E \to {\bf R}-\{0\}$ ;
\item[(ii)] $\Phi^p : \mathscr V^p \to {\bf C}$ and  $\tilde a^p : \mathscr E^p \to {\bf R}-\{0\}$, for each $p \in \mathscr V$ ; 
\item[(iii)] ${\bf e} \in {\bf C}$ and $t \in {\bf R}$ close to $0$,
\end{enumerate}
all smoothly  depending on the ${\bf f}_r^p$ such that~:
\begin{enumerate}

\item[(a)] For all $p\in \mathscr V$ and for all $[r, r'] \in {\mathscr E}^{p}$, we have
\[
| \tilde r-  \tilde r'| = 1 - \alpha_\ell (\tilde a^p_{[ r , r']}) ,
\]
where  $\tilde r : =  \Phi^p_r$ and $\tilde r' : =  \Phi^p_{r'}$.

\item[(b)] For all $[p,q] \in \mathscr E$, we have
\[
|  (\kappa \,  \tilde q  + \tilde r^q_p ) - (\kappa \, \tilde p  + \tilde r^p_q) | = 2 \, m_{[p,q]} \, \left( 1 - \alpha_\ell (\tilde a_{[ p,q]}) \right),
\]
where $\tilde p := \Phi_p$, $\tilde q : =  \Phi_q$, $\tilde r^q_p : =  \Phi^q_{r^q_p} $ and $\tilde r^p_q  : =  \Phi^p_{r^p_q}$.

\item[(c)] If $r \in {\mathscr N}^{p}$ is an internal point of $\mathscr V^p$, then  
\[
\sum_{r' \in {\mathscr V}_r^{p}} \tilde a^{p}_{[r'r]} \, \frac{ \tilde r' -  \tilde r}{| \tilde r' -  \tilde r |}  = {\bf f}^p_r + \frac{{\bf e} + i \, t \, p}{n_p} ,
\]
where  $\tilde r : =  \Phi^p_r$ and $\tilde r' : =  \Phi^p_{r'}$.

\item[(d)] If $r \in {\mathscr N}^{p}$ is an external point of $\mathscr V^p$, then 
\[
\sum_{r' \in {\mathscr V}_r^{p}} \tilde a^{p}_{[r'r]} \, \frac{ \tilde r' -  \tilde r}{| \tilde r'  -  \tilde r|}  +  \sum_{q \in {\mathscr V}_{p,r}} \tilde a_{[p,q]} \, \frac{ (\kappa \,  \tilde q  + \tilde r^q_p ) - (\kappa \, \tilde p  + \tilde r^q_p) }{| (\kappa \,  \tilde q  + \tilde r^q_p ) - (\kappa \, \tilde p  + \tilde r^p_q) |} =  {\bf f}^{p}_r + \frac{{\bf e} + i \, t \, p}{n_p} ,
\]
where $\tilde r : = \Phi^p_r $, $\tilde r' : = \Phi^p_{r'}$,  $\tilde p := \Phi_p$, $\tilde q : =  \Phi_{q}$, $\tilde r^p_q : =  \Phi^p_{r^p_q} $ and $\tilde r^q_p  : =  \Phi^q_{r^q_p}$.

\item[(e)] For all $p \in {\mathscr V}$, 
\[
\sum_{r \in \mathscr V^p} \tilde r =0 ,
\]
where $\tilde r : =  \Phi^p_r$.

\item[(f)] Finally
\[
\sum_{p \in \mathscr V} (\tilde p  - p ) =0 ,
\qquad \text{and} 
\qquad 
\sum_{p \in \mathscr V} (\tilde p - p) \wedge \tilde p =0 .
\]
where $\tilde p :=  \Phi_p$ and 
\[
\begin{array}{rlllll}
\displaystyle \ell \sup_{p \in \mathscr V}  |\tilde p - p| + \sup_{[p,q] \in \mathscr E} |a_{[p,q]}  - \tilde a_{[p,q]}| + \sup_{p \in {\mathscr V}}  \left( \sup_{r \in {\mathscr V}^p} |r -  \tilde r| + \sup_{[r,r'] \in {\mathscr E}^p} | a^p_{[r,r']} - \tilde a^{p}_{[r,r']} |\right) \qquad \qquad \\[3mm]
\displaystyle  \leq C \, \ell^2 \, \left( |\bf f| + \frac{1}{\kappa}\right).
\end{array}
\]
\end{enumerate}
\label{pr:thisisit}
\end{proposition}
\begin{proof}
The proof is a simple modification of the proofs of the previous related results. The key observation is that, letting $\kappa$ tend to infinity and summing the equations in (c) and (d), we get
\[
{\bf F}_{(\Phi, \tilde a)} (p) =  \sum_{r\in \mathscr V^{p}} \, {\bf f}^p_r  +  {\bf e} + i \, t \, p ,
\]
which shows that the system in the main networks and the sub-networks is somehow in diagonal form. 
\end{proof}

Let us briefly comment on this result. Starting from a balanced network $( {\mathscr N}, a)$, we first replace each vertex $p \in \mathscr V$ by a subnetwork $({\mathscr V}^p, a^p)$ and build a network whose set of vertices is the union of the vertices of each subnetwork ${\mathscr N}$. The result of proposition~\ref{pr:thisisit} asserts that we can move the vertices of the subnetworks in such a way that the resulting force at each $z^p_r$ is given by ${\bf f}^p_r$ (modulo $\frac{{\bf e} + t \, i \, p}{n_p}$).

\subsection{Construction of the approximate solution}

We build on the result of Proposition~\ref{pr:thisisit}. As in the statement of this Proposition, we assume that $(\mathring {\mathscr N} , \mathring a )$ is a closable, flexible network, and, for each $p \in \mathring {\mathscr V}$, we also assume that $( \mathring{\mathscr N}^{p} , \mathring a^{p} )$  is a flexible unitary network, satisfying properties (i)-(vii) in \S\ref{se:dsds}. For all $\ell \geq \ell_*$, for all $\kappa \geq \kappa_*\, \ell^3$ and for all sets of forces ${\bf f}_{r}^{p} \in {\bf C}$,  such that 
\[
\ell^2 \, |{\bf f}_{r}^{p}|\leq  \epsilon_*,
\] 
we denote  by $({\mathscr N}, a)$  and $({\mathscr N}^p, a^p)$, the weighted network and sub-networks whose existence follow from the result of Proposition~\ref{pr:thisisit} (with slight abuse of notations, we have used the same notations for the vertices of $\mathring{\mathscr V}$ and $\mathscr V$). Let us insist that these networks and subnetworks do depend on the choice of ${\bf f}^p_r$. 

We dilate the network ${\mathscr N}$ by a factor $\kappa \, \ell$ and, for each $p \in \mathscr V$, we replace the vertex $\kappa \, \ell \, p$  by the sub-network $\mathscr N^p$ which in turn is dilated by a factor $\ell$ and translated by  $\kappa \, \ell \, p$. We get a new network whose vertices are given by
\[
z_r^p : = \ell \, (\kappa \, p + r) ,
\]
for $p \in \mathscr V$ and $r \in \mathscr V^p$ and whose edges are either of the form $[z_r^p, z_{r'}^p] $ for some $r , r' \in \mathscr V^p$ and some $p \in \mathscr V$ or of the form $[ z_{r}^p, z_{r'}^q]$ for some external vertex $r = r^p_q \in \mathscr V^p$ and some external vertex  $r'=r^q_p \in \mathscr V^q$ for some $p\neq q \in \mathscr V$. Therefore, the number of vertices of our new network is equal to the sum over $p \in \mathscr V$ of the number of vertices of each $\mathscr N^p$, while the number of edges of our new network is equal to the sum of the number of edges of $\mathscr N$ plus the sum over $p \in \mathscr V$ of the number of edges of each $\mathscr N^p$.

Observe that, by construction,  the length of the edge $[z_r^p, z_{r'}^p] $ where $r , r' \in \mathscr V^p$ for some $p \in \mathscr V$ is given by
\[
| z_r^p - z_{r'}^p |  = \ell - \lambda^p_{[r,r']} ,
\]
where 
\[
\lambda_{[r,r']}^p :  = \ell \,  \alpha_\ell (a^p_{[r,r']} ), 
\]
while the length of the edge $[ z_{r}^p, z_{r'}^q]$ where $r = r^p_q \in \mathscr V^p$ and where  $r'=r^q_p \in \mathscr V^q$ for some $p\neq q \in \mathscr V$ is given by
\[
| z_{r}^p, z_{r'}^q  |  = 2 \, m_{[p,q]} \, ( \ell- \lambda_{[p,q]} ) ,
\]
where 
\[
\lambda_{[p,q]} :  = \ell \,  \alpha_\ell (a_{[p,q]} ) .
\]
In particular, we can insert exactly  $2\, m_{[p,q]} -1$ points between $z_{r^p_q}^p$ and $z_{r^q_p}^q$, in such a way that the distance between two consecutive points is exactly equal to $\ell  - \lambda_{[p,q]}$. More precisely, if we define 
\[
{\bf e}_{pq} : =  \frac{r^q_p - r^p_q}{|r^q_p - r^p_q|},
\]
we can label the points we evenly distribute along the edge $[ z_{r^p_q}^p, z_{r^q_p}^q]$ by
\[
z_j^{pq} : = z_{r^p_q}^p + j \,  (\ell  - \lambda_{[p,q]} ) \, {\bf e}_{pq},
\]
for $j=0, \ldots, 2m_{[p,q]}$. Observe that, by definition $z_0^{pq} = z^p_{r^p_q}$ and $z_{2m_{pq}}^{pq} = z_{r^q_p}^q$. Moreover, since ${\bf e}_{pq} = - {\bf e}_{qp}$, we have
\[
z_{m_{[p,q]} + j}^{pq}  = z_{m_{[p,q]} - j}^{qp} ,
\]
for $j=-m_{[p,q]} , \ldots, m_{[p,q]}$. 

We define the set $Z$ as the union of the sets of vertices we have just defined
\[
Z : = \left\{ z^p_r \, : \, \forall  p \in \mathscr V^p,  \, \forall r \in  \mathscr V^p \right\} \cup \left\{z_j^{pq}  \, : \, \forall  [p,q] \in \mathscr E , \quad \forall j=1, \ldots, 2m_{[p,q]} -1 \right\}.
\]

\begin{center}
\includegraphics[width=10cm]{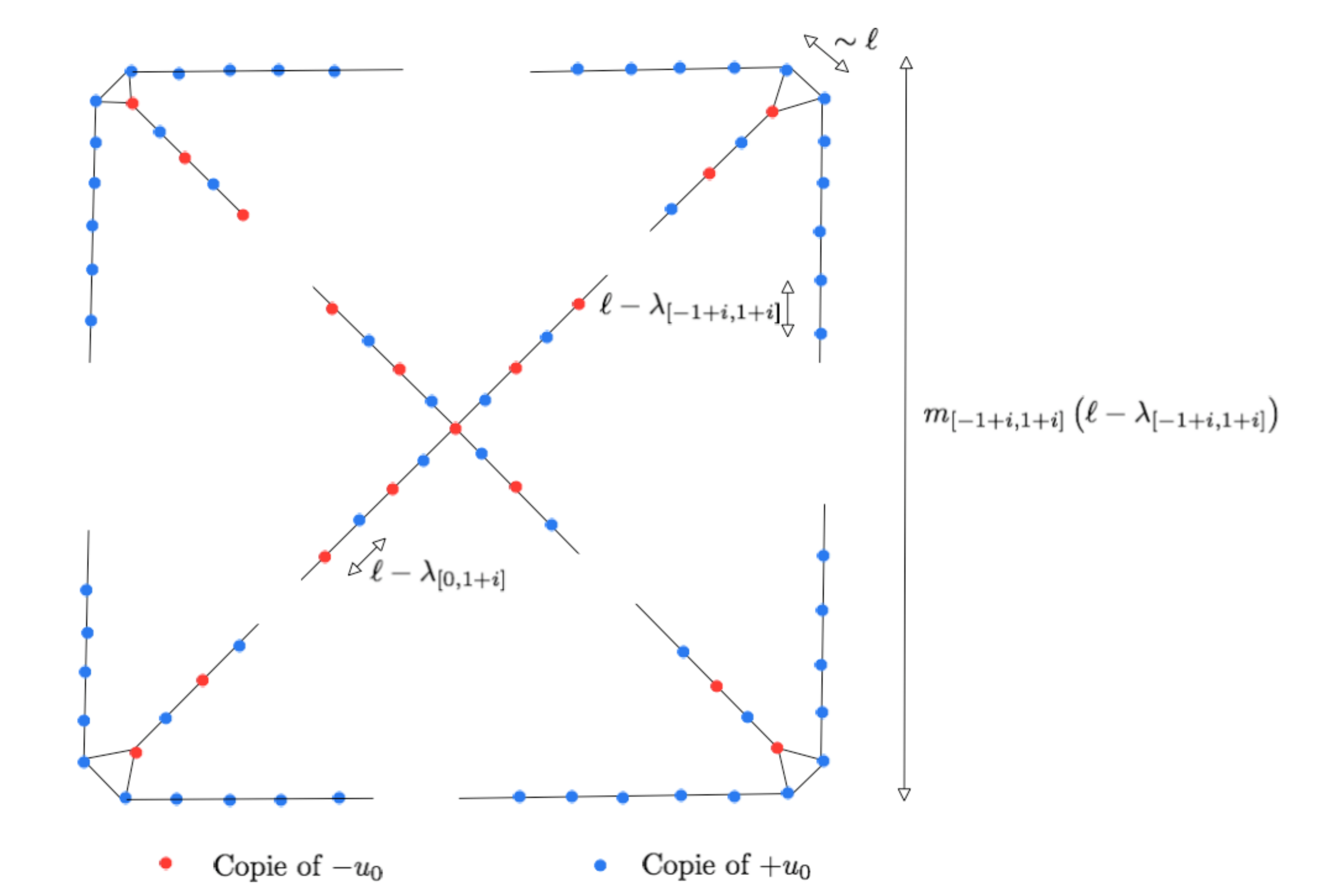}\\
Fig 12 : The blue and red dots correspond to the set of points $Z$ which one obtains starting from the network and sub-network described in Fig 9. 
\end{center}

We now need to distinguish, among the points of $Z$, which are the points that belong to $Z^+$ and the points that belong to $Z^-$.  Recall that, each sub-networks $\mathring{\mathscr N}^p$ enjoy property (vii) and we can also equip the sub-networks  $\mathscr N^p$ with a function $\eta^p : \mathscr V^p \to \{\pm 1\}$ satisfying (vii) by choosing that $\eta^p$ at the vertex $r \in \mathscr V^p$ is equal to the value it had at the corresponding vertex in $\mathring{\mathscr V}^p$.  Hence, for each $p \in \mathscr V$ and for each $r \in \mathscr V^p$, we define
\[
\eta_{z^p_r} := \eta^p_r.
\]
Now, for each $[p,q] \in \mathscr E$ and each $j=0, \ldots, 2m_{[p,q]}$, we define 
\[
\eta_{z_j^{pq}} : = (-1)^j \, \eta^p_{r^p_q} ,
\]
when $a_{[p,q]} <0$, while we define 
\[
\eta_{z_j^{pq}} : =  \eta^p_{r^p_q} ,
\]
when $a_{[p,q]} > 0$.  Observe that property (vii) implies that this is well defined. In particular, when $a_{[p,q]} <0$ we have
\[
\eta_{z_{m_{[p,q]} +j}^{pq}} = \eta_{z_{m_{[p,q]} -j}^{qp}} ,
\]
for $j=-m_{[p,q]} , \ldots, m_{[p,q]}$.

By definition 
\[
Z^\pm : =  \{ z \in Z \, : \, \eta_z = \pm 1\}. 
\]
We recall that, when constructing the approximate solution, we will center copy of $+u_0$ at each of the points of $Z^+$ and copies of $-u_0$ at the points of $Z^-$. By construction, the points of $Z$ which belong to the edge $[ z_{r^p_q}^p, z_{r^q_p}^q ]$ and which are not the end points, are balanced in the sense that
\begin{equation}
\sum_{z' \in N_z} \eta_z \, \eta_{z'} \, \Upsilon (|z'-z|) \, \frac{z'-z}{|z'-z|} =0 ,
\label{eq:balj}
\end{equation}
where, as in section 2, $N_z$ is the set of closest neighbors of $z$ in $Z$. In fact each of such a point has only two closest neighbors $z'$ and $z''$ such that $z'-z= z-z''$ and the identity follows at once. 

While, at points $z\in Z$ of the form $z = z^p_r \in \mathscr V^p$, we have 
\[
\sum_{z' \in N_z} \eta_z \, \eta_{z'} \, \Upsilon (|z'-z|) \, \frac{z'-z}{|z'-z|} =  \Upsilon (\ell) \, \left( {\bf f}_r^p + \frac{{\bf e} + i \, t \, p}{n_p} \right).
\]
where $n_p$ is the cardinal of $\mathscr V^p$. Indeed, it follows from the definition of $\alpha_\ell$ given in (\ref{eq:net-4}) that, if $z= z_r^p$ and $z' = z^p_{r'}$ are closest neighbors, where $r,r' \in \mathscr V^p$ , then 
\[
\eta_z \, \eta_{z'} \, \Upsilon (\ell - \lambda_{[r,r']})  =  \Upsilon (\ell) \, a^p_{[p,q]}
\]
while, if $z = z^{pq}_0$ and $z' = z^{pq}_1$, then 
\[
\eta_z \, \eta_{z'} \,  \Upsilon (|z'-z|) = \Upsilon(\ell) \, a_{[p,q]}
\]
and the identity follows from (c) and (d) in Proposition~\ref{pr:thisisit}. 

Now, for each $[p,q] \in \mathscr E$ and for each $ j=1, \ldots, 2m_{[p,q]} -1$, we choose a point $\tilde z_j^{[p,q]}$ close to the point $z_j^{[p,q]}$ and we define 
\[
\tilde Z : = \left\{ z^p_r \, : \, \forall  p \in \mathscr V^p,  \, \forall r \in  \mathscr V^p \right\} \cup \left\{ \tilde z_j^{[p,q]}  \, : \, \forall [p,q] \in \mathscr E , \quad \forall j=1, \ldots, 2m_{[p,q]} -1 \right\},
\]
We will assume that, for all $[p,q] \in \mathscr E$ and for all $j=1, \ldots, 2m_{[p,q]}-1$, we have 
\begin{equation}
| \tilde z_j^{[p,q]} - z_j^{[p,q]}   | \leq e^{- \gamma_0 \, \ell} ,
\label{eqdsz}
\end{equation} 
for some $\gamma_0 >0$ which will be fixed later on. Observe that we do not modify the points $z \in Z$ of the form $z^p_r$ but we only modify the points on the edges $[z^p_{r^p_q}, z^q_{r^q_p}]$.  We define a function $\tilde \eta : \tilde Z \to \{\pm 1\}$ by requiring that $\tilde \eta_{z} = \eta_z$ if $z=z^p_r$ for some $r \in \mathscr V^p$ and $\tilde \eta_{\tilde z_j^{[p,q]}} = \eta_{z_j^{[p,q]}}$.

The approximate solution $\tilde u$ is then given by
\begin{equation}
\tilde u : = \sum_{z \in \tilde Z} \tilde \eta_z \, u_0 (\cdot - z) ,
\label{eq:tildeu}
\end{equation}
where $\tilde \eta_z$ is the sign assigned to the vertex $z \in \tilde Z$. The set of closest neighbors of $z$ in $\tilde Z$ will be denoted by $\tilde N_z$.

\section{Linear analysis}

In this section, we keep the assumptions and notations introduced in \S \ref{se:dsds} and we study the operator
\[
\tilde L : =  \Delta -1 + 3\, \tilde u,
\]
where $\tilde u$ is defined in (\ref{eq:tildeu}).

The mapping properties of all the linear operators we will consider  rely on the following~:
\begin{lemma}
\label{le:injL0}
The bounded kernel of $L_0 : = \Delta -1+ 3u_0$ is spanned by $\partial_x u_0$ and $\partial_y u_0$.
\end{lemma}
We refer to \cite{nitakagi} for a proof of this result. 

Given $\delta \in {\bf R}$, we define the weighted space
\[
L^\infty_\delta ({\bf C}) : =  e^{Ê\delta \, \sqrt{1+|z|^2}} \, L^\infty ({\bf C}),
\]
and agree that 
\[
\| v \|_{L^\infty_\delta (\bf C)} : =  \left\| e^{Ê- \delta \, \sqrt{1+|z|^2}}\, v \right\|_{L^\infty (\bf C)}.
\] 
We have the~:
\begin{proposition}
\label{pr:invl0}
Assume that $\delta \in (-1,0)$. Then, for all $f \in L^\infty_\delta ({\bf C})$ there exists a unique $v \in L^\infty_\delta ({\bf C}) $ and $c \in {\bf C}$ such that 
\[
L_0 \, v + \langle c, \nabla u_0 \rangle_{\bf C} = f ,
\]
in ${\bf C}$ and 
\[
\iint_{\bf C} v \, \partial_x u_0 \, dx \, dy = \iint_{\bf C} v \, \partial_y u_0 \, dx \, dy = 0 .
\]
Moreover, 
\[
\| v \|_{L^\infty_\delta (\bf C)} + |c |  \leq C \, \| f \|_{L^\infty_\delta (\bf C)},
\]
for some constant $C >0$ which does not depend on $f$.
\end{proposition}
\begin{proof}
We consider the Hilbert space
\[
H := \left\{ v \in H^1 ({\bf C} ) \, : \, \iint_{{\bf C}} \partial_{x} u_0  \, v \, dx\, dy  =  \iint_{{\bf C}}  \partial_{y} u_0  \, v \, dx\, dy   =  0\right\} .
\]
Assume that we are given $h \in L^2 ({\bf C})$. Standard arguments (i.e. Lax-Milgram's Theorem) imply that
\[
v \in H \longmapsto \frac{1}{2} \, \iint_{\bf C} \left( |\nabla v|^2 + v^2 - v \, h \right) \, dx\, dy  , 
\]
has a unique minimizer  $v \in H$ (here we implicitly use the fact that $\delta < 0$ so that the last term is a continuous linear functional defined in $H$). Then, $v$ is the unique weak solution of
\[
\Delta \, v  - v - h  \in \displaystyle   \mbox{Span} \, \left\{ \partial_x u_0, \partial_y u_0 \right\}  , 
\]
which belongs to $H$. In other words, if we define the operator 
\[
L_0^\flat  (v, c) : = \Delta v - v + \langle c , \nabla  u_0\rangle_{\bf C}, 
\]
we have obtained the existence and uniqueness of a solution of
\[
L_0^\flat (v, c) = h  , 
\]
with $v \in H$ and $c \in {\bf C}$.  The solvability of
\[
\Delta v - v + 3 \, u_0^2 \, v  + \langle c , \nabla  u_0\rangle_{\bf C} = h  , 
\]
in $H \times {\bf C}$ can then by rephrased in the invertibility of the operator $I + K$, where by definition
\begin{equation}
\label{eq:ft}
	K (v , c) : = (L_0^\flat) ^{-1}  (3 \, u_0^2Ê\, v)  . 
\end{equation}
Using the fact that $u_0$ decays exponentially at infinity, it is easy to check that the operator $K$ is compact, hence the invertibility of (\ref{eq:ft}) follows from the application of Fredholm theory. Since injectivity follows from the results of Lemma~\ref{le:injL0}. Fredholm alternative implies that $I + K$ is
therefore an isomorphism.

So far, we have obtained a function $v$  solution of $L_0 \, v + \langle c , \nabla  u_0\rangle_{\bf C}  = h$ which belongs to $H^1({\bf C})$ but elliptic regularity implies that $v \in L^\infty
({\bf C})$ and that 
\[
\| v\|_{L^\infty (\bf C)} \leq  C \, \| f\|_{L^\infty_\delta (\bf C)} ,
\] 
for some constant $C >0$.  We need to check that the solution $v$ has the correct behavior at infinity. To this aim, just remark that if we define
\[
v_\epsilon (r) : = e^{\delta r} + \epsilon \, e^{-\delta r},
\]
then 
\[
L_0 \, v_\epsilon \geq  - \frac{(1 - \delta^2)}{2} \, v_0 
\]
on the complement of the ball of radius $r_0$, provided is fixed large enough.  Hence, the function $\left( \|v\|_{L^\infty} + \frac{2}{1-\delta^2} \, \|Êf\|_{L^\infty_\delta} \right) \, v_\epsilon$  is certainly a super-solution for our problem on the complement of the disc of radius $r_0 >0$ and, passing to the limit as $\epsilon$ tends to $0$, this proves that there exists a constant $C >0$ such that 
\[
\| v\|_{L^\infty_\delta  (\bf C)} \leq  C \, \| f\|_{L^\infty_\delta (\bf C)} ,
\] 
This completes the proof of the existence of the solution. The uniqueness and the corresponding 
estimate follow at once from the result of Lemma~\ref{le:injL0}.
\end{proof}

Building on the previous result, we prove a similar result for the operator $\tilde L$ (see also  \cite{MPW} for more details). First we need to define weighted spaces adapted to $\tilde L$.  Given $\delta <0$, we define the weighted space
\[
{\mathbb L}^\infty_\delta ({\bf C}) : =  \left( \sum_{z\in  Z} e^{ \delta \, \sqrt{1+|\cdot - z|^2}}\right) \, L^\infty ({\bf C}),
\]
with the natural associated norm which is defined to be
\[
\| v \|_{{\mathbb L}^\infty_\delta ({\bf C})} : =  \left\| \left( \sum_{z\in  Z} e^{\delta \, \sqrt{1+|\cdot - z|^2}}\right)^{-1} \, v \right\|_{L^\infty (\bf C)}.
\] 
Observe that we could have used the points of $\tilde Z$ instead of the points of $Z$ to define these spaces and this would not have changed anything since the respective norms would have been uniformly equivalent independently of $\ell \gg 1$. 

We define a cutoff function 
\[
\chi (s) :=\left\{
\begin{array}{rllll}
1 \quad \text{if} & \quad s\leq - 1\\[3mm]
0 \quad \text{if} & \quad s\geq 1, 
\end{array}
\right.
\]
and, for all $\bar s >0$ we define  
\[
\chi_{\bar s} (s) : =   \chi \left( s - \bar s \right).
\]
We also define for all $z \in \tilde Z$, the vector field 
\begin{equation}
\Xi_z  :  = \chi_{\ell/4}  (|\cdot -z|) \, \nabla u_0 (\cdot -z),
\label{eq:xiz}
\end{equation}
being understood that we identify vectors in ${\bf R}^2$ with complex numbers. 

The main result of this section reads~:
\begin{proposition}
Assume that $\delta \in (-1,0)$. Then, there exists $\ell_*>0$ (larger than or equal to the one defined in Proposition~\ref{pr:thisisit}) and, for all $\ell \geq \ell_*$, there exists a linear operator 
\[
\tilde G : {\mathbb L}^\infty_\delta ({\bf C})  \to {\mathbb L}^\infty_\delta ({\bf C}) \times {\bf C}^n ,
\]
where $n$ is the cardinal of $Z$, such that, for all $f \in {\mathbb L}^\infty_\delta ({\bf C})$, $\tilde G\, f = : (v , (c_z)_{z \in Z})  $ satisfies 
\[
\tilde L \, v + \sum_{z\in \tilde Z} \langle c_z, \Xi_z \rangle_{\bf C} = f ,
\]
in ${\bf C}$. Moreover, 
\[
\| v \|_{{\mathbb L}^\infty_\delta ({\bf C})} + \sup_{z \in \tilde Z} |c_z|  \leq C \, \| f \|_{{\mathbb L}^\infty_\delta ({\bf C})},
\]
for some constant $C >0$ which does not depend on $f$ and, if, for $i=1,2$, $\tilde G^{(i)}$ is the right inverse corresponding to $\tilde z_j^{[p,q], (i)}$, we have
\begin{equation}
\interleave \tilde G^{(2)} - \tilde G^{(1)} \interleave \leq C \, \sup_{[p,q] \in \mathscr E} \, \sup_{j=1, \ldots, 2m_{[p,q]}-1}   \left| \tilde z_j^{[p,q], (2)} - \tilde z_j^{[p,q], (1)} \right|,
\label{eq:G2mG1}
\end{equation}
for some constant $C>0$.
\label{pr:lana}
\end{proposition}
\begin{proof}
We decompose $f$ as
\[
f = \left( 1- \sum_{z \in Z} \chi_{\ell/4} (|\cdot -z|) \right) \, f + \sum_{z \in Z} \chi_{\ell/4} (|\cdot - z|) \, f
\]
For each $z \in Z$, we use the result of Proposition~\ref{pr:invl0} to solve
\[
L_0 v_z +  \langle c_z, \nabla u_0  \rangle_{\bf C} = \chi_{\ell/4}  \, f(\cdot + z).
\]
We know that we have
\[
\| v_z \|_{L^\infty_\delta (\bf C)} +  |c_z|  \leq C \, \| f \|_{L^\infty_\delta (\bf C ; Z)},
\]
with similar estimates for the first partial derivatives of $v_z$.

Next, we solve
\[
(\Delta -1) v_\infty = \left( 1- \sum_{z \in Z} \chi_{\ell/4} (|\cdot -z|) \right) \, f  -  \sum_{z \in Z} [ÊL_0, \chi_{\ell/2}] \, v_z (\cdot -z),
\]
where $[A,B]$ denotes the commutator of $A$ and $B$.  Since $(\Delta -1) \, 1 =-1$, the maximum principle implies that 
\[
\| v_\infty\|_{L^\infty ({\bf C})} \leq C \, e^{\delta \ell/4}Ê\, \| f\|_{{\mathbb L}^\infty_\delta ({\bf C})}.
\]
Now, observe that, provided $\ell$ is chosen large enough, the function 
\[
v_1 (z) : =  \sum_{z' \in Z} \, e^{\delta \sqrt{1+|z - z'|^2}} ,
\]
satisfies 
\[
(\Delta -1) \, v_1 \leq - \frac{(1 - \delta^2)}{2} \, v_1,
\]
away from the discs of radius $\ell/4$ centered at the points of $Z$ and the maximum principle implies that 
\[
e^{ -\delta \ell/4}  \, \| v_\infty\|_{L^\infty ({\bf C})}  +  \| v_\infty\|_{{\mathbb L}^\infty_\delta ({\bf C})} \leq C \, \| f\|_{{\mathbb L}^\infty_\delta ({\bf C})},
\]
with similar estimates for the first partial derivatives of $v_\infty$.

We then define 
\[
v : =  \left( 1- \sum_{z \in Z} \chi_{\ell/8} (|\cdot -z|) \right)  \, v_\infty + \sum_{z \in Z} \chi_{\ell/2} (|\cdot - z|) \, v_z (\cdot + z).
\]
Using the equations satisfied by $v_z$ and $v_\infty$, one gets
\begin{eqnarray*}
\tilde L v-\sum_{z\in \tilde Z}\langle c_z,\Xi_z\rangle_{\bf C}-f =[\tilde L,\chi_{\ell/8}]v_\infty+3\tilde u^2(1-\chi_{\ell/8})v_\infty+\sum_{z\in\tilde Z}3(\tilde u^2-u^2_z)\chi_{\ell/2}v_z,
\end{eqnarray*}
and, using the estimates satisfied by $v_z$ and $v_\infty$, one checks that 
\[
\left\| \, \tilde L v -\sum_{z\in \tilde Z}  \langle c_z, \Xi_z \rangle_{\bf C}  - f \, \right\|_{{\mathbb L}^\infty_\delta ({\bf C})} \leq C \, e^{-\kappa \ell} \, \| f\|_{{\mathbb L}^\infty_\delta ({\bf C})},
\]
for some $\kappa >0$ and also  that 
\[
\| v \|_{{\mathbb L}^\infty_\delta ({\bf C})} + \sup_{z \in \tilde Z} |c_z|  \leq C \, \| f \|_{{\mathbb L}^\infty_\delta ({\bf C})},
\]
for some constant $C >0$ which does not depend on $f$. The result then follows from a simple perturbation argument, provided $\ell$ is taken large enough.
\end{proof}

\section{Perturbation of the approximate solution}

In this section, we keep the assumptions and notations introduced in \S \ref{se:dsds} and we assume that $\ell \geq \ell_*$ and $\kappa \geq \kappa_*\, \ell^3$ so that the results of the previous sections do hold. The solution to (\ref{eq:nls}) we are looking for has the  form $u=\tilde u + v$, where $v$ is a small function, in a sense to be made precise later on and where $\tilde u$ is defined in (\ref{eq:tildeu}).  We have already defined 
\[
\tilde L := \Delta -1 + 3 \, \tilde u^2 ,
\]
and we now define the error
\[
\tilde E := \Delta \tilde u- \tilde u +\tilde u^3,
\]
as well as the nonlinear functional
\[
\tilde Q (v) := (\tilde u +v)^3 - \tilde u^3-  3 \, \tilde u^2 \, v ,
\]
which, given our nonlinearity simplifies into 
\[
\tilde Q (v) := 3\, \tilde u \, v^2 + v^3 .
\]
With these notations, the solvability of (\ref{eq:nls}) reduces to find  a function $v$ and complex numbers $c_z$, for $z\in \tilde Z$, solutions of the nonlinear problem
\[
\tilde L \, v+ \tilde E + \tilde Q (v) = \sum_{z\in \tilde Z}  \langle c_z, \Xi_z \rangle_{\bf C} .
\]
where $\Xi_z$ has been defined in (\ref{eq:xiz}). Then, we will explain how to find the points $\tilde z_j^{[p,q]} $ as defined in (\ref{eqdsz}) and the forces ${\bf f}_r^p\in {\bf C}$ so that $F_z=0$ for all $z \in \tilde Z$.

For the time being, the main purpose of this section is to prove the~:
\begin{proposition}
There exists $\ell_* >0$ (larger than or equal to the one defined in Proposition~\ref{pr:lana})  such that for all $\ell \geq \ell_*$, there exists $v \in {\mathbb L}^\infty_\delta ({\bf C} ; Z)$ and, for each $z\in \tilde Z$ there exits $F_z \in {\bf C}$ such that the function $u :=  \tilde u + v$ solves
\[
\Delta u -u + u^3 =  \sum_{z\in \tilde Z}  \langle F_z, \Xi_z \rangle_{\bf C},
\]
and 
\[
\| v \|_{{\mathbb L}^\infty_\delta ({\bf C})} + \sup_{z \in \tilde Z} \, |F_z| \leq  C \,  \Upsilon (\ell),
\]
for some constant $C>0$.  Moreover the function $v$  and the vectors $F_z$ depend continuously on the forces ${\bf f}^p_r$ given in the statement of Proposition~\ref{pr:thisisit} and depend smoothly on the points $\tilde z_j^{[p,q]}$ satisfying (\ref{eqdsz}). In particular, if  the function $v^{(i)}$ is the solution corresponding to the points $\tilde z_j^{[p,q], (i)}$, we have
\begin{equation}
\| v^{(2)} - v^{(1)} \|_{{\mathbb L}^\infty_\delta ({\bf C})} \leq C \,  \Upsilon (\ell) \, \sup_{[p,q] \in \mathscr E} \, \sup_{j=1, \ldots, 2m_{[p,q]}-1}   \left| \tilde z_j^{[p,q], (2)} - \tilde z_j^{[p,q], (1)} \right|,
\label{eq:v2mv1}
\end{equation}
for some constant $C>0$.
\label{pr:nla}
\end{proposition}

We begin with the~:
\begin{lemma}
Assume that $\delta \in (-1, 0)$ is fixed. Then, there exists a constant $C_0 >0$, independent of $\ell\geq \ell_*$ and all parameters of the construction, such that 
\[
\| \tilde E\|_{{\mathbb L}^\infty_\delta ({\bf C})} \leq C_0 \, \Upsilon (\ell) .
\]
\end{lemma}
\begin{proof}
We start from the fact that
\[
\tilde E =  \left( \sum_{z' \in \tilde Z} \eta_{z'}  u_0 (\cdot -z')\right)^3 - \sum_{z' \in \tilde Z} \left( \eta_{z'}  u_0 (\cdot -z')\right)^3.
\]

We then estimate $\tilde E$ near a given point $z\in\tilde Z$. In a ball of radius $\ell/2$ centered at $z$, we can write
\[
\tilde E = \left( \eta_{z}  u_0 (\cdot -z) +  \sum_{z' \neq z} \eta_{z'}  u_0 (\cdot -z') \right)^3 - \left( \eta_{z}  u_0 (\cdot -z) \right)^3  - \sum_{z' \neq z} \left( \eta_{z'}  u_0 (\cdot -z')\right)^3, 
\]
and hence, we get
\[
| \tilde E |\leq  C \, \Upsilon(\ell) \, e^{\delta |\cdot-z|} \leq C \, \Upsilon(\ell) \, \sum_{z'\in \tilde Z}e^{\delta |\cdot-z'|} ,
\]
for some constant $C >0$. While, away from the balls of radius $\ell/2$ centered at the points of $\tilde Z$, we take the advantage that $u_0$ decays exponentially fast to $0$ at infinity, to prove that
\begin{eqnarray*}
|\tilde E| \leq  C \sum_{z\in \tilde Z}\ell^{-3/2}e^{-3|\cdot-z|} \leq  C \Upsilon(\ell )\sum_{z\in \tilde Z}e^{\delta |\cdot-z|},
\end{eqnarray*}
for some constant $C >0$. The estimate for $\tilde E$ then follows at once. Observe that the estimate is achieved near the points of $\tilde Z$. 
\end{proof}

We will also need the 
\begin{lemma}
Assume that $\delta \in (-1, 0)$ is fixed. Then, there exists a constant $C_1 >0$, independent of $\ell\geq \ell_*$ and all parameters of the construction, such that 
\[
\| \tilde Q(v') - \tilde Q (v) \|_{{\mathbb L}^\infty_\delta ({\bf C})} \leq C_1 \,  \Upsilon (\ell) \, \| v'-v \|_{{\mathbb L}^\infty_\delta ({\bf C})},
\]
provided $\| v' \|_{{\mathbb L}^\infty_\delta ({\bf C})} \leq  2 \, C_0 \,  \Upsilon (\ell)$,
\end{lemma}
\begin{proof}
The estimate follows from the expression 
\[
\tilde Q (v) =  v^3 + 3 \tilde u \, v^2,
\]
we leave the details to the reader. 
\end{proof}

The result of Proposition~\ref{pr:nla} then follows from these two results, the result of Proposition~\ref{pr:lana} and a simple application of a fixed point theorem for contraction mappings in the closed ball of radius $2 \, C_0 \,  \Upsilon (\ell)$ in ${\mathbb L}^\infty_\delta ({\bf C})$, provided $\ell$ is chosen large enough. Proofs with all details are given in \cite{MPW}. The estimate (\ref{eq:v2mv1}) follows from taking the difference between the equations satisfied by the two solutions and using (\ref{eq:G2mG1}).

\section{Projection of the error}

Again, we keep the assumptions and notations introduced in \S \ref{se:dsds} and we assume that $\ell \geq \ell_*$ and $\kappa \geq \kappa_*\, \ell^3$ so that the results of the previous sections do hold. As explained in the introduction, we now give the expansion of the vectors $F_z$ as $\ell$ tends to infinity.  In the above statements, quantities of the form $ {\mathcal O} (e^{- \gamma \ell})$ depend continuously on the forces ${\bf f}^p_r$ and depend smoothly on the points $\tilde z_j^{[p,q]}$.

We start with the general~:
\begin{lemma}
There exists $\gamma_1 >0$ such that, for all $z \in \tilde Z$, we have 
\[
F_z =  - C_* \, \sum_{z' \in \tilde N_z} \eta_{z'}  \, \Upsilon (|z'-z|) \, \frac{z'-z}{|z'-z|} + \Upsilon (\ell) \,  {\mathcal O} (e^{- \gamma_1 \ell}) ,
\]
where $\tilde N_z$ denotes the set of closest neighbors of $z$ in $\tilde Z$ and $C_* >0$ is explicitly given by
\[
\frac{1}{C_*}  : =  \iint_{\bf C} | \partial_x u_0|^2 \, dx \, dy.
\]
\label{le:mlml}
\end{lemma}
\begin{proof}
We start from the fact that, by construction, the solution $u$ given by the result of Proposition~\ref{pr:nla} can be decomposed as $u= \tilde u +v$ where $\tilde u$ is defined in (\ref{eq:tildeu}) and where $v$ is a solution of
\[
\tilde L \, v+ \tilde E + \tilde Q (v) = \sum_{z'\in \tilde Z}  \langle F_{z'}, \Xi_{z'} \rangle_{\bf C} .
\]
To obtain the expansion of $F_z$, it is enough to integrate the above equation against $\Xi_z$, for some given $z \in \tilde Z$.  One immediately gets from Proposition~\ref{pr:nla}, that there exists $\gamma >0$ such that 
\[
\iint_{\bf C} \tilde Q (v) \, \langle  c ,  \Xi_z \rangle_{\bf C} \, dx \, dy = \Upsilon (\ell) \, \mathcal O ( e^{- \gamma \ell } ).
\]
for any unit vector $c \in {\bf C}$. Next, an integration by parts leads to
\[
\iint_{\bf C} \tilde L v \,  \langle  c ,  \Xi_z \rangle_{\bf C} \, dx \, dy = \iint_{\bf C}  v \,  \tilde L \langle  c ,  \Xi_z \rangle_{\bf C} \, dx \, dy  .
\]
Since $L_0 \langle  c ,  \nabla u_0 \rangle_{\bf C} =0$, we can write 
\[
\begin{array}{rllll}
 \displaystyle  \iint_{\bf C} \tilde L v \,  \langle  c ,  \Xi_z \rangle_{\bf C} \, dx \, dy  & = & 3 \,  \displaystyle  \iint_{\bf C}  v \, (\tilde u^2 - u_0^2 (\cdot -z)) \,  \langle  c ,  \chi_{\ell/4} \, \nabla u_0 (\cdot -z ) \rangle_{\bf C} \, dx \, dy  \\[3mm]
&  + & \displaystyle \iint_{\bf C}  v \,  \tilde L \left( (1- \chi_{\ell/4}) \, \langle  c ,  \nabla u_0 (\cdot -z) \rangle_{\bf C} \right) \, dx \, dy  ,
\end{array}
\]
and it is then easy to conclude that there exists $\gamma >0$ such that
\[
\iint_{\bf C} \tilde L v \,  \langle  c ,  \Xi_z \rangle_{\bf C} \, dx \, dy =  \Upsilon (\ell) \, \mathcal O ( e^{- \gamma \ell } ) ,
\]
for any unit vector $c \in {\bf C}$.  

Finally, to estimate the last term, we write
\[
\tilde E =  \left( \sum_{z' \in \tilde Z} \eta_{z'}  u_0 (\cdot -z')\right)^3 - \sum_{z' \in \tilde Z} \left( \eta_{z'}  u_0 (\cdot -z')\right)^3
\]
Since $\Xi_z$ is supported in the disc of radius $\ell/4+1$, centered at $z$, we distinguish the closest neighbors of $z$ and the other points of $Z$. Hence, we can write
\[
\tilde E =  3 \, \, u_0^2 (\cdot -z) \, \sum_{z' \in \tilde N_z} \eta_{z'}  u_0 (\cdot -z')  + \Upsilon (\ell) \, \mathcal O ( e^{- \gamma \ell } ) ,
\]
in $D(z, \ell/2)$, for some $\gamma >0$. The result then follows from the definition of $\Upsilon$. Then $\gamma_1$ in the statement of the result is the least of the $\gamma$ which appear in the above estimates.
\end{proof}

There are two different consequences according to whether $z \in \tilde Z$ is one of the vertices of $z^p_r$ for some $r \in \mathscr V^p$ or one of the $\tilde z^{[p,q]}_j$ for some $[p,q] \in \mathscr V$ and some $j=1, \ldots , 2m_{[p,q]}-1$. In the former case, we have~:
\begin{corollary}
There exists $\gamma_1 >0$ such that, if $z \in \tilde Z$ is one of the $z^p_r$ for some $r \in \mathscr V^p$ and some $p \in \mathscr V$, then 
\[
F_z = - C_* \, \eta_z \, \Upsilon (\ell) \, \left( {\bf f}_r^p  + \frac{{\bf e} + i t \, p}{n_p}\right)  +\Upsilon (\ell) \, \left( \mathcal O ( e^{- \gamma_0 \ell }) + \mathcal O ( e^{- \gamma_1 \ell } ) \right) . 
\]
\end{corollary}
Observe that, in this expansion, according to the result of Proposition~\ref{pr:thisisit}, the ${\bf f}^p_r$ are vectors which can be prescribed arbitrarily while ${\bf e} \in {\bf C}$ and $t \in {\bf R}$ cannot be prescribed. Also, $n_p$ is the number of vertices of $\mathscr V^p$.  

Now, when $z \in \tilde Z$ is one of the $z = z^{[p,q]}_j$ for some $[p,q] \in \mathscr V$, then, because of (\ref{eq:balj}) and (\ref{eqdsz}), the estimate in Lemma~\ref{le:mlml} reduces to
\[
F_z = \Upsilon(\ell) \, \left( \mathcal O (e^{- \gamma_0  \ell}) + \mathcal O (e^{- \gamma_1 \ell}) \right),
\]
where $\gamma_0$ is the constant used in (\ref{eqdsz}). Hence, in this case we need to be more precise and expand the first term in the estimate of Lemma~\ref{le:mlml}.

Recall that we have defined in section 5.2
\[
{\bf e}_{pq} : = \frac{r^p_q-r^q_p}{|r^p_q-r^q_p|} .
\]
We decompose
\[
\tilde z^{[p,q]}_j - z^{[p,q]}_j =  \dot z_j   \,  {\bf e}_{pq} ,
\]
where $\dot z_j \in {\bf C}$.  We set $\dot z_0 = \dot z_{2m_{[p,q]}} =0$ in agreement with the fact that we do not want to modify the end points $r^p_q$ and $r^q_p$. Finally, we set 
\[
\ell_{[p,q]} : = \ell \, (1- \alpha_\ell (a_{[p,q]})).
\]
Then we have the~:
\begin{corollary}
There exists $\gamma_1 >0$ and $C >0$, such that, if we assume that $z \in \tilde Z$ is one of the $\tilde z^{[p,q]}_j$ for some $[p,q] \in \mathscr V$ and some $j=1, \ldots , 2m_{[p,q]}-1$, then 
\[
\begin{array}{rllll}
F_z & = &  \pm \displaystyle C_* \,  \left(  \Upsilon'(\ell_{[p,q]}) \, \Re \, \left( \dot z_{j+1} - 2 \dot z_j + \dot z_{j-1} \right)  + i \, \frac{\Upsilon (\ell_{[p,q]})}{\ell_{[p,q]}} \,  \Im \, \left( \dot z_{j+1} - 2 \dot z_j + \dot z_{j-1} \right)  \right) \, {\bf e}_{pq}  \\[3mm]
&   +  & \Upsilon (\ell) \, \left(  \mathcal O ( e^{- \gamma_1 \ell } ) +  \mathcal O ( e^{-2 \gamma_0 \ell } )\right) ,
\end{array}
\]
where the $\pm$ depends on the sign of $\eta_{z'}$ where $z'$ is one of the closest neighbors of $z$ in $\tilde Z$. 
\end{corollary}
\begin{proof}
Observe that $z$ has only two closest neighbors which we denote by $z' = \tilde z^{[p,q]}_{j-1}$ and $z'' : = \tilde z^{[p,q]}_{j+1}$. According to Lemma~\ref{le:mlml}, we have 
\[
F_z = -  \displaystyle C_* \,  \left( \eta_{z'} \, \Upsilon (|z'-z|) \, \frac{z'-z}{|z'-z|} + \eta_{z''} \, \Upsilon (|z''-z|) \, \frac{z''-z}{|z''-z|} \right) + \Upsilon (\ell) \, \mathcal O ( e^{- \gamma_1 \ell } ) .
\]
The result follows at once from the expansions of $\Upsilon$ given in (\ref{eq:upsilon1}) and (\ref{eq:upsilon2}), the $\pm$ which appears in the statement of the Lemma depends on the sign of $\eta_{z'}$. For a more detailed proof of this expansion, we refer to \cite{MPW}, Section 5.
\end{proof}

As a consequence, the set of equations $F_z =0$, for $z= \tilde z_1^{[p,q]}, \ldots ,  \tilde z_{2m_{[p,q]}-1}^{[p,q]}$, reduces to solving a system of the form
\[
\dot z_{j+1} - 2 \dot z_j + \dot z_{j-1}   =  {\mathcal O} (\ell \, e^{- \gamma_1 \ell}) +  {\mathcal O} (\ell \, e^{- 2\gamma_0 \ell}) ,
\]
where we recall that, by assumption, $\dot z_0 = \dot z_{2m_{[p,q]}} =0$.

For all $m \geq 2$, we define the $m \times m$ matrix
\begin{equation}
\label{defA0} {\mathbb T} : =
\begin{pmatrix}
2 & - 1 &  0 & \ldots  & 0 \\
-1 & 2 & \ddots  & \ddots &\vdots \\
0& \ddots    &  \ddots    & \ddots  &0 \\
\vdots    & \ddots & \ddots & 2 &  -1  \\
 0 & \ldots &  0&-1& 2 \\
\end{pmatrix}
\in M_{m \times m}  . 
\end{equation}
It is easy to check that the inverse of $\mathbb T$ is the matrix ${\mathbb T}^{-1}$ whose entries are given by
\[
{\mathbb T}^{ij} : =   \min (i,j) - \frac{ij}{m+1}   . 
\]
Hence, the above system of equation can also be written as 
\[
\dot z_j =  {\mathcal O} (m \, \ell \, e^{- \gamma_1 \ell})  + {\mathcal O} (m \, \ell \, e^{- 2 \, \gamma_0 \ell}) .
\]
where $$m := \max_{[p,q] \in \mathscr E} m_{[p,q]}.$$

We choose 
\[
\gamma_0  :=  \gamma_1/4.
\]
As a consequence, it is easy to apply a fixed point theorem for contraction mappings to prove the~:
\begin{proposition}
\label{pr:z}
There exists $\ell_* >0$ (larger than or equal to the $\ell_*$ which appears in Proposition~\ref{pr:nla}) such that if $\ell \geq \ell_*$ and if $m \leq e^{\gamma_1 \ell/4}$, there exist $\tilde z^{[p,q]}_j$, for $[p,q] \in \mathscr E$ and $j=1, \ldots, 2m_{[p,q]}-1$, such that 
\[
F_z =0,   
\]
for all $z \in \tilde Z$ of the form $\tilde z^{[p,q]}_j$ for some $[p,q] \in \mathscr E$ and some $j=1, \ldots, 2m_{[p,q]}-1$.
Moreover, 
\[
\left| \tilde z^{[p,q]}_j - z^{[p,q]}_j \right| \leq e^{-\gamma_0 \ell}.
\]
and the $\tilde z^{[p,q]}_j$ depend continuously on the ${\bf f}^p_r$. 
\end{proposition}
\begin{proof}
It is enough to choose $\gamma_0 >0$ close enough to $0$. This result is then a consequence of a fixed point theorem for contraction mappings. 
\end{proof}

\section{The existence of infinitely many solutions of (\ref{eq:nls})}

As usual, we keep the assumptions and notations introduced in \S \ref{se:dsds} and we assume that $\ell \geq \ell_*$ and $\kappa \geq \kappa_*\, \ell^3$ so that the results of the previous sections do hold.  Building on the previous analysis, we prove the~:
\begin{proposition}
There exist ${\bf f}^p_r$ for all $p \in \mathscr V$ and all $r \in \mathscr V^p$ and there exists ${\bf e} \in {\bf C}$ and $t \in \bf R$, such that 
\[
F_z =C_* \, \eta_z \, \Upsilon (\ell) \,  \left( \frac{{\bf e} + i t \, p}{n_p}\right).
\]
for all $z \in \tilde Z$ of the form $z^p_r \in \mathscr V^p$ for some $p \in \mathscr V$. Moreover, 
\[
| {\bf f}^p_r| \leq e^{-\gamma_3 \ell} ,
\]
for some constant $\gamma_3 >0$. 
\end{proposition}
\begin{proof}
This result is just a consequence of Brower's fixed point theorem.
\end{proof}

To complete the proof of the existence of a solution of (\ref{eq:nls}) close to $\tilde u$ given by (\ref{eq:tildeu}), we use a Pohozaev type argument. To explain this, let us assume that the function $u$ solves 
\begin{equation}
\Delta u - u + u^3 =  f,
\label{eq:deltau=f}
\end{equation}
in ${\bf C}$ and further assume that both $u$ and $f$ are tending to $0$ exponentially fast at infinity. Then, we have the following result which is a consequence of Pohozaev identity.
\begin{lemma}
Given any Killing vector field $\Xi$  (i.e. a vector field which generates a group of isometries of $\bf C$), the following identity holds 
\begin{equation}
\iint_{\bf C} \langle \Xi , \nabla  u\rangle_{\bf C} \,  f \, dx \, dy =0  .
\label{eq:xi}
\end{equation}
\end{lemma}
\begin{proof} Multiplying (\ref{eq:deltau=f}) by $\Xi \cdot \nabla u$ and using simple manipulations, we get 
\[
\mbox{div} \left(  ( \Xi \cdot \nabla u) \, \nabla u - {1 \over 2} \, (|\nabla u|^2 + u^2)  \, \Xi + {1 \over 4} \, u^4  \, \Xi
\right) =   \langle \Xi , \nabla u \rangle_{\bf C} \,  f , 
\]
Then, the divergence theorem implies that
\[
\iint_{\bf C}  \langle \Xi , \nabla u \rangle_{\bf C} \,  f  \, dx \, dy =0  , 
\]
provided $u$ and $f$ decay fast enough at infinity.
\end{proof}

In our case, 
\[
f  : = C_* \, \sum_{p \in \mathscr V} \sum_{r \in \mathscr V^p} \eta_r \, \Upsilon (\ell) \,  \chi_{\ell/4}Ê\, \left\langle \frac{{\bf e} + i \, t \, p}{n_p} , \nabla  u_0 (\cdot - r) \right\rangle_{\bf C}.
\]
Plugging this expression into (\ref{eq:xi}), one concludes that ${\bf e} =0$ and $t=0$ provided $\ell $ is chosen large enough. 

Let us describe the general existence result we have obtained. 
\begin{theorem}
Assume that $(\mathring {\mathscr N} , \mathring a )$ is a  closable, flexible network and further assume that, for each $p \in \mathring {\mathscr V}$, there exists a flexible unitary network  $( \mathring{\mathscr N}^p , \mathring a^p )$ such that properties (i)-(vii) in \S~\ref{se:dsds} are fulfilled. Then, there exists $\ell* >0$ and $\kappa_* >0$ such that, for all $\ell \geq \ell_*$ and $\kappa \geq \kappa_* \, \ell^3$, there exist a network $({\mathscr N}, a)$ and subnetworks $({\mathscr N}^p, a^p)$ and a solution of (\ref{eq:nls}) which is close to the approximate solution $\tilde u$ defined in (\ref{eq:tildeu}). 
\label{th:general-result}\end{theorem}

\begin{remark}
Observe that, in our construction, we need to assume that the integers $m_{[p,q]}$ which appear in (\ref{eq:mpq}) do satisfy
\[
\ell^3 \ll \max_{[p,q] \in \mathscr E} m_{[p,q]}  \ll e^{\gamma \ell}
\]
for some $\gamma >0$. The inequality on the left comes from Proposition~\ref{pr:net-7.44} while the inequality on the right comes from Proposition~\ref{pr:z}. The constraint $m_{[p,q]} \ll e^{\gamma \ell}$ is purely technical and can be removed in the case where one is dealing with (\ref{eq:nls}), however, for other applications it is not clear that this constraint can be removed.  
\end{remark}

\section{Examples}

We give here some examples of balanced, closable networks which can be used in the construction. In particular, this will complete the proof of Theorem~\ref{th:main}. Checking the flexibility of such networks is not so difficult. However, checking whether such a network is closable or not might be a complicated task which have to be done using for example Mathematica. 

\medskip

\noindent {\bf Example 10.1 :} \label{ex:10.3}  An interesting example with symmetry group of order $3$ is given by the following.  Given $0 < \theta  <\pi/4$, the  set of vertices of the network $\mathscr N_{V}$ is given by
\[
\mathscr V_{V} : =  \{ 0 , \tan \theta  , - \tan \theta ,  i \,   ,  - i  \} ,
\]
while its set of edges is defined to be
\[
\begin{array}{rllll}
\mathscr E_{V} & := & \Big\{ [0, \tan \theta ],  [ 0  ,  i ] ,   [ 0  , - \tan \theta ],   [ 0  ,-  i ] ,   [\tan \theta  ,  i ],  [ i,  -\tan \theta  ] , \\[3mm]
&  & \hspace{65mm}  [ -\tan \theta  , - i ],  [ - i  ,  \tan \theta ] \Big\} .
\end{array}
\]
Observe that the network is invariant under the symmetries with respect to the $x$-axis and the $y$-axis.  

\begin{center}
\includegraphics[width=4cm]{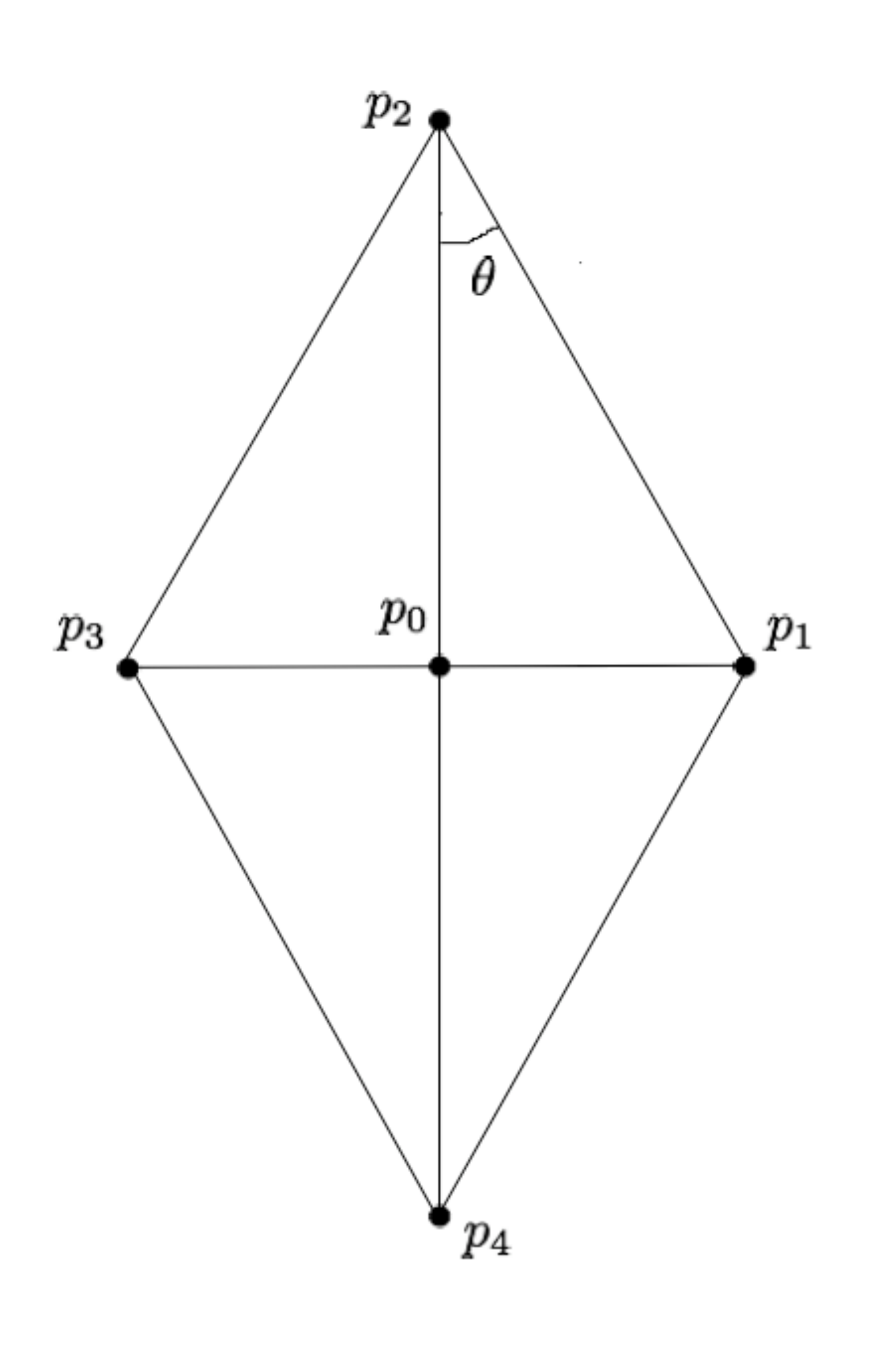}\\
Fig 13 : The network $\mathscr N_V$.
\end{center}

We define the weight function $a : \mathscr E_V \to {\bf R}-\{0\}$ by
\[
a_{[0, \tan \theta ]} = a_{[0, - \tan \theta ]} := - 2 \, \sin \theta,  \qquad  a_{[0, i]} =  a_{[0 , - i]} := - Ê2 \, \cos \theta, 
\]
and
\[
a_{[\tan \theta , i]} = a_{[i ,  - \tan \theta ]} = a_{[-\tan \theta , - i]} = a_{[-i, \tan \theta ]} :=  1 .
\]
It is easy to check that the network $(\mathscr N_V, a)$ is balanced. We also have the 
\begin{lemma}
The balanced network $(\mathscr N_V, a)$ is flexible in the sense of Definition~\ref{de:net-3.2} and closable in the sense of Definition~\ref{de:clos}.
\label{le:clso}
\end{lemma}
\begin{proof}
In this example $m=8$ and $n=5$ and hence $2m= n-2$. Therefore, to prove that the network is flexible, it is enough to show that ${\rm D}_a {\bf F}_{({\rm Id}, a)}$ has rank equal to $7$.

Let us assume that ${\rm D}_a {\bf F}_{({\rm Id},a)}(\dot{a})=0$ and also that $\dot{a}_{[p_1,p_2]}=0$. Then looking at the component of ${\rm D}_a {\bf F}_{({\rm Id},a)}$ at $p_1$, we find that
\begin{equation}
\dot{a}_{[p_0,p_1]}\frac{p_0-p_1}{|p_0-p_1|}+\dot{a}_{[p_1,p_4]}\frac{p_4-p_1}{|p_4-p_1|}=0.
\end{equation}
Since  $p_0-p_1$ and $p_4-p_1$ are not $\bf R$-collinear, we conclude that $\dot{a}_{[p_0,p_1]}=\dot{a}_{[p_1,p_4]}=0$. Then one looks at the component of ${\rm D}_a {\bf F}_{({\rm Id},a)}$ at $p_2$, we have
\begin{equation}
\dot{a}_{[p_0,p_2]}\frac{p_0-p_2}{|p_0-p_2|}+\dot{a}_{[p_2,p_{3}]}\frac{p_{3}-p_2}{|p_{3}-p_2|} =0.
\end{equation}
Since  $p_0-p_2$ and $p_2-p_3$ are not $\bf R$-collinear, we conclude that $\dot{a}_{[p_0,p_2]} = \dot{a}_{[p_2,p_3]} =0$. Arguing similarly at $p_3$ and $p_4$, we conclude that $\dot a =0$ and hence ${\rm D}_a {\bf F}_{({\rm Id},a)}$, restricted to the hyperplane $\dot{a}_{[p_1,p_2]}=0$, is injective. Therefore, this map has rank at least equal to $m-1 =7$. By Proposition \ref{pr:net-3.1}, this shows that the network is flexible.

Now, it remains to check that the network is closable. This amounts to check that the image of  ${\rm D} {\bf L}_{\rm Id}$ does not contain the vector ${\bf T}$. Namely, that the only solution of 
\[
{\rm D} {\bf L}_{\rm Id}(\dot \Phi) = \lambda \, {\bf T},
\]
is $\lambda =0$ and $\dot \Phi =0$. Writing $p_i- p_0 = z_i$ and $\dot \Phi_{p_i} - \dot \Phi_{p_0} = \dot z_i$, this amount to check that the only solution to
\[
\left\{ \begin{array}{rllll}
\langle 1, \dot z_1\rangle_{\bf C} & = & \lambda \, \tan \theta \, \ln (2 \sin \theta)\\[3mm]
\langle i, \dot z_2\rangle_{\bf C} & = &  \lambda \, \ln (2 \cos \theta)\\[3mm]
\langle 1, \dot z_3\rangle_{\bf C} & = &  - \lambda \tan \theta \, \ln (2 \sin \theta) \\[3mm]
\langle 1, \dot z_4\rangle_{\bf C} & = & - \lambda \, \ln (2 \cos \theta)
\end{array}
\right. \hspace{20mm}
\left\{ \begin{array}{rllll}
\langle z_2 -z_1 , \dot z_2-\dot z_1 \rangle_{\bf C} & = & 0 \\[3mm] 
\langle z_3 -z_2 , \dot z_3-\dot z_2 \rangle_{\bf C} & = & 0 \\[3mm]
\langle z_4 -z_3 , \dot z_4-\dot z_3 \rangle_{\bf C} & = & 0  \\[3mm]
\langle z_1 -z_4 , \dot z_1-\dot z_4 \rangle_{\bf C} & = & 0
\end{array}
\right.
\]
is given by $\dot z_j=0$ for $j=1, \ldots, 4$ and $\lambda =0$.  Using the second system together with the fact that $z_3 =-z_1$ and $z_2=-z_3$, we get 
\[
\langle z_2 -z_1 , \dot z_1 + \dot z_3 -\dot z_2 - \dot z_4 \rangle_{\bf C} = \langle z_3 -z_2 , \dot z_1 + \dot z_3 -\dot z_2 - \dot z_4 \rangle_{\bf C} = 0 
\]
and, since $z_2-z_1$ and $z_3-z_2$ are ${\bf R}$-independent, we conclude that 
\begin{equation}
\dot z_1 + \dot z_3 = \dot z_2 + \dot z_4.
\label{56}
\end{equation}

Using the first system, we get 
\[
\langle 1, \dot z_1 + \dot z_3 \rangle_{\bf C}  =  \langle i, \dot z_2 + \dot z_4 \rangle_{\bf C} =0
\]
and, together with (\ref{56}), this implies that $\dot z_1 + \dot z_3 = \dot z_2 + \dot z_4=0$. Using this information back into the system yields
\[
\left\{ \begin{array}{rllll}
\langle 1, \dot z_1\rangle_{\bf C} & = & \lambda \, \tan \theta \, \ln (2 \sin \theta)\\[3mm]
\langle i, \dot z_2\rangle_{\bf C} & = &  \lambda \, \ln (2 \cos \theta)\\[3mm]
\langle z_2 -z_1 , \dot z_2-\dot z_1 \rangle_{\bf C} & = & 0 \\[3mm] 
\langle z_3 -z_2 , \dot z_2+\dot z_1 \rangle_{\bf C} & = & 0 \\[3mm]
\end{array}
\right.
\]
Since $z_1-z_2 = \tan \theta - i$ and $z_2-z_3 = \tan \theta +i$, the sum of the last two equations implies that 
\[
\tan \theta \, \langle 1, \dot z_1\rangle_{\bf C}  + \langle i, \dot z_2\rangle_{\bf C} = 0,
\] 
using the first two equations, we conclude that ${\bf T}$ is not in the image of ${\rm D} {\bf L}_{\rm Id}$ unless 
\[
 \sin^2 \theta \, \ln (2 \sin \theta) + \cos^2 \theta \, \ln (2 \cos \theta) =0,
\]
which never happens. So, by Definition \ref{de:clos}, the network is closable.
\end{proof}

In the next picture we illustrate the variety of applications of our construction. We start from the network $\mathscr N_V$ and assume that $\theta \in (0, \pi/4]$. We give examples of subnetworks which can be used at the vertices $p_0$, $p_1$ and $p_2$ (similar subnetworks can of course be constructed at the other vertices). The color code is the one we have already used with copies of $+ u_0$ centered at the blue points and copies of $-u_0$ centered at the red points. 

\begin{center}
\includegraphics[width=10cm]{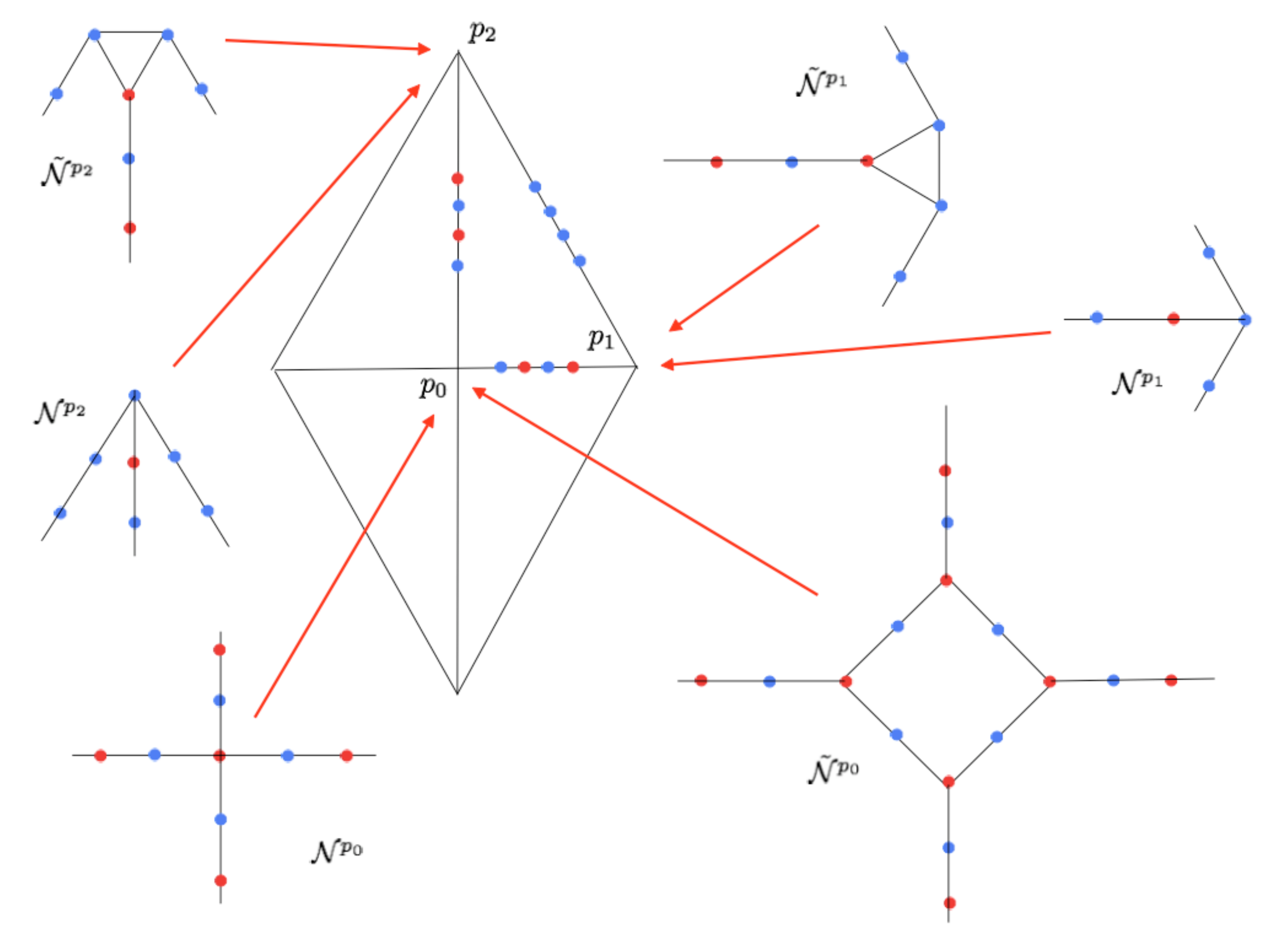}\\
Fig 14 : The network $\mathscr N_V$ with possible subnetworks which can be used at the vertices $p_0$, $p_1$ and $p_2$.
\end{center}

Some comments are due. First observe that the signs of the different subnetworks are compatible with the signs of $(\mathscr N_V, a)$ (see (vii) in the list of properties a subnetwork should fulfill). Let us now concentrated on the subnetworks we can insert at $p_2$. There are two possibilities : either $\mathscr N^{p_2}$ or $\tilde{\mathscr N}^{p_2}$. Observe that one can only use $\mathscr N^{p_2}$ when $\theta > \pi/3$ since otherwise property (vi) is not fulfilled for this subnetwork. For the same reason, $\tilde{\mathscr N}^{p_2}$, which is the unbalanced network described in Example 3.3, can only be used when $\theta >\pi/6$. Analyzing the situation at $p_1$, we see that we can use $\mathscr N^{p_1}$ if $\pi/2 - \theta > \pi/3$ and we can use $\tilde{\mathscr N^{p_1}}$ if $\pi/2 - \theta > \pi/6$.  Finally, we concentrate on the subnetworks which can be used at the point $p_0$. Here, independently of the value of $\theta$,  one can make use of  $\mathscr N^{p_0}$ or one can make use of $\tilde{\mathscr N}^{p_0}$ which is the unbalanced network described in Example 3.2 (namely the network $\mathscr N_{Pol,k}$ for any even integer $k$). To summarize, given the zoology of subnetworks we have at our disposal, we need to restrict 
\[
\theta \in (\pi/6, \pi/3).
\]
But there are certainly infinitely many other choices of subnetworks one can use. 

\medskip

\noindent {\bf Example 10.2 :} \label{ex:3.5} Here is an example of balanced network for which $m < 2n-2$.  Given $0 < \nu < \mu$, the  set of vertices of the network $\mathscr N_{Y}$ is given by
\[
\mathscr V_{Y} : =  \{\mu  + i  , \mu  - i , - \mu  + i  ,  -\mu  - i  ,   \nu ,  -\nu  \} ,
\]
while its set of edges is defined to be
\[
\begin{array}{rllll}
\mathscr E_{Y} & := & \Big\{ [-\nu, \nu],  [ \nu  ,\mu + i ] ,   [ \nu  ,\mu - i ],   [ - \nu  ,- \mu + i ] ,   [- \nu  , -\mu - i ],  [ \mu + i,  \mu - i ] ,  
\\[3mm]
&  & \qquad \qquad  [ -\mu + i  , - \mu + i ],  [ -\mu + i  ,  \mu  + i ],  [ -\mu - i  ,  \mu - i ] \Big\} .
\end{array}
\]
Observe that the network is invariant under the symmetries with respect to the $x$-axis and the $y$-axis.  

\begin{center}
\includegraphics[width=8cm]{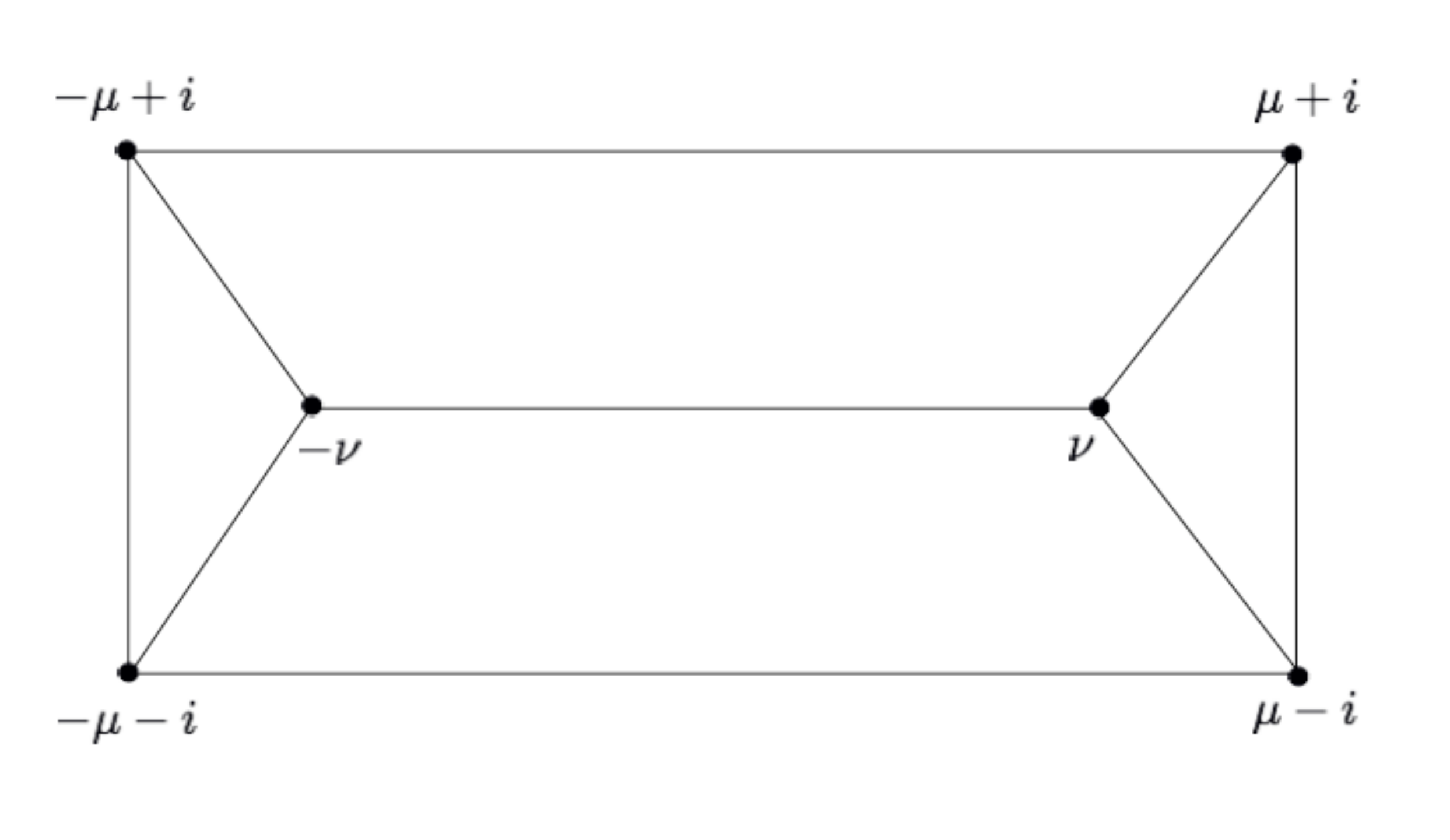}\\
Fig 15 : The network $\mathscr N_Y$.
\end{center}

We define the weight function $a : \mathscr E_Y \to {\bf R}-\{0\}$ by
\[
a_{[-\nu , \nu ]} :=2 \, \cos \theta, \qquad  a_{[\nu, \mu +i]} = a_{[\nu, \mu +i]} = a_{[-\nu, -\mu +i]} = a_{[-\nu, -\mu +i]} := 1
\]
\[
a_{[\mu +i, \mu-i]} =  a_{[-\mu +i, -\mu - i]} := Ê\sin \theta, \qquad \mbox{and} \qquad  a_{[-\mu +i, \mu+i]} =  a_{[-\mu -i, \mu - i]} := \cos \theta ,
\]
where 
\begin{equation}
\cos \theta : =  \frac{\mu-\nu}{\sqrt{1+ (\mu-\nu)^2}} \qquad \text{and} \qquad \sin \theta : = \frac{1}{\sqrt{1+ (\mu-\nu)^2}}.
\label{eq:cs}
\end{equation}
It is easy to check that the network $(\mathscr N_Y, a)$ is balanced. We also have the 
\begin{lemma}
The balanced network $(\mathscr N_Y, a)$ is flexible in the sense of Definition~\ref{de:net-3.2} and closable in the sense of Definition~\ref{de:clos}.
\label{le:net-5.2}
\end{lemma}
\begin{proof}
In this example, $m= 9$ and $n=6$ and hence we need to show that the rank of $\Lambda$ is equal to $2n+m-4 = 17$. We set
\[
z_1 : = \mu  + i  , \quad   z_2 : = \mu  - i ,  \quad z_3 : = - \mu  + i  , \quad  z_4 : = -\mu  - i  , \quad  z_0 : = \nu ,  \quad \hat z_0 : = -\nu  .
\]

Since we know that $\Lambda$ has at least a kernel of dimension $4$ spanned by the infinitesimal translations, infinitesimal rotation and the dilation of the weight function, we can assume that we only consider perturbations such that  
\begin{equation}
\dot \Phi_{z_0} = - \dot \Phi_{\hat z_0} \in {\bf R},
\label{eq:jh}
\end{equation}
which takes care of the invariance with respect to translations and rotations, and
\begin{equation}
\hat a_{[z_0, \hat z_0]} =0,
\label{eq:jk}
\end{equation}
which takes care of the invariance with respect to dilations of $a$. For such perturbations, we need to show that $\Lambda$ is injective. So, let us assume that $\dot \phi$ and $\dot a$ are chosen is such a way that 
\[
\Lambda (\dot \Phi, \dot a) =0.
\]
and also that (\ref{eq:jh}) and (\ref{eq:jk}) do hold. We adopt the notations
\[
\dot z_j : =  \dot \Phi_{z_j},
\]
and $\dot{\hat z}_0 : =  \dot \Phi_{\hat z_0}$.

We first exploit the fact that ${\rm D} {\bf L} (\dot \Phi) =0$. Looking at the component of ${\rm D} {\bf L} (\dot \Phi) $ corresponding to the edge $[\hat z_0; z_0]$, we get
\[
\langle \hat z_0 - z_0 ,  \dot{\hat z}_0 - \hat z_0\rangle_{\bf C} =0.
\]
Because of (\ref{eq:jh}), we conclude that $\dot z_0 = \dot {\hat z}_0 =0$. Looking at the component of ${\rm D} {\bf L} (\dot \Phi) $ at $[z_0, z_{1}]$, we get
\[
\langle z_1 - z_0 ,  \dot z_1  \rangle_{\bf C} =0.
\]
and hence, there exists $x_1 \in \bf R$ such that 
\[
\dot z_1  = i \, x_1 \, (z_1 -z_0).
\]
similarly, we find that $\dot z_2  = i \, x_2 \, (z_2 -z_0)$, $\dot z_3  = i \, x_3 \, (z_3 -\hat z_0)$ and $\dot z_4  = i \, x_4 \, (z_4 - \hat z_0)$, for some $x_2, x_3, x_4 \in {\bf R}$. 

Looking now at the component of ${\rm D} {\bf L} (\dot \Phi) $ corresponding to the edge $[z_1, z_2]$, we get 
\[
\langle z_2 -z_1, \dot z_2 - \dot z_1\rangle_{\bf C}=0,
\]
pluging into this identity the information we already have concerning $\dot z_2$ and $\dot z_1$ and using the expression for $z_2$ and $z_2$, on gets
\[
\langle -2 i,  i \, x_2\,  (\mu-\nu -i) - i\, x_1 \, (\mu-\nu+i) \rangle_{\bf C}=0,
\]
and this implies that that $x_2=x_1$. Arguing similarly with the edges $[z_3,z_4], \ldots , [z_4, z_1]$, we conclude that $x_1=x_2=x_3=x_4$. Let us call by $x$ this common value. 

We now exploit the fact that ${\rm D} F_{({\rm Id}, a)} =0$. Summing the components corresponding to $z_0, z_1$ and $z_4$ we get
\[
\begin{array}{llll}
\displaystyle a_{[z_2,z_1]} \, \left( \frac{\dot z_2 - \dot z_1}{|z_2-z_1|} -\frac{\langle z_2-z_1, \dot z_2-\dot z_1\rangle_{\bf C}}{|z_2-z_1|^2}\right) + \dot a_{[z_2,z_1]} \frac{z_2 - z_1}{|z_2-z_1|}\\[3mm]
\hspace{10mm}\displaystyle a_{[z_3,z_4]} \, \left( \frac{\dot z_3 - \dot z_4}{|z_3-z_4|} -\frac{\langle z_3-z_4, \dot z_3-\dot z_4\rangle_{\bf C}}{|z_3-z_4|^2}\right) + \dot a_{[z_3,z_4]} \frac{z_3 - z_4}{|z_3-z_4|} =0
\end{array}
\]
Using the information we already have on the $\dot z_j$ and using the fact that $z_2-z_1 = z_3-z_4$, we conclude that
\[
\left(  \dot a_{[z_2,z_1]} + \dot a_{[z_3,z_4]} + i \, 2  \, x  \right) \, \frac{z_2 - z_1}{|z_2-z_1|} = 0.
\]
and we conclude that  $x =0$. Therefore, we have proven that $\dot \Phi =0$.

The proof now proceeds as in the proof of Lemma~\ref{le:clso}. For example, looking at the component of ${\rm D} F_{({\rm Id}, a)}$ corresponding to $z_0$, we get
\[
\dot a_{[z_1,z_0]} \frac{z_1 - z_0}{|z_1-z_0|} + \dot a_{[z_4,z_0]} \frac{z_4 - z_0}{|z_4-z_0|} =0
\]
and, since $z_1-z_0$ and $z_4-z_0$ are ${\bf R}$-independent, we conclude that $\dot a_{[z_1,z_0]} =\dot a_{[z_4,z_0]} =0$. Proceeding similarly for the other components of ${\rm D} F_{({\rm Id}, a)}$, we prove that $\dot a =0$. This completes the proof of the fact that the network is flexible.

It remains to check that the network is closable. This amounts to check that the only solution to ${\rm D} {\bf L}_{\rm Id} (\dot \Phi ) = \lambda \, {\bf T}$ is $\dot \Phi=0$ and $\lambda =0$. Now observe that the equations in this system are of the form
\[
\frac{\langle z-z',  \dot \Phi_z- \dot \Phi_z'\rangle_{\bf C}}{|z'-z|} = \lambda \, |z'-z| \, \ln |a_{[z',z]}|, 
\]  
or equivalently
\[
\langle z-z',  \dot \Phi_z- \dot \Phi_z'\rangle_{\bf C}  = \lambda \, |z'-z|^2 \, \ln |a_{[z',z]}|.
\]  
This is this last system we will consider. 

We write 
\[
\dot \Phi_{z_4} = x_1+ i x_2, \quad \dot \Phi_{z_1} = x_3+ i x_4, \quad  \Phi_{z_2} = x_5+ i x_6, \quad \Phi_{z_3} = x_7+ i x_8,
\]
and 
\[
\dot \Phi_{\hat z_0} = x_9+ i x_{10}, \qquad \dot \Phi_{z_0} = x_{11}+ i x_{12}
\]
and we identify the image of ${\rm D} {\bf L}_{\rm Id}$ with ${\bf R}^9$ starting by labeling the edges in the following order $[\hat z_0, z_0], [\hat z_0, z_4], [z_0, z_1], [z_0, z_2], [\hat z_0, z_3]$ and next $[z_1, z_2], [z_1, z_3],[z_3, z_4],[z_2, z_4]$ to give a vector in ${\bf R}^9$.

{ 
We recall that $\cos \theta$ and $\sin \theta$ have been defined in (\ref{eq:cs}).   We need to check that the vector of ${\bf R}^9$ whose coordinates are given by
\[
\left( \nu^2 \, \ln (2 \cos \theta) , 0, 0, 0, 0,  \ln ( \sin \theta) , \mu^2 \, \ln (\cos \theta),    \ln ( \sin \theta) , \mu^2 \, \ln (\cos \theta) \right) 
\]
is not in the image of 
\[
\left(
\begin{array}{cccccccccccccccccccc}
0 & 0 & 0	& 0  & 0 & 0 & 0 & 0 & - 2\nu & 0 & 2\nu & 0 \\
\nu -\mu & - 1 & 0 & 0 & 0 & 0 & 0 & 0 &  \mu-\nu  &1    & 0 & 0\\
0 & 0 & \mu -\nu & 1 & 0 & 0 & 0 & 0 &  0&   0      & \nu-\mu & - 1\\
0 & 0 & 0 	& 0 & \mu-\nu & -1 & 0 &0 &0 & 0 & \nu-\mu  &1 \\
0 & 0 & 0 	&0& 0 & 0 & \nu-\mu & 1 & \mu -\nu & -1 & 0  &0\\
0 & 0 & 0 & 2 & 0 & -2 & 0 & 0 & 0 & 0 & 0 & 0 \\
0 & 0 & 2\mu & 0 & 0 & 0 & -2\mu & 0 & 0 & 0 & 0 & 0 \\
0  & -2 & 0 & 0 & 0 & 0 & 0 & 2 & 0 & 0 &  0 & 0\\
- 2\mu & 0 & 0 & 0 & 2\mu & 0  & 0 &0 &0& 0&0 & 0
\end{array}
\right)
\]
This can be checked using Mathematica.}
\end{proof}

\medskip

\noindent {\bf Example 10.3 :} \label{ex:10.4}  Finally, we give an explicit example  of flexible closable network whose symmetry group reduces to the identity.  The  set of vertices of the network $\mathscr N_{C}$ is given by
\[
\mathscr V_{C} : =  \{ a+ib , 1+i  , -1+i ,  -1-i , 1 - i  \} ,
\]
for $ 0 < a < b <1$, while its set of edges is defined to be
\[
\begin{array}{rllll}
\mathscr E_{C} & := & \Big\{ [a+ib,1+i ],  [a+ib,-1+i ] ,   [a+ib,-1-i  ], [ a+ib,1-i] ,\\[3mm]
&   & \hspace{20mm} [1+i,-1+i ],    [-1+i,-1-i] ,  [-1-i,1-i],   [1-i,1+i] \Big\} .
\end{array}
\]
\begin{center}
\includegraphics[width=6cm]{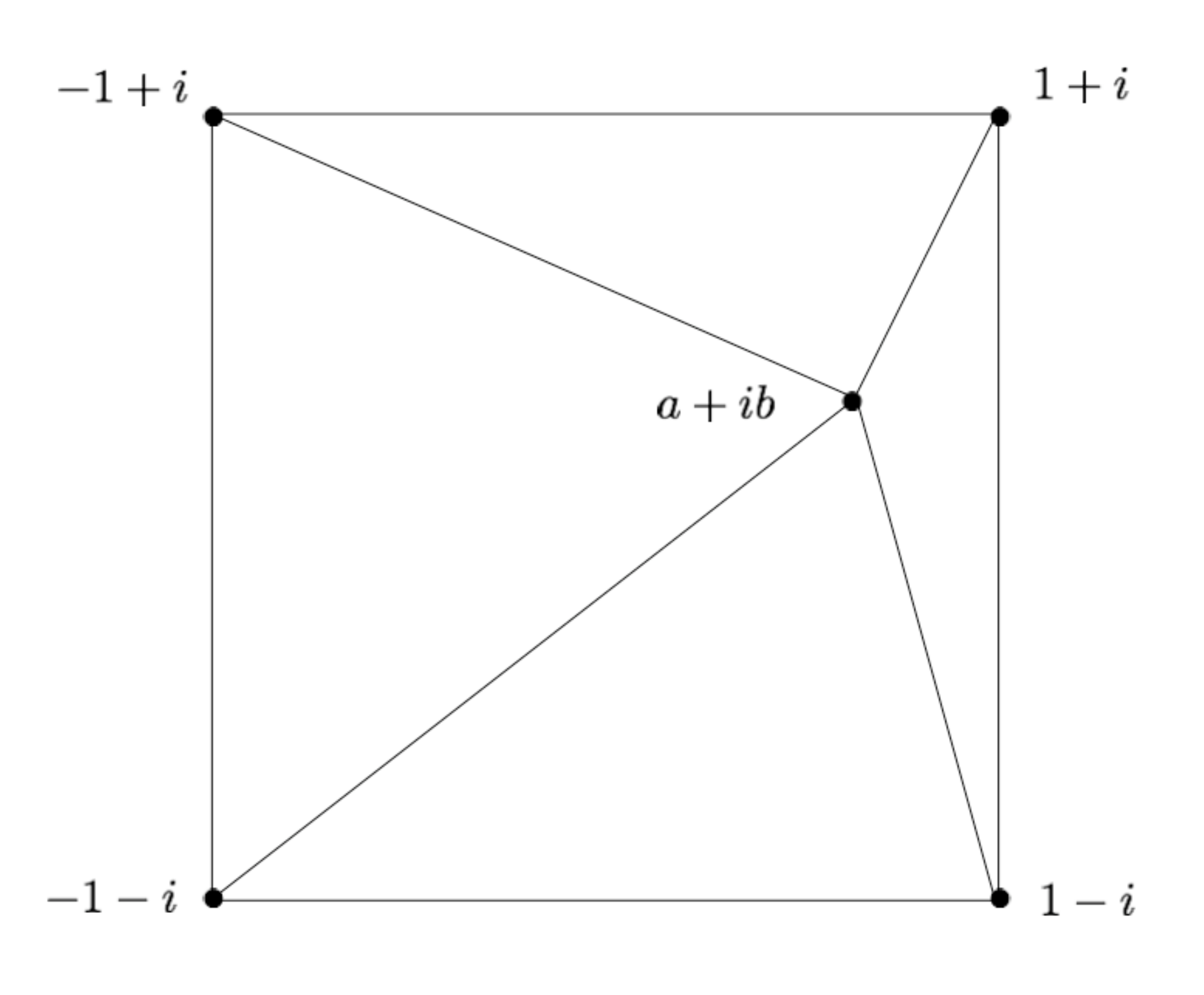}\\
Fig 16 : The nonsymmetric network $\mathscr N_V$.
\end{center}

Observe that this network has symmetry group which reduces to the identity.

We define the weight function $a : \mathscr E_C \to {\bf R}-\{0\}$ by
\[
\begin{array}{rllll}
a_{[a+ib,1+i ]} & = & \displaystyle -{\sqrt{ \frac{1}{(1-a)^2} + \frac{1}{(1-b)^2}}}\\[3mm]
a_{[a+ib,-1+i]}& = & \displaystyle  -{\sqrt{ \frac{1}{(1+a)^2}+ \frac{1}{(1-b)^2}}}\\[3mm]
a_{[a+ib,-1-i]}& = & \displaystyle -{\sqrt{ \frac{1}{(1+a)^2}+ \frac{1}{(1+b)^2}}}\\[3mm]
a_{[a+ib,1-i]}& = & \displaystyle - {\sqrt{ \frac{1}{(1-a)^2}+ \frac{1}{(1+b)^2}}}
\end{array}
\]
and
\[
a_{[1+i,-1+i]}=\frac{1}{1-b},\quad  a_{[1+i,1-i]}=\frac{1}{1-a}, \quad 
a_{[-1+i,-1-i]}=\frac{1}{1+a},\quad a_{[-1-i,1-i]}=\frac{1}{1+b}.
\]
With this choice, it can be checked that the network $(\mathscr N_C, a)$ is balanced. We also have the~:
\begin{lemma}
The balanced network $(\mathscr N_C, a)$ is flexible in the sense of Definition~\ref{de:net-3.2} and closable in the sense of Definition \ref{de:clos}.
\end{lemma}

\begin{proof}
Since $m=8$ and $n=5$, we have $m=2n-2$ and we only need to check ${\rm D}_a {\bf F}_{({\rm Id},a)}$ has rank equal to 7. The proof of this fact is identical to the corresponding proof in Lemma~\ref{le:clso}. 

Therefore, it remains to check that the network is closable. This amounts to check that the only solution to ${\rm D} {\bf L}_{\rm Id} (\dot \Phi ) = \lambda \, {\bf T}$ is $\dot \Phi=0$ and $\lambda =0$. As in the previous proof, we need to show that the system
\[
\langle z-z',  \dot \Phi_z- \dot \Phi_z'\rangle_{\bf C}  = \lambda \, |z'-z|^2 \, \ln |a_{[z',z]}|,
\]  
has no solution except $\dot \Phi =0$ and $\lambda =0$.

We write 
\[
\dot \Phi_{1+i} = x_1+ i x_2, \quad \dot \Phi_{-1+i} = x_3+ i x_4, \quad  \Phi_{-1-i} = x_5+ i x_6, \quad \Phi_{1-i} = x_7+ i x_8,
\]
and 
\[
\dot \Phi_{a+ib} = x_9+ i x_{10}.
\]
and we identify the image of ${\rm D} {\bf L}_{\rm Id}$ with ${\bf R}^8$ starting by labeling the edges in the following order $[a+ib, 1+i], [a+ib, 1-i], [a+ib, -1-i], [a+ib, -1+i]$ and next $[1-i, 1+i], [1+i, -1+i],[-1+i, -1-i],[-1-i, 1-i]$ to give a vector in ${\bf R}^8$.

{ 
Therefore, we need to check that the vector of ${\bf R}^8$ whose coordinates are given by
\[
\begin{array}{rlllll}
& \displaystyle  \Big( \left( (1- a)^2+ (1- b)^2\right)  \ln \left(  \sqrt{ \frac{1}{(1-a)^2} + \frac{1}{(1-b)^2}} \right),\\
&  \displaystyle \qquad  \left( (1+a)^2+ (1-b)^2 \right) \ln \left(  \sqrt{ \frac{1}{(1+a)^2} + \frac{1}{(1-b)^2}} \right),\\
&  \displaystyle \qquad  \left( (1+a)^2+ (1+b)^2  \right) \ln \left( \sqrt{ \frac{1}{(1+a)^2} + \frac{1}{(1+b)^2}} \right),\\
&  \displaystyle \qquad  \left( (1-a)^2+ (1+b)^2  \right) \ln \left(  \sqrt{ \frac{1}{(1-a)^2} + \frac{1}{(1+b)^2}} \right),\\
&   \displaystyle  \qquad  -4 \ln(1-b), - 4 \ln(1-a) , - 4 \ln (1+a), - 4 \ln(1+b)\Big) 
\end{array}
\]
is not in the image of
\[
\left(
\begin{array}{ccccccccccccccc}
1-a & 1- b & 0 & 0 & 0 & 0 &0 & 0 & a-1 & b-1 \\
 0  &  0 & -1-a & 1-b  & 0 & 0 & 0 & 0 & 1+a & b-1\\
0 & 0 & 0&0  & -1-a & -1-b & 0 & 0 & 1+a & 1+b\\
0 & 0 & 0 & 0 & 0 &0 & 1-a &-1-b &a-1 & 1+ b\\
2&0&-2&0&0& 0&0&0&0&0\\
0&2&0&0&0& 0&0&-2&0&0\\
0&0&0&2&0 &-2&0&0&0&0\\
0&0&0&0&-2&0&2&0&0&0
\end{array}
\right)
\]
This can be checked using Mathematica.}
\end{proof}

The question is now the following~: whether or not can we use this network to complete the proof of Theorem~\ref{th:main}~? When $a+ib=0$ this network corresponds to the network we have already studied in Example 5.2, for which we have found subnetworks which are flexible. By perturbation, one can use subnetworks similar to the one described in this example at least when $a+ib$ is close enough to $0$. This completes the proof of Theorem~\ref{th:main}.

\section{More general nonlinearities and higher dimensional problems}

As already mentioned in the introduction, our result does not only hold for the equation (\ref{eq:nls}) but applies to a broader class of equations. For example, it applies to the equations of the form
\[
\Delta u - u+ f(u) =0,
\]
defined in ${\bf R}^2$, where the nonlinearity $f$ is odd, at least $\mathcal C^{1, \mu} $ for some $ \mu \in (0,1)$ and satisfy the following conditions~:

\begin{enumerate}
\item[(H.1)] $ f(0)=\partial_u f (0)=0$.\\
\item[(H.2)] The equation
\begin{equation}
\label{ground} \Delta \, u-  u +  f (u) = 0 ,
\end{equation}
has a unique {\em positive} (radially symmetric) solution $u_0$ which tends to $0$ exponentially fast at infinity. \\

\item[(H.3)] The solution $u_0$ is {\em nondegenerate}, in the sense that
\begin{equation}
\label{Ker} 
	\mbox{Ker}  \, \left(\Delta - 1 + \partial_u f  (u_0) \right) \cap L^\infty ({\bf R}^2) = \mbox{Span} \left\{ \partial_{x_1}	u_0,  \partial_{x_2} u_0 \right\} .
\end{equation}
\end{enumerate}

Typical example of nonlinearities $f$ satisfying all the above assumptions are given by the function
\[
	f(u)=  (|u|^{p-1}\, u - c \,  |u|^{q-1} \, u)  ,
\]
where $c \geq 0$ and  $1 < q < p$. In this case, the existence of $w_i$ is standard and follows from well known arguments in the calculus of variation while the uniqueness  follows from results of Kwong \cite{K} and Kwong and Zhang \cite{kz}. Concerning the nondegeneracy condition (which essentially follows from the uniqueness of the solutions), we refer to Appendix C of \cite{nitakagi}.

For example, when $c =0$, the nonlinearity is just given by 
\[
f(u) =  |u|^{p-1} \, u  .
\]

In the general case, the function $\Upsilon$ given in (\ref{martin}) for the nonlinearity $u \to u^3$, has to be replaced by
\[
\Upsilon (s) := - \iint_{{\bf C}} u_0 (z -s {\bf e}) \, \mbox{\rm div} \, \left( f(u_0) (z) \, {\bf e} \right) \, dx\, dy.
\]

Let us emphasize that our construction also generalizes to nonlinearities which are not necessarily even (see \cite{MPW} for a precise description of the nonlinearities which are allowed). However in this case, we need to define 4 different type of interaction functions and then the statement of Proposition~\ref{pr:thisisit} become even really involved. This is the reason why, we have chosen not to follow this route even though the constructions are still possible. 

Also, we should emphasize that constructions in higher dimension are also possible. Obviously, if all the network under consideration is included in a plane, one can work equivariantly and extend to construction (we again refer to \cite{MPW} for a description of the nonlinearities which are allowed for such constructions.  Also, the notions of balanced, flexible and closable networks can be extended to higher dimensions in a rather natural way. However, the construction of examples becomes quite difficult and we believed that this was not worth the effort.

\end{document}